\pdfoutput=1
\documentclass[11pt,a4paper]{article}

\usepackage[margin=1in]{geometry}
\usepackage{amsmath,amssymb,amsfonts,mathtools,amsthm,mathrsfs,bbm,booktabs,enumitem,tikz,graphicx,xcolor,ifpdf,nicefrac}
\usepackage[numbers,sort&compress]{natbib}
\usepackage[hidelinks]{hyperref}

\ifpdf
  \DeclareGraphicsExtensions{.pdf,.png,.jpg}
\else
  \DeclareGraphicsExtensions{.eps}
\fi
\graphicspath{{./}}

\hypersetup{
  pdftitle={Mass Lumping and Numerical Quadrature for Approximation of Fractional Elliptic Differential Equations Driven by Gaussian White Noise},
  pdfauthor={Kelvin J. R. Almeida-Sousa, David Bolin, and Alexandre B. Simas}
}

\providecommand{\email}[1]{\href{mailto:#1}{\texttt{#1}}}
\providecommand{\funding}[1]{#1}

\title{Mass Lumping and Numerical Quadrature for Approximation of Fractional Elliptic Differential Equations Driven by Gaussian White Noise\thanks{\funding{This publication is based upon work supported by King Abdullah University of Science and Technology (KAUST) under Award No.~ORFS-CRG12-2024-6399.}}}
\author{Kelvin J. R. Almeida-Sousa\thanks{Statistics Program, King Abdullah University of Science and Technology, Saudi Arabia (\email{kelvinjhonson.silva@kaust.edu.sa}, \email{david.bolin@kaust.edu.sa}, \email{alexandre.simas@kaust.edu.sa}).}
\and David Bolin\footnotemark[2]
\and Alexandre B. Simas\footnotemark[2]}
\date{}

\newenvironment{keywords}{\par\smallskip\noindent\textbf{Keywords: }}{\par}
\newenvironment{MSCcodes}{\par\smallskip\noindent\textbf{MSC (2020): }}{\par\medskip}

\theoremstyle{plain}
\newtheorem{theorem}{Theorem}[section]
\newtheorem{lemma}[theorem]{Lemma}
\newtheorem{proposition}[theorem]{Proposition}
\newtheorem{corollary}[theorem]{Corollary}
\theoremstyle{definition}
\newtheorem{definition}[theorem]{Definition}
\newtheorem{assumption}[theorem]{Assumption}
\newtheorem{example}[theorem]{Example}
\theoremstyle{remark}
\newtheorem{remark}[theorem]{Remark}

\numberwithin{equation}{section}

\newcommand{\bbN}{\mathbb{N}}

\newcommand{\bbR}{\mathbb{R}}

\newcommand{\cA}{\mathcal{A}}
\newcommand{\cB}{\mathcal{B}}

\newcommand{\cD}{\mathcal{D}}
\newcommand{\cE}{\mathcal{E}}

\newcommand{\cH}{\mathcal{H}}
\newcommand{\cI}{\mathcal{I}}

\newcommand{\cR}{\mathcal{R}}

\newcommand{\cT}{\mathcal{T}}

\newcommand{\cV}{\mathcal{V}}
\newcommand{\cW}{\mathcal{W}}

\newcommand{\norm}[2]{\|#1\|_{#2}}

\newcommand{\scalar}[2]{\left(#1\right)_{#2}}
\newcommand{\LD}{}
\newcommand{\LDD}{}
\newcommand{\HS}{L_2}
\newcommand{\LL}{L}

\newcommand{\white}{\cW}

\newcommand{\rd}{\mathrm{d}}
\newcommand{\from}{\colon}

\DeclarePairedDelimiter\floor{\lfloor}{\rfloor}

\newcommand{\mv}[1]{{\boldsymbol{\mathrm{#1}}}}

\newcommand{\proper}{\mathsf}
\newcommand{\pN}{\proper{N}}

\newcommand{\inner}[2]{\left\langle#1,#2\right\rangle}
\DeclareMathOperator*{\esssup}{ess\,sup}

\newcommand{\<}{\langle}
\renewcommand{\>}{\rangle}

\allowdisplaybreaks

\begin{document}

\maketitle

\begin{abstract}
Fractional elliptic stochastic partial differential equations (SPDEs) are widely used
in statistics and machine learning for computationally efficient and flexible modeling
of Gaussian random fields. The computational efficiency of the SPDE approach relies on
finite element approximations combined with numerical quadrature and mass lumping,
which enable sparse matrix methods during inference. Although many works have studied
finite element approximations of fractional SPDEs, the effect of the mass lumping and
quadrature approximations used in practice has not been fully analyzed.
To fill this gap, we derive convergence rates for numerical approximations of
fractional SPDEs based on finite element discretizations combined with numerical
quadrature and mass lumping. Specifically, we obtain explicit convergence rates for the
mean-squared error of the covariance function in a general framework that covers the
main settings where mass lumping is used in the SPDE approach. We also analyze
non-stationary variance-control factors of the form $L^\beta(\tau u)=\mathcal{W}$,
where $\tau$ is spatially varying, and derive covariance error estimates showing how
the regularity of $\tau$ affects the convergence rate.
As specific examples, we provide results for random fields on bounded Euclidean
domains, Riemannian manifolds, and metric graphs. Numerical experiments are presented
that confirm the theoretical results.
\end{abstract}



\section{Introduction and Notation}

\subsection{Introduction}
The stochastic partial differential equation (SPDE) approach \cite{lindgren11} is widely used to represent Gaussian random fields in spatial statistics \cite{lindgren2022spde}. It considers Gaussian random fields on a domain $\cD\subset\bbR^d$ that can be represented as solutions to SPDEs of the form
\begin{equation}\label{spde}
    L^\beta u = \mathcal{W}
    \qquad \text{in } \cD,
\end{equation}
where $\mathcal{W}$ is Gaussian white noise on $L^2(\cD)$, and $L^\beta$ is a fractional power of a second-order elliptic differential operator $L$. The choice of $L$ determines the covariance structure of $u$. A motivating example is the elliptic operator $L=-\operatorname{div}_{\cD}\bigl(\boldsymbol H\nabla_{\cD}\bigr)+\kappa^2$,
where $\boldsymbol H$ is a suitable matrix-valued function and $\kappa$ is a real-valued function on $\cD$. Here, $\nabla_{\cD}$ and $\operatorname{div}_{\cD}$ denote the gradient and divergence operators on the underlying spatial domain. This construction is closely related to Gaussian random fields with Matérn covariance functions \cite{matern60}, arguably the most widely used class of covariance models in spatial statistics and machine learning \cite{Porcu2024}. In particular, when $\cD=\bbR^d$, $\boldsymbol H$ is constant, and $\kappa>0$ and $\beta>d/4$ are constants, the solution to \eqref{spde} has a possibly anisotropic Matérn covariance function \cite{whittle63}. The parameters $\kappa$ and $\beta$ control the correlation range and smoothness of the field, respectively.
	
	The main idea of the SPDE approach, as introduced in \cite{lindgren11}, is that one can obtain a computationally efficient Gaussian Markov random field (GMRF) approximation of $u$ by considering a finite element method (FEM) approximation
	$$
	u_h(s) = \sum_{i=1}^n u_i\varphi_i(s)
	$$
	 of \eqref{spde} in terms of continuous piecewise linear ``hat functions'' $\{\varphi_i\}_{i=1}^n$ defined by a triangulation of the domain $\cD$.  The distribution of the weights $\mv{u} = (u_1,\ldots, u_n)^\top$ is calculated through the Galerkin method, which for the case with $L = \kappa^2 - \Delta$ and $\beta=1$, results in $\mv{u} \sim \pN(\mv{0},\mv{Q}^{-1})$, where the precision matrix is given by 
	\begin{equation}\label{eq:Q}
	\mv{Q} = (\kappa^2\mv{C} + \mv{G})^\top\mv{C}^{-1}(\kappa^2\mv{C} + \mv{G}).
	\end{equation}
	Here $\mv{C}$ is the mass matrix with elements $C_{ij} = \scalar{\varphi_i,\varphi_j}{L^2(\cD)}$, $\mv{G}$ the stiffness matrix with elements ${G_{ij} = \scalar{\nabla\varphi_i, \nabla\varphi_j}{L^2(\cD)}}$, and $(\cdot,\cdot)_{L^2(\mathcal{D})}$ the $L^2(\mathcal{D})$ inner product. 
	As shown in \cite{lindgren11}, similar GMRF representations can be obtained for all $2\beta \in \{1,2,3,\ldots\}$, and one can allow the parameter $\kappa$ to be spatially varying to obtain a non-stationary Gaussian field. The only practical difference in the non-stationary case is that $\kappa^2\mv{C}$ is replaced by a matrix $\mv{C}_{\kappa}$ with elements 
	$\scalar{\kappa^2\varphi_i,\varphi_j}{L^2(\cD)}$. 
	
	Since the introduction of the method, it has been used in a large number of applications (see \citep{bakka2019,lindgren2022spde} for recent reviews) and it has been extended and analyzed theoretically in a number of papers. In particular, the half-integer restriction on $\beta$ has been removed by combining the FEM approximation with a rational approximation of the fractional power, and convergence rates of the resulting approximations have been derived \citep{BK2020rational,BKK2018,BKK2020,cox2020}. Further, extensions to non-Gaussian models \citep{bolin14,Wallin15,suuronen2022cauchy,bw20}, models on manifolds and surfaces \citep{herrmann2020multilevel,borovitskiy2020matern}, non-stationary models with different features \citep{Hildeman2020,fuglstad2015exploring,bakka2019}, spatio-temporal models \citep{allard2021,lindgren2024diffusion,cameletti_spatio-temporal_2013} have been considered and other computational methods have been proposed \citep{Harbrecht2021,Alonso2021}. Statistical properties of the SPDE-based models and the FEM approximations have recently also been considered more carefully \citep{bolin2023equivalence,Alonso2021fem}. 
	
	The key property that makes GMRFs computationally efficient is the sparsity of the precision matrix. However, as introduced above, \eqref{eq:Q} is not a sparse matrix. Therefore, a key step in the SPDE approach is a so-called mass lumping, where the matrix $\mv{C}^{-1}$ in \eqref{eq:Q} is replaced by a diagonal matrix $\widetilde{\mv{C}}^{-1}$, where $\widetilde{\mv{C}}_{ii} = \scalar{\varphi_i,1}{L^2(\cD)}$. Since $\mv{C}$ and $\mv{G}$ are sparse, this results in a sparse precision matrix. This additional approximation is always used in the SPDE approach, but the effects of the approximation have not yet been fully investigated. In fact, as far as we know, all papers which have presented theoretical results on the SPDE approach simply ignore this additional approximation and state the results without mass lumping (see for example \citep{Alonso2021fem,BK2020rational,cox2020}). This is not ideal since the theoretical results are then not directly applicable to the method actually used in practical applications. The only result available in the literature is the discussion on the mass lumping in \cite{lindgren11}, where it is stated (without a detailed proof) that the additional approximation does not reduce the convergence rate of the approximation in the non-fractional case $\beta=1$ with a constant $\kappa$, and \cite{bolin_comparison_2013} who studied the effects of the mass lumping numerically. 
	
	Our main contribution is to fill this gap in the theory of the SPDE approach by
providing a convergence analysis for mass-lumped and quadrature-based discretizations
of covariance functions induced by fractional SPDEs. The analysis is carried out in a
general operator-theoretic framework, but it is motivated by the matrix approximations
used in practical implementations of the SPDE approach.

More specifically, we first analyze the effect of replacing exact mass matrices by
mass-lumped ones. For instance, in the case $\beta=1$, this includes both the basic
mass-lumped approximation, where the matrix $\mv{C}^{-1}$ in \eqref{eq:Q} is replaced
by $\widetilde{\mv{C}}^{-1}$, and the fully mass-lumped approximation, where the mass
matrices appearing in $\kappa^2\mv{C}+\mv{G}$ are also replaced by their lumped
counterparts. Thus, in the constant-coefficient case, the latter approximation leads to
the precision matrix
$
(\kappa^2\widetilde{\mv{C}}+\mv{G})^\top
\widetilde{\mv{C}}^{-1}
(\kappa^2\widetilde{\mv{C}}+\mv{G}).
$
Our results show that these replacements can be performed without reducing the
covariance convergence rate, under suitable assumptions on the quadrature rules and the
finite element spaces.
Second, we consider quadrature-based approximations of spatially varying coefficients.
In the non-stationary case, matrices such as $\mv{C}_\kappa$  are typically not assembled exactly.
Instead, they are approximated using quadrature rules, often leading to diagonal or
mass-lumped matrix representations involving the nodal values of $\kappa$. We provide
explicit convergence rates for these approximations in the fractional setting, thereby
covering the matrix constructions commonly used in applications.

In the abstract formulation developed below, these practical matrix replacements are
encoded through admissible discrete inner products and admissible discretizations of
the underlying operator. This viewpoint allows us to treat exact Galerkin
discretizations, mass-lumped discretizations, and quadrature-based non-stationary
discretizations within the same framework. It also clarifies why the inverse mass
matrix appearing in expressions such as \eqref{eq:Q} is tied to the choice of discrete
inner product. The precise operator-level formulation of this point is introduced later,
after the finite-dimensional spaces and discrete inner products have been defined.

Finally, we also study a commonly used extension of \eqref{spde} in which the model includes a
spatially varying variance-control factor. More precisely, we consider models of the
form
$
    L^\beta(\tau u)=\cW,
$
where $\tau:\cD\to(0,\infty)$ is used to model non-stationary marginal variance
separately from the dependence structure induced by the elliptic operator $L$. When
$\tau$ is constant, it simply rescales the variance. When $\tau$ varies over $\cD$,
however, multiplication by $\tau$ has to be discretized, usually through nodal or
mass-lumped approximations. This leads to another natural quadrature problem. We derive
covariance error estimates for this case and show explicitly how the regularity of
$\tau$ affects the convergence rate.

The outline of the manuscript is as follows. The model considered in this work is
presented in Section~\ref{sec:elliptic}. The approximation methods and main results
are presented in Section~\ref{sec:fem}. Section~\ref{sec:examples} describes the most
important applications of the general theory, and the proofs are given in
Section~\ref{sec_proof_results}. Numerical experiments confirming our theoretical
findings are presented in Section~\ref{sec:experiments}, while
Section~\ref{sec:discussion} concludes with a brief discussion.

\subsection{Notation}

%
Throughout this work, $(\Omega,\mathcal{F},\mathbb{P})$ denotes a complete probability
space. We let $(\cD,\mathcal{B},\mu)$ be a measure space, where $\cD$ denotes
the spatial domain, $\mathcal{B}$ is a $\sigma$-algebra on $\cD$, and $\mu$ is a
reference measure. In the examples considered here, $\cD$ is typically a
separable metric space equipped with a natural choice of $\mathcal{B}$ and $\mu$.
The precise assumptions on $(\cD,\mathcal{B},\mu)$ will be given in
Section~\ref{sec:elliptic}.
For a Hilbert space $E$, we denote its inner product by $\scalar{\cdot,\cdot}{E}$ and
its norm by $\norm{\cdot}{E}$. We write $L^{2}(\cD)$ as shorthand for
$L^{2}(\cD,\mathcal{B},\mu)$, the Hilbert space of $\mu$-square-integrable
functions on $\cD$, and use the notation
$\norm{\cdot}{}:=\norm{\cdot}{L^{2}(\cD)}$ for its canonical norm.

For Hilbert spaces $E$ and $F$, we denote by $L(E,F)$ and $L_{2}(E,F)$ the spaces of
bounded linear operators and Hilbert--Schmidt operators from $E$ to $F$, respectively.
When $E=F$, we write $L(E)$ and $L_{2}(E)$. The corresponding operator and
Hilbert--Schmidt norms are denoted by $\norm{\cdot}{L(E,F)}$ and
$\norm{\cdot}{L_{2}(E,F)}$, respectively.
If $E,F\subseteq V$ for some vector space $V$, and if the identity map on $V$, restricted
to $E$, defines an element of $L(E,F)$, we say that $E$ is continuously embedded in $F$.
If, in addition, $F$ is continuously embedded in $E$, then the two norms are equivalent;
in this case, we write $(E,\scalar{\cdot,\cdot}{E})\cong
(F,\scalar{\cdot,\cdot}{F})$. For a Banach space $E$, we denote by $E'$ its dual space,
namely the space of all continuous linear functionals on $E$.
Finally, given a non-empty set $S$ and functions $a,b:S\to\mathbb{R}$, we write
$a\lesssim b$ if there exists a constant $C>0$ such that $a(s)\leq Cb(s)$ for every
$s\in S$. The constant $C$ is understood to be independent of the relevant parameter
being varied; in particular, when $S$ is a set of discretization parameters, $C$ is
independent of the discretization parameter.


	\section{Model description}\label{sec:elliptic}
	We are interested in Gaussian random fields specified through fractional-order SPDEs of the form \eqref{spde}. To be able to develop a unified numerical analysis that applies across the wide range of possible spatial domains, we only make the following weak assumption regarding the domain $\cD$ on which the SPDE is posed. 
	
		\begin{assumption}[Domain]\label{ass_domain}
        A \textit{spatial domain} refers to a finite measure space $(\mathcal{D},\mathcal{B},\mu)$ equipped with a countably generated $\sigma$-algebra $\mathcal{B}$ on $\mathcal{D}$. 
	    \end{assumption}

	    The requirement of a countably generated $\sigma$-algebra is important because the corresponding $L^{2}(\cD)$ space is then separable (see, e.g., \cite[Proposition 3.4.5]{Cohn_2013}).

\begin{example}\label{example_domain} Spatial domains that satisfy Assumption \ref{ass_domain} and are typically considered in real-world applications are:
		\begin{enumerate}
			\item[(D1)] Bounded, connected, and convex open subsets of $\bbR^{d}$ (for $d=1,2,3$) with polytope boundaries. In this case, we take $\mu$ to be the Lebesgue measure over the Borel sets of $\cD$;
			\item[(D2)] Compact, connected, orientable hypersurfaces of $\bbR^{3}$ without boundary. In this case, we take $\mu$ to be the $2$-dimensional Hausdorff measure over the Borel sets of $\cD$;
			\item[(D3)] Compact and connected metric graphs $\cD=(\cE, \cV)$, with 
			$$\mu(A):=\sum_{e\in \mathcal{E}} \mu_{e}(A\cap e),\quad (A\,\mbox{a Borel subset of}\,\mathcal{D})$$ where, if $\ell_{e}$ denotes the length of $e$, $\mu_{e}$ is the pushforward of the Lebesgue measure $\mu_{[0,\ell_{e}]}$ on the interval $[0,\ell_{e}]$ by the (arc-length) parametrization $p_{e}:[0, \ell_{e}]\to e$ of the edge $e$. That is, $\mu_{e}(B)=\mu_{[0,\ell_{e}]}\left(p_{e}^{-1}(B)\right)$.
		\end{enumerate}
\end{example}
		
In Section \ref{sec:examples}, we revisit these examples and illustrate how our theoretical results apply to each class of spatial domains described in Example~\ref{example_domain}.
Next, we introduce the assumptions on the differential operator $L$ in \eqref{spde}.
	
\begin{assumption}[Operator]\label{ass_operator}
The operator $L:\mathscr{D}(L)\subset L^2(\cD)\to L^2(\cD)$ is densely defined,
self-adjoint, bijective, with a compact  inverse
${L^{-1}:L^2(\cD)\to L^2(\cD)}$. 
It is further strictly positive, which means that
$\scalar{Lu,u}{L^2(\cD)}>0$ for every $u\in\mathscr{D}(L)\setminus\{0\}$.
Consequently, there exists an orthonormal basis
$(e_j)_{j\in\mathbb{N}}\subset\mathscr{D}(L)$ of $L^2(\cD)$ consisting of
eigenfunctions of $L$. We denote the corresponding eigenvalues by
$(\lambda_j)_{j\in\mathbb{N}}$ and enumerate them in non-decreasing order, counting
multiplicities. Thus, $L e_j=\lambda_j e_j$, with
$0<\lambda_1\leq\lambda_2\leq\cdots$ and $\lambda_j\to\infty$ as $j\to\infty$.
We further assume that the eigenvalues satisfy the following Weyl-type growth condition:
there exist $\alpha>0$ and constants $c_{\alpha},C_{\alpha}>0$ such that
\begin{equation}\label{weyl_law}
  c_{\alpha}j^{\alpha}\le \lambda_j\le C_{\alpha}j^{\alpha},
  \qquad j\in\mathbb{N}.
\end{equation}
\end{assumption}

	Due to Assumption \ref{ass_operator} we may introduce, for $\beta\geq 0$, the $\beta$-fractional power of $L$ as the operator $L^{\beta}: \mathscr{D}(L^{\beta/2})\to L^{2}(\cD)$, where
	$$\mathscr{D}(L^{\beta/2}) := \dot{H}_L^\beta(\mathcal{D}) := \left\{\psi\in L^2(\mathcal{D}):  \sum_{j\in\mathbb{N}}  \lambda_j^{\beta} (\psi, e_j)_{L^2(\mathcal{D})}^2<\infty\right\}$$
	and $L^{\beta}\psi:=\sum_{j\in \mathbb{N}}\lambda_{j}^{\beta/2}(\psi, e_j)e_{j}$ for $\psi \in  \dot{H}_L^\beta(\mathcal{D})$. Note that $\dot{H}_L^\beta(\mathcal{D})$ endowed with the inner product
	$$
	(\psi,\phi)_{\beta} := \scalar{\psi,\phi}{\dot{H}_L^\beta(\mathcal{D})} := \scalar{L^{\beta/2}\psi, L^{\beta/2}\phi}{L^2(\mathcal{D})} = \sum_{j\in\mathbb{N}}{\lambda_j^\beta} (\psi, e_j)(\phi, e_j)
	$$
	is a Hilbert space which induced norm is denoted by $\|\psi\|_{\dot{H}_L^\beta(\mathcal{D})}$ or simply $\|\psi\|_{\beta}$. Finally, in case $\beta<0$, we define $\dot{H}_L^\beta(\mathcal{D}):=(\dot{H}_L^{-\beta}(\mathcal{D}))^{\prime}$ and consequently, we are allowed to take the $\beta$-fractional power of $L$ for any $\beta\in \bbR$. In this sense, note that if $\beta<0$ the operator $L^{\beta}:=(L^{-\beta})^{-1}:L^{2}(\cD)\to \dot{H}_L^\beta(\mathcal{D})$ is well-defined.

The bilinear form $a_L:\dot{H}^{1}_{L}(\cD)\times\dot{H}^{1}_{L}(\cD)\to\bbR$,
defined by $a_L(u,v):=(u,v)_1$, plays an important role in this work. We refer to it
as the Dirichlet form associated with $L$. If $u_j=(u,\phi_j)$ and
$v_j=(v,\phi_j)$, then $a_L(u,v)=\sum_{j\in\mathbb{N}}\lambda_j u_jv_j$. In particular,
since $\lambda_j\geq\lambda_1>0$ for every $j\in\mathbb{N}$, we obtain
\begin{equation}\label{dirich_form_coerc_rel}
    a_L(u,u)=\sum_{j\in\mathbb{N}}\lambda_j u_j^2
    \geq \lambda_1\sum_{j\in\mathbb{N}}u_j^2
    =\lambda_1\|u\|^2,
    \qquad u\in\dot{H}^{1}_{L}(\cD).
\end{equation}
Thus, $a_L$ is coercive with respect to the $L^2(\cD)$-norm.
Moreover, as ${\mathscr{D}(L)=\dot{H}^{2}_{L}(\cD)}$, the spectral representation of
$L$ gives the following relation:
\begin{equation}\label{friedric_ext_rel}
	\forall u\in\mathscr{D}(L), 
v\in\dot{H}^{1}_{L}(\cD) \qquad (Lu,v)=a_L(u,v).
\end{equation}
In particular, $(Lu,u)=a_L(u,u)\geq\lambda_1\|u\|^2$ for
every $u\in\mathscr{D}(L)$.

Since $\mathscr{D}(L)$ is dense in $\dot{H}^{1}_{L}(\cD)$, the space
$\dot{H}^{1}_{L}(\cD)$ coincides with the energetic space associated with the symmetric
and $L^2(\cD)$-strongly monotone bilinear form
$\mathfrak{a}_L:=a_L|_{\mathscr{D}(L)\times\mathscr{D}(L)}$; see
\cite[Chapter 5, Example 4]{zeidler2012applied}. In this sense, $L$ is the
self-adjoint operator associated with the closed form $a_L$, equivalently the
Friedrichs extension of $\mathfrak{a}_L$.
 
The final component in \eqref{spde} that remains to be specified is the white noise
term $\cW$. We define $\cW$ as Gaussian white noise on $L^2(\cD)$, that is, as an
isonormal Gaussian process over $L^2(\cD)$; see, e.g., \cite{Nualart_1995}. Thus,
$\cW=\{\cW(f)\}_{f\in L^2(\cD)}$ is a centered Gaussian family satisfying
$\mathbb{E}[\cW(f)\cW(g)]=(f,g)$ for every $f,g\in L^2(\cD)$, and the map
$f\mapsto \cW(f)$ is linear as a map from $L^2(\cD)$ into $L^2(\Omega)$.
Let $(e_j)_{j\in\mathbb{N}}$ be the orthonormal basis of $L^2(\cD)$ from
Assumption~\ref{ass_operator}, and define $\xi_j:=\cW(e_j)$. Then
$(\xi_j)_{j\in\mathbb{N}}$ is a sequence of independent standard Gaussian random
variables. Moreover, for every $f\in L^2(\cD)$,
\begin{equation}\label{eq:loeve}
    \cW(f)=\sum_{j\in\mathbb{N}}\xi_j(e_j,f),
\end{equation}
where the series converges in $L^2(\Omega)$ and almost surely for each fixed $f$.
We refer to \eqref{eq:loeve} as the Karhunen--Loève expansion of $\cW$.
	
The covariance operator induced by \eqref{spde}, whenever the solution is well defined
as an $L^2(\cD)$-valued random variable, is $L^{-2\beta}$. In this context, we refer to
$L^{2\beta}$ as the precision operator. The operator $L^{-2\beta}$ is closely related
to its covariance kernel $\varrho^\beta$, which is the main object of interest in this
work and motivates the discretizations considered in this manuscript.
Indeed, by the Weyl-type condition \eqref{weyl_law},
$\|L^{-\beta}\|_{L_2(L^2(\cD))}<\infty$ if and only if
$2\beta>1/\alpha$. Equivalently, $L^{-2\beta}$ is trace class if and only if
$2\beta>1/\alpha$. In this regime, the solution of \eqref{spde} can be represented as
the centered Gaussian random element
$u=\sum_{j\in\mathbb{N}}\lambda_j^{-\beta}\xi_j e_j$ in $L^2(\cD)$, where
$(\xi_j)_{j\in\mathbb{N}}$ are the Gaussian random variables in \eqref{eq:loeve}. Its covariance operator is $L^{-2\beta}$.

Moreover, $L^{-2\beta}$ is a Hilbert--Schmidt operator and therefore admits a kernel
$\varrho^\beta\in L^2(\cD\times\cD)$, defined $\mu\otimes\mu$-a.e. by
\begin{equation}\label{rel:fracov}
    \varrho^\beta(x,y)
    =\sum_{j=1}^{\infty}\lambda_j^{-2\beta}e_j(x)e_j(y).
\end{equation}
Thus, for every $f\in L^2(\cD)$, $L^{-2\beta}f$ is given by
\begin{equation}\label{L_ker_idtt}
L^{-2\beta}f(x)=\int_{\cD}\varrho^\beta(x,y)f(y)\,\mu(dy), 
\end{equation}
where the equality is 
understood in $L^2(\cD)$. Consequently, we have the equality  
$\|L^{-2\beta}\|_{L_2(L^2(\cD))}=\|\varrho^\beta\|_{L^2(\cD\times\cD)}.$
When pointwise evaluations of $u$ are meaningful, the kernel $\varrho^\beta$ coincides
with the covariance function of the field, in the sense that
$\varrho^\beta(x,y)=\mathbb{E}[(u(x)-\mathbb{E}u(x))(u(y)-\mathbb{E}u(y))]$ for
$\mu\otimes\mu$-a.e. $(x,y)\in\cD\times\cD$.
	
	\section{Discretization Methods for the Covariance Function}\label{sec:fem}

	In this section, we present and analyze several discretization methods for the covariance function $\varrho^\beta$ introduced in the previous section. We begin with the finite element method as a foundation and then introduce alternative schemes designed to address specific computational challenges. For each method, we state our main theoretical results and provide a discussion of their respective advantages and limitations, offering insight into their practical applicability across different settings.



\subsection{Galerkin Discretization for the Covariance}

We begin by showing how a FEM discretization of \eqref{spde} naturally leads
to a discretization $\widehat{\varrho}_h^\beta$ of the covariance function. This approximation,
which we also refer to as the Galerkin-based approximation, serves as our starting point
for alternative discretization strategies.

We first impose an assumption on the finite-dimensional trial spaces and on the
interpolation operators used throughout the paper. The assumption also involves a class
$\cR^k$ of sufficiently regular multiplier functions. In standard Euclidean finite
element settings, $\cR^k$ is typically chosen as $W^{k,\infty}(\cD)$, or as an
appropriate subspace of $W^{k,\infty}(\cD)$ compatible with the scale
$\dot H_L^k(\cD)$.

\begin{assumption}[Trial space and interpolator]\label{assump2}
Let $k\geq 1$. We assume the existence of a family of finite-dimensional subspaces
$\{V_h\}_{h\in(0,1)}\subseteq \dot{H}^{1}_{L}(\cD)$ and a family of linear operators
$\{\cI_h:\mathscr{D}(\cI_h)\subseteq \dot{H}^{1}_{L}(\cD)\to V_h\}_{h\in(0,1)}$.
Moreover, we assume that there exists a nontrivial normed subspace
$\cR^k\subseteq \dot{H}^{k}_{L}(\cD)\cap L^\infty(\cD)$, containing the constant
functions, such that, for every $h\in(0,1)$, the following properties hold:
\begin{enumerate}[label=(\roman*)]
    \item $\tau\psi\in\mathscr{D}(\cI_h)$ for every $\tau\in\cR^k$ and
    $\psi\in V_h$;

    \item $\|(\cI_h-I)(\tau\psi)\|\lesssim h^k\|\tau\|_{\cR^k}\|\psi\|_1$ for every
    $\tau\in\cR^k$ and $\psi\in V_h$;

    \item $\|\cI_h(\tau\psi)\|\lesssim \|\tau\|_{\cR^k}\|\psi\|$ for every
    $\tau\in\cR^k$ and $\psi\in V_h$;

    \item $\|\psi\|_1\lesssim h^{-1}\|\psi\|$ for every $\psi\in V_h$.
\end{enumerate}
\end{assumption}

	The procedure we follow to discretize $\varrho^\beta$ relies  on first discretizing the most basic operator that is used to define $\varrho^\beta$, namely, the operator $L$. To do so, we take advantage of relation \eqref{friedric_ext_rel}. Indeed, the Galerkin discretization of the operator $L$ is the unique linear operator $\widehat{L}_h: V_h\to V_h$  that satisfies the relation
	\begin{equation}\label{galerkinLdiscretization}
		(\widehat{L}_h \phi_h, \psi_h) = a_L(\phi_h, \psi_h), \quad (\phi_h, \psi_h \in V_h),
	\end{equation}
	where $a_{L}$ is the Dirichlet form associated with $L$.

	Note that $\widehat{L}_h$ is a positive-definite, symmetric, linear operator on the finite-dimensional space $V_h$. Hence, we may arrange the eigenvalues of $\widehat{L}_h$ as
	$$0<\widehat{\lambda}_{1,h} \leq \widehat{\lambda}_{2,h} \leq \cdots\leq \widehat{\lambda}_{N_h, h},$$
	with corresponding eigenvectors $\{\widehat{e}_{j,h}\}_{j=1}^{N_h}$ which are orthonormal in $L^2(\mathcal{D})$. So, based on \eqref{rel:fracov} we finally define our FEM discretization, or simply Galerkin discretization, for $\varrho^\beta$ as the map $\widehat{\varrho}_h^\beta : L^{2}(\cD)\times L^{2}(\cD)\to\mathbb{R}$ given by
	\begin{equation}\label{rel:coveigen}
		\widehat{\varrho}_h^\beta(x,y) := \sum_{j=1}^{N_h} \widehat{\lambda}_{j,h}^{-2\beta} \widehat{e}_{j,h}(x) \widehat{e}_{j,h}(y),\quad \left(\hbox{for a.e. $(x,y)\in\mathcal{D}$}\right).
	\end{equation}
Alternatively, let $\Pi_h:L^2(\cD)\to V_h$ denote the $L^2(\cD)$-orthogonal projection
onto $V_h$. Since $\widehat{L}_h^{-2\beta}:V_h\to V_h$ is given by
\begin{equation}\label{frac_fdim_op}
    \widehat{L}_h^{-2\beta}\chi
    =\sum_{i=1}^{N_h}\widehat{\lambda}_{i,h}^{-2\beta}
    (\chi,\widehat{e}_{i,h})\widehat{e}_{i,h},
    \qquad \chi\in V_h,
\end{equation}
the discretization $\widehat{\varrho}_h^\beta$ can be identified with the kernel of
the operator $\widehat{L}_h^{-2\beta}\Pi_h:L^2(\cD)\to L^2(\cD)$. That is,
\begin{equation}\label{Lh_ker_idtt}
    \left(\widehat{L}_h^{-2\beta}\Pi_h f\right)(x)
    =\int_{\cD}\widehat{\varrho}_h^\beta(x,y)f(y)\,\mu(dy),
    \qquad f\in L^2(\cD),
\end{equation}
with the equality holding for $\mu$-a.e. $x\in\cD$.

As in \eqref{L_ker_idtt}, the relationship between the covariance operator and its
kernel in \eqref{Lh_ker_idtt} is central for deriving convergence rates for
$\widehat{\varrho}_h^\beta$ toward $\varrho^\beta$ in
$L^2_{\mu\otimes\mu}(\cD\times\cD)$. Indeed, such estimates are typically obtained
through the corresponding Hilbert--Schmidt norm of the associated covariance operators;
see, e.g., \cite{cox2020}.

\subsection{Precision-based discretization for the covariance}

The discretization methods proposed in this work are motivated by the need to approximate
the covariance function $\varrho^\beta$, which is related to the operator $L$ through
\eqref{L_ker_idtt}. In the Galerkin setting, the construction of an approximation to
$\varrho^\beta$ starts with the discretization of $L$, the basic operator from which
both the covariance operator $L^{-2\beta}$ and the precision operator $L^{2\beta}$ are
defined. In this sense, the central object to be discretized is the underlying precision
operator $L$; for this reason, we refer to this approach as a precision-based
discretization.

By \eqref{friedric_ext_rel}, the Galerkin discretization $\widehat{L}_h$ of $L$ is
completely determined by the restriction of the Dirichlet form $a_L$ to
$V_h\times V_h$. Therefore, discretizing $L$ amounts to specifying, for each $h$, a
bilinear form $a_h$ on $V_h\times V_h$ that plays the role of $a_L$ in
\eqref{galerkinLdiscretization}. This motivates the following assumption on the family
$\{a_h\}_{h\in(0,1)}$.

\begin{definition}[Admissible bilinear form]\label{def_admissible_bilinear_form}
A family of symmetric bilinear forms $\{a_h\}_{h\in(0,1)}$, with
$a_h:V_h\times V_h\to\bbR$, is called admissible if
\begin{equation}\label{cond_ord2_bilin_form}
    |a_L(\phi_h,\psi_h)-a_h(\phi_h,\psi_h)|
    \lesssim h^2\|\phi_h\|_1\|\psi_h\|_1,
    \qquad \phi_h,\psi_h\in V_h.
\end{equation}
\end{definition}

	\begin{remark}\label{remark_ord2_bilin_form}
		Note that, a simple consequence of the reverse triangle inequality on \eqref{cond_ord2_bilin_form} is that $a_{h}$ is continuous relative to the $\dot{H}_L^{1}(\cD)$-norm. Moreover, from \eqref{dirich_form_coerc_rel} there exists a constant $C>0$ independent of $h$ such that
		$$a_{h}(u,u) \geq a_{L}(u,u)-Ch^{2}\|u\|^{2}_{1} = (1-Ch^{2})a_{L}(u,u) \geq (1-Ch^{2}) \lambda_1\|u\|^{2}\,\quad (u\in V_{h}),$$
		implying that $a_{h}$ is non-degenerate for every $h\in (0,1)$.
	\end{remark}
Definition~\ref{def_admissible_bilinear_form} asserts that an admissible bilinear form
$a_h$ approximates $a_L$ with second-order accuracy. In Section~\ref{sec:examples},
where explicit expressions for $a_L$ are available, admissible families of bilinear
forms $a_h$ arise naturally by applying appropriate quadrature rules to the integrals
defining $a_L$. These quadrature rules also induce mass lumping, leading to diagonal
mass matrices in the corresponding matrix representations and, consequently, to more
efficient computations.

Observe that the right-hand side of \eqref{Lh_ker_idtt} depends on the representation
of the operator with respect to the $L^2(\cD)$-inner product on $V_h$. We may generalize
this construction by considering representations with respect to a broader class of
inner products. This motivates the following definition, which we also borrow from
\cite{almeida-sousa2026finite}.

\begin{definition}[Admissible inner product]\label{def_admissible_inner_product}
We say that a family of inner products $\{\langle\cdot,\cdot\rangle_h\}_{h\in(0,1)}$, with
$\langle\cdot,\cdot\rangle_h$ defined on $V_h\times V_h$, is admissible if
\begin{enumerate}
    \item $\|\cdot\|_h$ and $\|\cdot\|$ are equivalent on $V_h$, with constants
    independent of $h$;

    \item for every $\varphi,\chi\in V_h$,
    $|\langle\varphi,\chi\rangle_h-(\varphi,\chi)|
    \lesssim h^2|\varphi|_1|\chi|_1,$
\end{enumerate}
where $\|\cdot\|_h$ denotes the norm induced by $\langle\cdot,\cdot\rangle_h$.
\end{definition}

Each admissible choice of the bilinear form $a_h$ and inner product
$\langle\cdot,\cdot\rangle_h$ gives rise to a corresponding discretization of the
underlying operator $L$, and therefore of the precision operator $L^{2\beta}$. This
flexibility is particularly relevant for the fractional problems considered here; see,
e.g., \cite{almeida-sousa2026finite}. In the applications discussed in
Section~\ref{sec:examples}, the main example of an admissible inner product is the
lumped mass inner product.

The discretization of the precision $L$ based on $\{a_{h}\}_{h\in(0,1)}$ is the family of approximate operators $L_{h}:V_{h}\to V_{h}$ uniquely determined by the relation
\begin{equation}\label{a_lump}
	\langle L_h \phi_h, \psi_h\rangle_{h} = a_h(\phi_h, \psi_h), \quad (\phi_h, \psi_h \in V_h),
\end{equation}
where $\langle \cdot, \cdot \rangle_{h}$ is an admissible inner product. By Remark \ref{remark_ord2_bilin_form} the operator $L_h$ is invertible. Moreover, from the symmetry of $a_{h}$, we may determine a set of eigenvalues $0<\lambda_{1,h} \leq \lambda_{2,h}\leq\cdots\leq \lambda_{N_h,h}$ with associated eigenvectors $\{e_{j,h}\}_{j=1}^{N_h}$ of $L_h$, which are orthonormal in $(V_h, \langle \cdot,\cdot\rangle_h)$. Therefore, our discretization of $\varrho^{\beta}$ will be given by the family of functions $\varrho_h^\beta:\cD \times \cD \to\mathbb{R}$ given by
\begin{equation*}
	\varrho_h^\beta(x,y) := \sum_{j=1}^{N_h} \lambda_{j,h}^{-2\beta} e_{j,h}(x) e_{j,h}(y),\quad \left(\hbox{for a.e. $(x,y)\in\mathcal{D}$}\right).
\end{equation*}
Similarly to \eqref{frac_fdim_op}, we may consider fractional powers of $L_h$.
To relate the resulting operator to functions in $L^2(\cD)$, we introduce the operator
$\Lambda_h:L^2(\cD)\to V_h$ by
$$
    \langle \Lambda_h f,\psi_h\rangle_h=(f,\psi_h),
    \qquad f\in L^2(\cD),\ \psi_h\in V_h,
$$
where $\langle\cdot,\cdot\rangle_h$ denotes the inner product on $V_h$. Then, one can readily check that the
precision-based approximation of the covariance operator satisfies
\begin{equation}\label{tildeapproxforL}
    \left(L_h^{-2\beta}\Lambda_h f\right)(y)
    =\int_{\cD}\varrho_h^\beta(x,y)f(x)\,\mu(dx),
    \qquad f\in L^2(\cD).
\end{equation}

A natural question is what the practical effect is of replacing the Galerkin covariance
approximation $\widehat{\varrho}_h^\beta$ by the precision-based approximation
$\varrho_h^\beta$. Our first main result provides an answer to this question.

\begin{theorem}\label{main_result_1}
Assume that $2\beta>\frac{1}{\alpha}$ and that $\alpha>\frac{1}{2}$. Then, for every
$\eta<\min\{4\beta-\frac{1}{\alpha},2\}$, the estimate
\begin{equation}\label{est_main_res_1}
    \left\|\varrho^\beta-\varrho_h^\beta\right\|_{L^2(\cD\times\cD)}
    \lesssim
    \max\left\{
    \left\|\varrho^\beta-\widehat{\varrho}_h^\beta\right\|_{L^2(\cD\times\cD)},
    h^\eta
    \right\},
\end{equation}
holds for every sufficiently small $h>0$.
\end{theorem}

Estimate \eqref{est_main_res_1} shows that the precision-based approximation inherits
the convergence rate of the Galerkin approximation, up to the additional term $h^\eta$.
Thus, whenever the Galerkin covariance approximation is known to converge with order no
smaller than $\eta$, that order is preserved by the precision-based approximation.

The practical advantage of this replacement is algebraic. In Section~\ref{sec:examples},
we provide examples of admissible inner products and admissible bilinear forms for which the
matrix expressions of the discretized covariance operator $L_h^{-2\beta}\Lambda_h$
involve diagonal mass matrices at every occurrence. This is achieved through suitable
quadrature rules, which both approximate the relevant bilinear forms and induce mass
lumping. One should note that the inverse mass matrix appearing in the matrix
representations discussed in the introduction is related to the change from the
$L^2(\cD)$ inner product $(\cdot,\cdot)$ to the discrete inner product
$\langle\cdot,\cdot\rangle_h$. This change of inner product is precisely accounted for
by the operator $\Lambda_h$.

In Section~\ref{sec:examples}, we show how Theorem~\ref{main_result_1} applies to the
specific classes of spatial domains introduced in Example~\ref{example_domain}. Using
known convergence estimates from the literature, we verify that the order of convergence
of the Galerkin approximation is preserved in these examples.

\subsubsection{Approximation for the fractional power operation} 

Note that when $\beta\in \frac{n}{2}$ with $n\in \mathbb{Z}$, our covariance approximation is computable. On the other hand, this is not the case if the fractional power is not a half-integer and an additional approximation step must be incorporated to the method.  

Recall from \eqref{tildeapproxforL} that the approximating covariance function is
defined through the operator $L_h^{-2\beta}\Lambda_h$. Hence, the problem of
approximating the covariance function reduces to approximating
$L_h^{-2\beta}\Lambda_h$. Since $\Lambda_h$ is explicitly defined by the discrete
variational relation introduced in the previous section, the computationally nontrivial
part is the evaluation of negative fractional powers of $L_h$. For this reason, we
focus in this section on constructing accurate approximations of $L_h^{-\beta}$ for
$\beta>0$.

By definition, using the eigendecomposition of $L_h$ with respect to $(V_h, \langle\cdot,\cdot\rangle_h)$, if $(\lambda_{h,k},e_{h,k})_{k=1}^{N_h}$ denotes an $\langle\cdot,\cdot\rangle_h$-orthonormal
eigensystem of $L_h$, then
\begin{equation*}\label{widetiapproxop}
    L_h^{-\beta}\chi
    =\sum_{k=1}^{N_h}\lambda_{h,k}^{-\beta}\langle\chi,e_{h,k}\rangle_h e_{h,k},
    \qquad \chi\in V_h.
\end{equation*}
Thus, approximating $L_h^{-\beta}$ amounts to approximating the scalar function
$t\mapsto t^{-\beta}$ on the spectrum of $L_h$. By the rescaling
$L_h\mapsto \lambda_1^{-1}L_h$, we assume without loss of generality throughout this section that this
spectrum is contained in $[1,\infty)$. The following assumption specifies what we mean
by a sufficiently accurate approximation of this scalar function.

\begin{assumption}[$\beta$-power approximation]\label{fracpoappass}
Let $0<\beta<1$. We assume that, for every $m\in\mathbb{N}^{*}$, there exists a
continuous function $Q_m^\beta:[1,\infty)\to\mathbb{R}$ and an error bound
$f^\beta(\beta,m)>0$ such that
\begin{equation}\label{rel_beta_power_approx}
    \sup_{\lambda\geq 1}
    \left|\lambda^{-\beta}-Q_m^\beta(\lambda)\right|
    \lesssim f^\beta(\beta,m).
\end{equation}
\end{assumption}

The approximation functions introduced in Assumption~\ref{fracpoappass} will be used
to approximate the non-integer negative powers of $L_h$ appearing in the discrete
covariance operator. Let $s_\beta:=2\beta$ and, when $s_\beta\notin\mathbb{N}$, write
$\theta_\beta:=s_\beta-\lfloor s_\beta\rfloor\in(0,1)$. We define the approximate
discrete covariance operator by
$$
T_{m,h}^\beta :=
\begin{cases}
    L_h^{-2\beta}\Lambda_h, & 2\beta\in\mathbb{N},\\
    L_h^{-\lfloor 2\beta\rfloor}Q_m^{\theta_\beta}(L_h)\Lambda_h,
    & 2\beta\notin\mathbb{N}.
\end{cases}
$$
The fractional-power approximation of the covariance function is then defined as the
kernel $\varrho_{m,h}^\beta:\cD\times\cD\to\mathbb{R}$ associated with
$T_{m,h}^\beta$, that is,
\begin{equation}\label{rho_approx_cov}
    T_{m,h}^\beta f(x)
    =\int_{\cD}\varrho_{m,h}^\beta(x,y)f(y)\,\mu(dy),
    \qquad f\in L^2(\cD),
\end{equation}
with the equality understood in $L^2(\cD)$, or $\mu$-a.e. in $x$.

When $2\beta\in\mathbb{N}$, no approximation of a fractional power is needed, and
$\varrho_{m,h}^\beta$ coincides with the Galerkin covariance kernel
$\varrho_h^\beta$. When $2\beta\notin\mathbb{N}$, the explicit form of
$\varrho_{m,h}^\beta$ depends on the choice of the scalar approximation
$Q_m^{\theta_\beta}$. In the applications later, we focus on two
such choices: sinc-quadrature approximations, see \cite{bonito2024numerical}, and
rational approximations, see \cite{bolin2024covariance}. Both lead to explicit
expressions for the approximate covariance kernel. We examine these two cases in more
detail below.

\subsubsection{Rational Approximation} 

The rational approximation of the covariance is an additional step used to efficiently
approximate $\varrho^\beta$ when $\beta\neq n/2$ for every $n\in\bbN$. In this case,
we do not have explicit expressions for $\varrho_h^\beta$ in terms of the stiffness
and (lumped) mass matrices.
The main idea is to reduce the problem to the rational approximation of the scalar
function $x\mapsto x^\beta$. Indeed, negative fractional powers of $L_h$ can be written
as positive fractional powers of its inverse, namely $L_h^{-\beta}=(L_h^{-1})^\beta$.
This formulation is particularly convenient because the spectrum of $L_h^{-1}$ is
contained in a bounded interval, uniformly in $h$, whereas the spectrum of $L_h$ grows
as $h\to0$. Thus, instead of approximating $x\mapsto x^{-\beta}$ on an $h$-dependent
unbounded range, we approximate $x\mapsto x^\beta$ on a compact interval containing the
spectrum of $L_h^{-1}$.

In what follows, we adopt the procedure introduced in \cite{bolin2024covariance} to
construct a rational approximation of $L_h^{-\beta}$. This approach is widely used in
practical applications and leads to explicit matrix-based expressions for the resulting
approximation of the covariance kernel.
We first write $x^\beta=\widehat f(x)x^{\floor{\beta}}$, where
$\widehat f(x):=x^{\beta-\floor{\beta}}$. If $\sigma(L_{h}^{-1})$ denotes the spectrum of
$L_{h}^{-1}$, then, since $\sigma\left(L_h^{-1}\right) \subseteq\left[0,1 / \lambda_{1, h}\right] \subseteq\left[0,1 / \lambda_1\right]$, after dividing by $1 / \lambda_1$ we may assume $\sigma\left(L_h^{-1}\right) \subseteq[0,1]$.

Let $\widehat r_m=q_1/q_2$ be the best rational approximation of $\widehat f$ in the
$L^\infty([0,1])$-norm among rational functions whose numerator and denominator have
degree at most $m$; see, for instance, \cite{stahl1993best}. We then define the
corresponding rational approximation of $x^\beta$ by
$\overline r_{\beta,m}(x):=\widehat r_m(x)x^{\floor{\beta}}$. This approximation is
not necessarily the best rational approximation of $x^\beta$, but it is obtained
directly from the best rational approximation of its fractional part.
Equivalently, this construction yields a rational approximation of $x^{-\beta}$ on
any compact interval $[a,b]\subset[1,\infty)$ by setting
$$
r_{-\beta,m}(x):=\overline r_{\beta,m}(x^{-1})
=\widehat r_m(x^{-1})x^{-\floor{\beta}}
=\frac{q_1(x^{-1})}{q_2(x^{-1})x^{\floor{\beta}}}
=\frac{p_1(x)}{p_2(x)}.
$$
Here $p_1$ and $p_2$ are polynomials of degree at most $m$ and
$m+\floor{\beta}$, respectively.

In this setting, our candidate approximation of $L_h^{-2\beta}$ is
$$
R_{2\beta,m,h}:=r_{-2\beta,m}(L_h)
=\overline r_{2\beta,m}(L_h^{-1}).
$$
Consequently, when $2\beta\notin\mathbb{N}$, we approximate the covariance operator
$L_h^{-2\beta}\Lambda_h$ by $R_{2\beta,m,h}\Lambda_h$.
As emphasized above, the purpose of the rational approximation is to obtain a
computable approximation of $\varrho^\beta(x,y)$ when $2\beta\notin\mathbb{N}$. In this
case, let $r_{m,h}^\beta$ denote the kernel associated with the operator
$R_{2\beta,m,h}\Lambda_h$. We define the rational-based approximation of the covariance
kernel by
\begin{equation}\label{rat_approx_cov}
    \varrho_{m,h}^\beta(x,y) :=
    \begin{cases}
        \varrho_h^\beta(x,y), & 2\beta\in\mathbb{N},\\
        r_{m,h}^\beta(x,y), & 2\beta\notin\mathbb{N}.
    \end{cases}
\end{equation}

The proof of Proposition~2 in \cite{bolin2024covariance} yields, for
$2\beta\notin\mathbb{N}$,
\begin{equation}\label{rel_exp_conv}
    \left\|L_h^{-2\beta}\Lambda_h-R_{2\beta,m,h}\Lambda_h\right\|_{L_2(L^2(\cD))}
    \lesssim h^{-1/\alpha}e^{-2\pi\sqrt{\{2\beta\}m}},
\end{equation}
where $\{s\}:=s-\floor{s}$ denotes the fractional part of $s$. This estimate is one of
the key ingredients in the proof of the following result.

\begin{theorem}\label{rat_approx_theorem}
Let $\widehat{\varrho}_h^\beta$ and $\varrho_{m,h}^\beta$ denote, respectively, the
Galerkin approximation and the rational-based approximation of $\varrho^\beta$. Under
the assumptions of Theorem~\ref{main_result_1}, for every
$\eta<\min\{4\beta-\frac{1}{\alpha},2\}$, the estimate
\begin{equation}\label{err_rat_cov}
    \begin{split}
        \left\|\varrho^\beta-\varrho_{m,h}^\beta\right\|_{L^2(\cD\times\cD)}
        &\lesssim
        \max\left\{
        \left\|\varrho^\beta-\widehat{\varrho}_h^\beta\right\|_{L^2(\cD\times\cD)},
        h^\eta
        \right\} \\
        &\quad
        +\mathbbm{1}_{\{2\beta\notin\mathbb{N}\}}
        h^{-1/\alpha}e^{-2\pi\sqrt{\{2\beta\}m}},
    \end{split}
\end{equation}
holds for every sufficiently small $h>0$.
\end{theorem}

A useful feature of the convergence estimate in \eqref{err_rat_cov} is that, when
$2\beta\notin\mathbb{N}$, the rational approximation term on the right-hand side
decays root-exponentially fast with respect to the rational degree $m$. Thus, the
parameter $m$ can be chosen as a function of the mesh size $h$ so that the rational
approximation error does not affect the overall convergence rate.

More precisely, for a given $h$, we may choose $m=m(h)$ such that
$$
h^{-1/\alpha}e^{-2\pi\sqrt{\{2\beta\}m(h)}}
\lesssim
\max\left\{
\left\|\varrho^\beta-\widehat{\varrho}_h^\beta\right\|_{L^2(\cD\times\cD)},
h^2
\right\}.
$$
With this choice, the rational approximation error is dominated by the spatial
discretization error, and \eqref{err_rat_cov} yields the formal convergence rate
$$
\max\left\{
\left\|\varrho^\beta-\widehat{\varrho}_h^\beta\right\|_{L^2(\cD\times\cD)},
h^{\min\{4\beta-\frac{1}{\alpha},2\}}
\right\},
$$
up to the usual interpretation that the exponent
$\min\{4\beta-\frac{1}{\alpha},2\}$ may be approached from below.

Equivalently, if one wants the rational approximation term to be of order $h^r$, for
some $r>0$, it is enough to choose $m(h)$ of order $|\log h|^2$. Indeed,
$$
m(h)\gtrsim
\frac{(r+1/\alpha)^2}{4\pi^2\{2\beta\}}|\log h|^2
$$
implies
$h^{-1/\alpha}e^{-2\pi\sqrt{\{2\beta\}m(h)}}\lesssim h^r$.


\begin{remark}
The proof of Theorem~\ref{rat_approx_theorem} only uses the root-exponential
uniform error bound in \eqref{rel_exp_conv}. Thus, the result applies to any
rational approximation satisfying this bound, including BURA
methods~\citep{stahl1993best,harizanov2020analysis}; related approaches include
sinc-quadrature~\citep{bonito2015numerical,bonito2019sinc}, rational SPDE
approximations~\cite{BK2020rational}, and covariance-based rational
approximations~\cite{bolin2024covariance}; see also
\cite{hofreither2020unified}.
\end{remark}

\subsection{Non-Stationary Variance Control Factor} 
In this section we introduce an approximation for the covariance function associated with fractional SPDEs 
\begin{equation}\label{tau_non_st_case}
	L^{\beta}(\tau u)=\cW.
\end{equation}
The main difference from the previously considered fractional SPDEs lies in the introduction of the factor $\tau$ that controls the variance of the solution $u$. When $\tau$ is a function of the spatial domain $\mathcal{D}$, we say that $\tau$ is non-stationary, and it is referred to as stationary only when $\tau$ is constant over the spatial domain.

When $\tau$ is constant, say $\tau\equiv\tau_0>0$, equation
\eqref{tau_non_st_case} reduces to $L^\beta u=\tau_0^{-1}\cW$.
Hence, the covariance kernel is simply
rescaled by the factor $\tau_0^{-2}$, and the discretizations developed in the previous
sections apply directly. The new case is therefore the one in which $\tau$
varies over the spatial domain.

Before turning to the technical details, we recall that spatially varying variance
parameters are commonly used in applications of the SPDE approach, in particular in
non-stationary models implemented through \textsf{R-INLA}; see, for instance,
\citep{lindgren2015rinla,krainski2019advanced}. More generally,
non-stationary SPDE models with spatially varying coefficients have been considered in
\citep{ingebrigtsen2014spatial,fuglstad2015exploring,bakka2019}. The practical
motivation for the discretization proposed below comes from the finite element matrix
representations used in these implementations, where spatially varying coefficients are
represented on the mesh and the corresponding multiplication operators are assembled
through nodal or quadrature-based approximations. Our aim is to analyze this procedure
at the level of the covariance function.

There are two points that have not been fully addressed in the convergence theory for
the SPDE approach with non-stationary variance factors. The first is to obtain an
explicit convergence rate for the covariance approximation when $\tau$ varies over
$\cD$. The second is to understand how the regularity of $\tau$ influences this rate.

Let us now give a technical motivation for the main idea behind our discretization.
Assume that $\tau$ is bounded and bounded away from zero, analogously to the assumption
imposed on $\kappa$ in Assumption~\ref{ass_coeff}. Let $U^\beta$ denote the solution of
$$L^\beta v=\cW.$$
Then, the solution $u$ of \eqref{tau_non_st_case}, by definition, is given by
$u=\tau^{-1}U^\beta$.

Suppose that $U_h^\beta$ is a finite element approximation of $U^\beta$, whose
covariance can be approximated with order up to $2$ by the methods developed in the
previous sections. A natural candidate for approximating $u$ is then
$\tau^{-1}U_h^\beta$. However, this function does not, in general, belong to the finite
element space $V_h$. In particular, it is not determined solely by its coefficients at
the mesh nodes. To obtain an approximation that remains in $V_h$, we instead consider
$\cI_h(\tau^{-1}U_h^\beta)$, where $\cI_h$ denotes the Lagrange interpolation operator,
which is exact on linear polynomials.
	
Let $M_{\tau^{-1}}:L^2(\cD)\to L^2(\cD)$ denote the multiplication operator defined by
$M_{\tau^{-1}}f=\tau^{-1}f$. Since $L^\beta(\tau u)=\cW$ is equivalent to
$u=M_{\tau^{-1}}U^\beta$, where $U^\beta$ solves $L^\beta U^\beta=\cW$, the exact
covariance operator is $M_{\tau^{-1}}L^{-2\beta}M_{\tau^{-1}}$. We denote its kernel by
$\varrho^{\beta,\tau}$. Equivalently,
$\varrho^{\beta,\tau}(x,y)=\tau(x)^{-1}\varrho^\beta(x,y)\tau(y)^{-1}$ for
$\mu\otimes\mu$-a.e. $(x,y)\in\cD\times\cD$.

We now define the corresponding discrete approximation, including both the rational
approximation and the mass-lumped discretization. Let
$B_h:=\cI_h M_{\tau^{-1}}|_{V_h}$, and let $B_h^*$ denote its adjoint with respect to
the $L^2(\cD)$ inner product on $V_h$. We define $\varrho_{m,h}^{\beta,\tau}$ as the
kernel associated with the finite-rank operator
$B_h R_{2\beta,m,h}\Lambda_h B_h^*\Pi_h$.
That is,
\begin{equation}\label{tau_cov_approx}
    \left(B_h R_{2\beta,m,h}\Lambda_h B_h^*\Pi_h f\right)(x)
    =
    \int_{\cD}\varrho_{m,h}^{\beta,\tau}(x,y)f(y)\,\mu(dy),
    \qquad f\in L^2(\cD),
\end{equation}
with the equality understood in $L^2(\cD)$, or equivalently $\mu$-a.e. in $x$.

We then have the following convergence result.

\begin{theorem}\label{prop_fin}
Assume that $\tau$ is bounded and bounded away from zero, in the same sense discussed in
Assumption~\ref{ass_coeff}~C2), and suppose that $\tau^{-1}\in\cR^k$, where $\cR^k$ is the
multiplier space introduced in Assumption~\ref{assump2}. Under the assumptions of
Theorem~\ref{main_result_1}, assume also that
$2\beta>\max\{\frac{1}{\alpha},\frac{1}{2\alpha}+\frac{1}{2}\}$. Then, for every
$\eta<\min\{4\beta-\frac{1}{\alpha},2\}$, the estimate
\begin{equation}\label{fin_res_coverr}
    \begin{split}
        \left\|\varrho^{\beta,\tau}
        -\varrho_{m,h}^{\beta,\tau}\right\|_{L^2(\cD\times\cD)}
        &\lesssim
        \max\left\{
        \left\|\varrho^\beta-\widehat{\varrho}_h^\beta\right\|_{L^2(\cD\times\cD)},
        h^\eta
        \right\}
        + h^k \\
        &\quad
        +\mathbbm{1}_{\{2\beta\notin\mathbb{N}\}}
        h^{-1/\alpha}e^{-2\pi\sqrt{\{2\beta\}m}},
    \end{split}
\end{equation}
holds for every sufficiently small $h>0$.
In particular, if $k\geq 2$, the additional error caused by the discretization of the
variance-control factor does not reduce the convergence order obtained in
\eqref{err_rat_cov}.
\end{theorem}

The estimate \eqref{fin_res_coverr} makes explicit how the regularity of the
variance-control factor affects the covariance approximation. The additional term
$h^k$ comes from replacing the product $\tau^{-1}U_h^\beta$ by its interpolant
$\cI_h(\tau^{-1}U_h^\beta)$, which is necessary in order to keep the approximation
inside the finite element space $V_h$. This interpolation error is present even without
rational approximation or mass lumping; it is intrinsic to the discretization of the
non-stationary multiplication factor. When $k\geq 2$, this term is of at least second
order and is therefore dominated by the spatial approximation order considered in
Theorem~\ref{main_result_1}.

\section{Examples and Applications}\label{sec:examples}

In this section we discuss how the results stated in the previous section apply for the particular cases of the spatial domain $\cD$ described in Example \ref{example_domain}. We begin with some required notation and a few general assumptions. The goal is to identify, in each setting, concrete choices of admissible inner products
and admissible bilinear forms in the sense of Definitions~\ref{def_admissible_inner_product}
and~\ref{def_admissible_bilinear_form}.

Let $d\in \mathbb{N}\setminus \{0\}$ be the dimension of the spatial domain $\cD$ in the sense that the relation \eqref{weyl_law} of Assumption \ref{ass_operator} is satisfied with $\alpha=\frac{2}{d}$. Further, let 
\begin{equation}\label{H1_def}
	H^{1}(\cD) := \left\{f\in L^{2}(\cD); \exists \nabla_{\cD} f \,\mbox{in the weak sense and}\, |\nabla_{\cD} f|_{\mathbb{R}^{d}}\in L^{2}(\cD)\right\},
\end{equation} 
denote the Sobolev space of all functions $f\in L^{2}(\cD)$ such that the weak gradient $\nabla_{\cD}f(x)=(f_{x_1}(x),\ldots,f_{x_d}(x))$ satisfies 
$|\nabla_{\cD} f|_{\mathbb{R}^{d}}:=\bigl(\sum_{i=1}^{d} f_{x_i}(x)^{2}\bigr)^{1/2} \in L^{2}(\cD)$. The notion of weak gradient is defined properly for each different type of spatial domain $\cD$ below. The space $H^{1}(\cD)$ is a Hilbert space when equipped with the inner product
\begin{equation}\label{H1_inner_prod}
	\langle f,g\rangle_{H^{1}(\cD)} := \int_{\cD} fg\,d\lambda + \int_{\cD} \nabla_{\cD} f \cdot \nabla_{\cD} g\,d\lambda.
\end{equation}
Inductively, if we assume well defined $H^{k}(\cD)$ for some $k\geq 2$, we may define the high order Sobolev space
\begin{equation}\label{Hk_def}
	H^{k+1}(\cD) := \left\{f\in H^{k}(\cD);\forall i=1,\ldots, d, \exists \nabla_{\cD} f_{x_i}\,\text{and}\,|\nabla_{\cD} f_{x_{i}}|_{\mathbb{R}^{d}}\in L^{2}(\cD) \right\}.
\end{equation} 
In this case $H^{k+1}(\cD)$ is a Hilbert space when equipped with the inner product
\begin{equation}\label{Hk_inner_prod}
	\langle f,g\rangle_{H^{k+1}(\cD)} := \inner{f}{g}_{H^{k}(\cD)} + \sum_{i=1}^{d} \inner{\nabla_{\cD} f_{x_i}}{\nabla_{\cD} g_{x_i}}_{L^2(\cD)}.
\end{equation}
In a similar fashion, for every $k\geq 1$ we may define $W^{k,\infty}(\cD)$ by replacing the space $L^{2}(\cD)$ in \eqref{Hk_def} and \eqref{H1_def} by $L^{\infty}(\cD)$. The space $W^{k+1,\infty}(\cD)$ is a Banach space when equipped with the norm
\begin{equation}\label{W1_infty_norm}
	\|f\|_{W^{k+1,\infty}(\cD)} := \|f\|_{W^{k,\infty}(\cD)} + \sum_{i=1}^{d} \esssup_{x\in \cD} |\nabla_{\cD} f_{x_i}(x)|.
\end{equation}

Throughout this section assume that $\dot{H}^{1}_{L}(\cD)\subseteq H^{1}(\cD)$ and we let $L:\mathscr{D}(L)\subset \dot{H}^{1}_{L}(\cD)\to L^2(\cD)$ be the second-order elliptic differential operator formally given by
\begin{equation}\label{form_expr_L}
	Lu=-\operatorname{div}_{\cD} \cdot(\boldsymbol{H} \nabla_{\cD} u)+\kappa^2 u, \quad(u \in \mathscr{D}(L)).
\end{equation}
The relation between $L$ and $a_{L}$ is one-to-one and this connection is given through \eqref{friedric_ext_rel}. Therefore, taking advantage of this correspondence, we focus on rigorously defining the expressions for $a_{L}$ and $a_{h}$. To do so, we make the following basic assumption concerning the coefficients $\boldsymbol{H}$ and $\kappa$ in \eqref{form_expr_L}.

\begin{assumption}[Coefficients]\label{ass_coeff}
	We assume the following general conditions on the coefficients:
	\begin{enumerate}[label=\roman*.]
		\item[C1)]\label{C1} $\boldsymbol{H}$ is a well defined map over $\cD$, such that $\boldsymbol{H}(x)$ is an endomorphism in the tangent space at $x$ (denoted by $T_{x}\cD$), that is represented by a symmetric matrix $(H_{i,j}(x))_{i,j}$. 
		Further, $\boldsymbol{H}$ is uniformly elliptic, i.e., there exists $H_{-}>0$ such that 
		\begin{equation*}
			 \boldsymbol{\xi} \cdot \boldsymbol{H}(x)\boldsymbol{\xi} \geq  |\boldsymbol{\xi}|_{\bbR^{d}}^2 H_{-}, \quad (\boldsymbol{\xi}\in T_{x}\cD, x\in\mathcal{D})
		\end{equation*}
		where $\boldsymbol{\xi}=(\xi_{1},\ldots,\xi_{d})$ is the vector-wise representation of the tangent vector at $x\in\cD$.
		\item[C2)]\label{C2} $\kappa : \mathcal{D}\to \mathbb{R}$ bounded and bounded away from zero, that is, there exist positive constants $\kappa_{-}$ and $\kappa_{+}$ such that 
			 $\kappa_{-} \le \kappa \le \kappa_{+}$.
		\item[C3)]\label{C3} We assume that $\kappa, H_{i,j} \in W^{1,\infty}(\cD)$, where  $H_{i,j}$, for every $1\le i,j\le d$, denote the coordinate functions of $\boldsymbol{H}$. 
	\end{enumerate}
\end{assumption}


The bilinear form corresponding to the operator $L$ is given by
\begin{equation}\label{bilin_form_L_exmp}
	a_{L}(\phi,\psi) := (\boldsymbol{H}\nabla_{\cD} \phi, \nabla_{\cD} \psi)_{L^2(\cD)} + (\kappa \phi, \kappa \psi)_{L^2(\cD)}, \quad (\phi,\psi \in \dot{H}^{1}_{L}(\cD)),
\end{equation}
where $(\boldsymbol{H}\nabla_{\cD} \phi, \nabla_{\cD} \psi)_{L^2(\cD)}:=\int_{\cD} \boldsymbol{H}\nabla_{\cD} \phi \cdot \nabla_{\cD} \psi\,d\lambda$ and since $\dot{H}^{1}_{L}(\cD)\subset H^{1}(\cD)$, the derivatives in \eqref{bilin_form_L_exmp} are understood in the weak sense.

The discussion conducted thus far provides sufficient foundation to ensure that Assumption \ref{ass_operator} is satisfied across all our examples. We now turn our attention to the construction and analysis of our discrete structures.
Given a family of triangulations $\{\cT_{h}\}_{h\in(0,1)}$ of $\overline{\cD}$ (the closure of $\cD$), indexed by the mesh width $h:=\max_{T\in\mathcal{T}_h} h_T$, where $h_T:= \mbox{diam}(T)$ is the diameter of the element $T\in \mathcal{T}_h$, 
we let $V_h \subset \dot{H}^{1}_{L}(\cD)$ be the (conforming) finite element space with continuous piecewise linear basis functions $\{\varphi_j\}_{j=1}^{N_h}$, with $N_h\in\mathbb{N}$. Additionally, denote by $\rho_T$ the radius of the largest ball inscribed in $T\in\mathcal{T}_h$. Further, we assume that the family of triangulations $\{\cT_{h}\}_{h\in(0,1)}$ is quasi-uniform, i.e., there exist constants $K_1,K_2>0$ such that $\rho_T \geq K_1 h_T$ and $h_T\geq K_2 h$ for all $T\in \mathcal{T}_h$ and $h\in(0,1)$. For each $\cD$ we let $\cI_{h}$ denote the corresponding Lagrange interpolator defined over $\mathscr{D}(\cI_{h})=\dot{H}^{1}_{L}(\cD)\cap C(\cD)$, where $C(\cD)$ denotes the space of continuous functions over $\cD$.

Before we proceed with the definition of the family of bilinear forms $a_{h}$ some basic notation and definitions are needed. Let $d$ be the dimension of $\cD$ and fix some $h\in (0,1)$. Also, let $V(T)=\{v_{j}(T)\}_{j=1}^{d+1}$ denote the vertices of $T$, $V=\bigcup_{T\in\cT_{h}} V(T)$ the set of nodes of $\cT_{h}$ and $b(T)$ the barycenter of the simplex $T$. Considering the segments that connect $b(T)$ to the midpoints of the edges of $T$ (or one-dimensional faces of $T$), we obtain a partition $\cA(T)=\left\{\cA_{z}(T)\right\}_{j=1}^{d+1}$ of $T$ into $d+1$ regions of equal $\lambda$-measure. That is, $\lambda(\cA_{j}(T))=\frac{\lambda(T)}{d+1}$, for all $z=1,\ldots, d+1$. The dual mesh of $\cT_{h}$ is the partition $\cB=\{B_{z}\}_{z\in V}$ of $S$ such that, given any node $z$ in the triangulation $\cT_{h}$, we put $B_{z}=\bigcup_{T\in \{L\in \cT_{h}; z \in V(L)\}}\cA_{z}(T)$.

The \textit{lumped mass inner product} on the trial space $V_{h}$, is the inner product 
\begin{equation}\label{quadrature}
	\begin{split}
	\<\phi,\psi\>_h &:= \sum_{z\in V} \sum_{T\in \{L\in \cT_{h}; z \in V(L)\}} \!\!\!\phi_{z}(z)\psi_{h}(z) \lambda(\cA_{z}(T)) 
	= \sum_{z\in V}\phi_{h}(z)\psi_{h}(z)\lambda(B_{z}).
\end{split}
\end{equation}
The inner product in \eqref{quadrature} is the prototype of the admissible inner
products considered in Definition~\ref{def_admissible_inner_product}. It is induced by
a quadrature rule that is exact on polynomials of degree at most $1$ on each element
$T\in\cT_h$, in the sense that
$\langle f,\mathbbm{1}_T\rangle_h=\int_T f\,d\lambda$ whenever $f$ is affine on $T$.
Moreover, in the nodal basis, this choice yields a diagonal mass matrix, which is the
mass-lumping property used in the matrix representations of the discretized covariance
operators. For further details, we refer the reader to \cite[Chapter 15]{thomee2007galerkin}
and \cite{jin2013error}.

Taking into account \eqref{quadrature}, for each $h\in (0,1)$, we define the family of bilinear forms $a_{h}:V_{h}\times V_{h}\to \bbR$ as
\begin{equation}\label{bilin_form_h_exmp}
	a_{h}(\phi,\psi) := \langle\boldsymbol{H}\nabla_{\cD} \phi, \nabla_{\cD} \psi\rangle_{h} + \langle\kappa \phi, \kappa \psi\rangle_{h}, \quad (\phi,\psi \in V_{h}),
\end{equation}
where 
$$\langle\boldsymbol{H}\nabla_{\cD} \phi, \nabla_{\cD} \psi\rangle_{h} := \sum_{z\in V} \sum_{T\in \{L\in \cT_{h}; z \in V(L)\}} \left[\boldsymbol{H}(z)\nabla_{\cD} \phi(z) \cdot \nabla_{\cD} \psi(z) \right] \lambda(\cA_{z}(T)).$$ 
Thus, $a_h$ is obtained by applying the same quadrature rule to the integrals defining
$a_L$. In the terminology of Definition~\ref{def_admissible_bilinear_form}, the family
$\{a_h\}_{h\in(0,1)}$ provides an admissible bilinear-form approximation of $a_L$.

\begin{remark}
We emphasize that, if desired, instead of a complete numerical approximation of both integrals, we may use a quadrature rule to approximate only one of the integrals on the right-hand side of \eqref{bilin_form_L_exmp}. In this case, the resulting bilinear form $a_{h}$ is still admissible, and therefore the results of this work remain valid. 
\end{remark}

\subsection{Euclidean domains}

We first consider the case in which $\cD$ is a bounded open subset of $\bbR^d$. This is
the most classical setting and serves as a reference point for the other examples
considered below. In this case, $L^2(\cD)$, $H^1(\cD)$, and the finite element spaces
$V_h$ are understood in the standard sense of Sobolev space theory; see, e.g.,
\cite{ciarlet2002finite}. Moreover, the Weyl-type condition \eqref{weyl_law} holds
with $\alpha=2/d$, where $d$ is the dimension of $\cD$.

The space $\dot H_L^1(\cD)$ encodes the boundary condition associated with the operator
$L$. In this section, we restrict attention to the two classical homogeneous cases:
pure Dirichlet and pure Neumann boundary conditions. Thus,
\begin{equation}\label{dirichlet_neumann_spaces}
\dot{H}^{1}_{L}(\cD) :=
\begin{cases}
    H^{1}_{0}(\cD):=\overline{C^{\infty}_c(\cD)}^{\|\cdot\|_{1}},
    & \text{in the Dirichlet case},\\
    H^{1}(\cD),
    & \text{in the Neumann case}.
\end{cases}
\end{equation}
Since we work with conforming finite element spaces, $V_h\subset\dot H_L^1(\cD)$, the
functions in $V_h$ inherit the corresponding boundary condition.

Let $\gamma_1:H^2(\cD)\to H^{-1/2}(\partial\cD)$ denote the trace operator associated
with the normal derivative; see, e.g., \cite[Section~1.5.3]{grisvard2011elliptic}.
The domain of $L$ is then specified by
\begin{equation}\label{domain_spec}
\mathscr{D}(L):=\dot H_L^2(\cD):=
\begin{cases}
    H^2(\cD)\cap H^1_0(\cD),
    & \text{in the Dirichlet case},\\
    \{u\in H^2(\cD):\gamma_1u=0\},
    & \text{in the Neumann case}.
\end{cases}
\end{equation}
In both cases, the bilinear form associated with $L$ is the one given in
\eqref{bilin_form_L_exmp}.

The fractional spaces $\dot H_L^\sigma(\cD)$ can be identified with suitable Sobolev or interpolation spaces, where such an identification means that the spaces are topologically isomorphic and have equivalent norms; see \cite[Lemma~2]{cox2020} and \cite[Proposition~1]{bolin2024covariance} for precise characterizations. In the following discussion, these identifications allow us to employ standard interpolation estimates in the finite element analysis and properties of the Borel functional calculus when studying fractional powers of the continuous and discrete operators.

The following lemma quantifies the accuracy of the lumped-mass quadrature rule in
approximating the classical $L^2(\cD)$ inner product. For the proof, we refer to
\citep{thomee2007galerkin,fix1972effects}.

\begin{lemma}\label{quad_err_est}
For $T\in\cT_h$, define the quadrature error functional
\begin{equation}\label{quad_error_func}
   E_T(\varphi):=(\varphi,\mathbbm{1}_T)-\langle\varphi,\mathbbm{1}_T\rangle_h,
   \qquad \varphi\in C(\cD).
\end{equation}
Then, for every $\phi_h,\psi_h\in V_h$, the following estimates hold:
\begin{enumerate}[label=\arabic*)]
    \item
    $|E_T(\kappa \phi_h\psi_h)|
    \lesssim h_T^2\|\kappa\|_{W^{2,\infty}(T)}
    \|\phi_h\|_{1,T}\|\psi_h\|_{1,T};$

    \item
    $|E_T(a_{ij}\partial_i\phi_h\,\partial_j\psi_h)|
    \lesssim h_T^2\|a_{ij}\|_{W^{2,\infty}(T)}
    \|\phi_h\|_{1,T}\|\psi_h\|_{1,T},$
    for each coordinate function $a_{ij}$ of $\boldsymbol H$.
\end{enumerate}
In particular, if $\kappa\equiv1$, then
\begin{equation}\label{L2_quadrature}
    |\langle\phi_h,\psi_h\rangle_h-(\phi_h,\psi_h)|
    \lesssim h^2
    \|\phi_h\|_{H^1(\cD)}
    \|\psi_h\|_{H^1(\cD)}.
\end{equation}
\end{lemma}

Lemma~\ref{quad_err_est} shows that the quadrature rule defining
$\langle\cdot,\cdot\rangle_h$ provides the required second-order control of the
quadrature error. Consequently, the lumped-mass inner product introduced in
\eqref{quadrature} is an admissible inner product in the sense of
Definition~\ref{def_admissible_inner_product}. Moreover, when the same quadrature rule
is used in the definition of $a_h$ in \eqref{bilin_form_h_exmp}, the family
$\{a_h\}_{h\in(0,1)}$ defines an admissible bilinear form in the sense of
Definition~\ref{def_admissible_bilinear_form}; see, for example, \cite[Remark~2.6]{almeida-sousa2026finite}. Thus, the assumptions required to apply
the abstract results of Section~\ref{sec:fem} are satisfied in the Euclidean setting.

In this Euclidean setting, the multiplier space $\cR^k$ appearing in
Assumption~\ref{assump2} can be taken to be the Sobolev--multiplier space
$W^{k,\infty}(\cD)$, with the natural compatibility conditions imposed by the boundary
condition associated with $L$. For $k\leq 2$, in the Neumann case this simply gives
$\cR^k=W^{k,\infty}(\cD)$. In the Dirichlet case, the finite element functions already
satisfy the homogeneous boundary condition, and multiplication by functions in
$W^{k,\infty}(\cD)$ preserves the required boundary behavior. The standard Lagrange
interpolation estimates and inverse inequalities on quasi-uniform meshes then imply
the approximation and stability properties required in Assumption~\ref{assump2}; see,
for example, \cite[Chapter~4]{brenner2008mathematical}.


For Euclidean domains of dimension $d\leq3$, \cite[Proposition~4]{cox2020} establishes the Galerkin covariance convergence order $\min\{4\beta-d/2,2\}$. Since $\alpha=2/d$ in this setting, Theorem~\ref{main_result_1} shows that mass lumping preserves this order. We therefore obtain the following consequence.

\begin{corollary}\label{cor1}
Let $d\in\{1,2,3\}$. Under the assumptions of Theorem~\ref{main_result_1}, the lumped-mass covariance approximation converges with order $\min\{4\beta-d/2,2\}$, which is the same order as for the corresponding Galerkin approximation. Moreover, analogous conclusions hold for the rational approximation and non-stationary variance-control settings of Theorems~\ref{rat_approx_theorem} and~\ref{prop_fin}, respectively.
\end{corollary}
    
In Figure~\ref{eucd_exmp}, we show samples of Gaussian random fields generated using
the precision-based covariance approximation on a spatial domain containing
Switzerland. Following common practice in applications of the SPDE approach, we enlarge
the computational domain before sampling in order to reduce boundary effects inside
the region of interest; see, for example,
\citep{lindgren11,lindgren2015rinla,BK2020rational}. More precisely, we first
construct an extended polygonal domain containing Switzerland and generate a finite
element mesh on this domain; see Figure~\ref{eucd_exmp}, upper left. We then impose
Neumann boundary conditions on the boundary of the extended computational domain and
simulate the field for $\beta=0.5,0.8,1,1.5$, and $2$. Finally, the simulated fields
are restricted to the subdomain corresponding to Switzerland.

\begin{figure}[t]
	\begin{center}
		\includegraphics[width=1\linewidth]{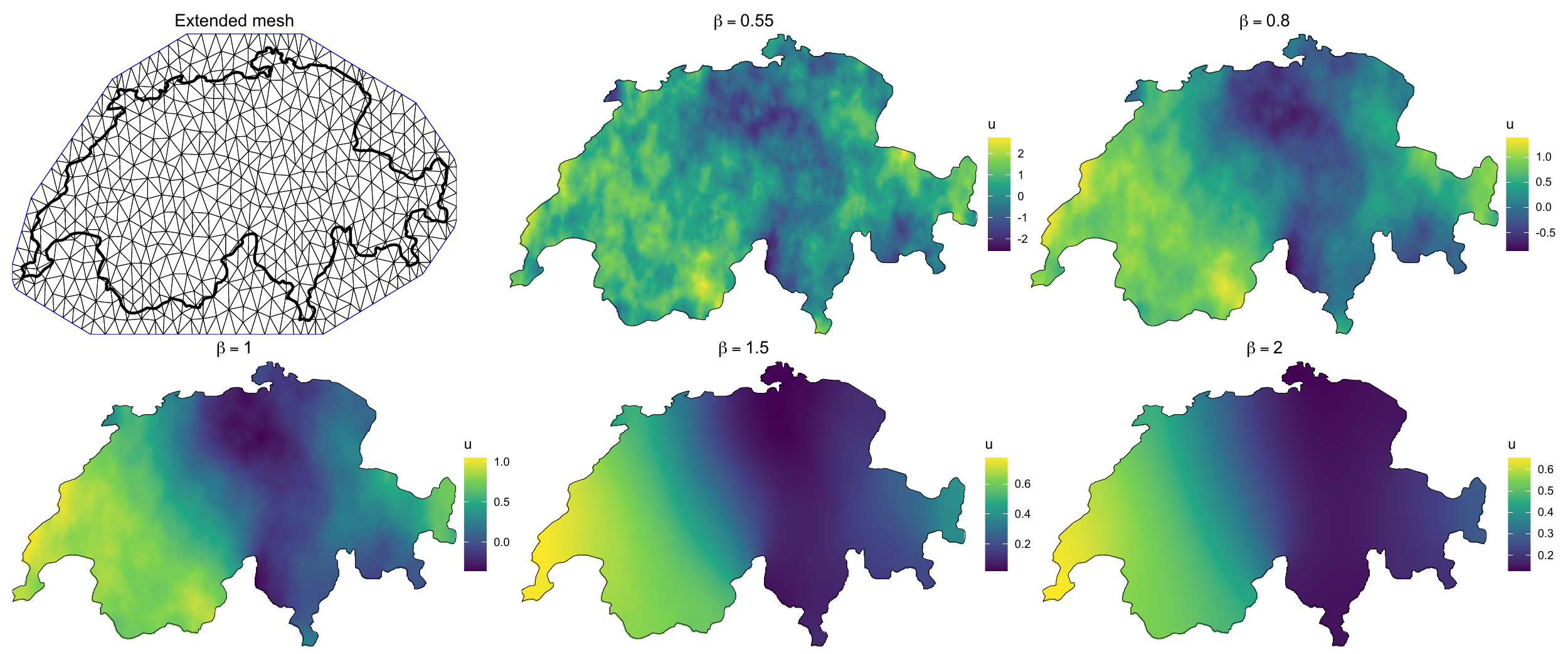}
	\end{center}
\caption{Extended computational domain and FEM mesh (top left), together with samples
of Gaussian random fields generated using the precision-based discretization with
$\boldsymbol{H}=I$, $\kappa=1$, and $\beta=0.5,0.8,1,1.5,2$. Neumann boundary
conditions are imposed on the boundary of the extended computational domain, and the
fields are then restricted to the subdomain corresponding to Switzerland. In all
simulations, the same realization of the Gaussian white noise is used.}
	\label{eucd_exmp}
\end{figure}

The simulations in Figure~\ref{eucd_exmp} were generated in \textsf{R}
\cite{R}. We used the \textsf{rnaturalearth} package \cite{rnaturalearth} to obtain
the Switzerland boundary and the \textsf{sf} package \cite{sf} for the corresponding
geometric operations. Mesh generation, construction of the extended domain, and
assembly of the finite element objects were performed using the \textsf{fmesher}
package \cite{fmesher}. The samples were generated using the \textsf{rSPDE} package \cite{rSPDEpackage}, which
implements the covariance-based rational approximation methodology of
\cite{bolin2024covariance}. We refer the reader to \cite{bolin2024covariance} for the
relevant matrix expressions for the precision operator, which are implemented in \textsf{rSPDE}.

\subsection{Surface domains}\label{sec:surface_domain} 

We now assume that $\cD$ is a compact, connected, orientable hypersurface in $\bbR^3$
without boundary. Let $g$ be the Riemannian metric on $\cD$ induced by the Euclidean
metric of $\bbR^3$. For a smooth function $f$ on $\cD$, the gradient
$\nabla_{\cD}f$ is the unique vector field satisfying
$\inner{\nabla_{\cD}f}{v}_{g}=v(f)$ for every smooth vector field $v$ on $\cD$.
We denote by $dV_g$ the Riemannian volume form associated with $g$. Since $g$ is induced
by the Euclidean embedding, the corresponding volume measure coincides with the
restriction of the $2$-dimensional Hausdorff measure $\cH^2$ to $\cD$. For a smooth
vector field $v$ on $\cD$, we denote by $\mathrm{div}_{\cD}v$ its Riemannian divergence,
which may be written as
$\mathrm{div}_{\cD}v=*^{-1}d(i_v dV_g)$, where $i_v$ denotes interior multiplication,
$d$ is the exterior derivative, and $*$ is the Hodge star operator; see, e.g.,
\cite{Lee2013Smooth}.

A function $f\in L^2(\cD)$ is said to have an $L^2(\cD)$ weak gradient if there exists
a vector field $\xi$ on $\cD$ with $|\xi|_g\in L^2(\cD)$ such that
$$
    \int_{\cD}\inner{\xi}{v}_g\,dV_g
    =
    -\int_{\cD} f\,\mathrm{div}_{\cD}v\,dV_g
$$
for every smooth vector field $v$ on $\cD$. In this case, $\xi$ is uniquely determined
and is denoted by $\nabla_{\cD}f$. With this notion of weak gradient, $H^1(\cD)$ is
precisely the space defined in \eqref{H1_def}. Since $\cD$ has no boundary, the spaces
$\dot H_L^1(\cD)$ and $\mathscr{D}(L)$ appearing in our assumptions are taken to be
$H^1(\cD)$ and $H^2(\cD)$, respectively. The spaces $W^{k,\infty}(\cD)$ are defined
analogously, replacing $L^2$-integrability by essential boundedness.

\begin{remark}
We emphasize that, while the above description is consistent with the Sobolev spaces $H^k(\mathcal{D})$ and $W^{1,\infty}(\mathcal{D})$ introduced earlier, the equivalent description adopted in \cite{dziuk2013finite}, which expresses differential structures and function spaces in terms of the ambient space $\mathbb{R}^3$, is more suitable in the finite element analysis. This approach facilitates the derivation of explicit and practically useful error estimates for families of surfaces $\mathcal{D}_h$ in the same ambient space that approximate $\mathcal{D}$.
\end{remark}

We assume that $\cD$ admits a tubular neighborhood $\cD_\delta$ of radius $\delta>0$.
Within this neighborhood, we consider a polygonal surface $\cD_h$, with vertices on
$\cD$, approximating $\cD$. Let $\eta(a)$ denote the unit normal vector to $\cD$ at
$a\in\cD$. For every $x\in\cD_\delta$, there exists a unique closest point
$a(x)\in\cD$ such that
$x=a(x)+d(x)\eta(a(x))$, where $d(x)$ is the signed distance to $\cD$. Restricting this
projection to $\cD_h$ defines a bijective map $\alpha:\cD_h\to\cD$.

The map $\alpha$ allows us to identify functions on $\cD_h$ with functions on $\cD$ by
means of the lift operator. If $F$ is defined on $\cD_h$, its lift to $\cD$ is
$F^\ell:=F\circ\alpha^{-1}$. Conversely, if $F$ is defined on $\cD$, we define its
inverse lift to $\cD_h$ by $F^{-\ell}:=F\circ\alpha$.

Since $\alpha$ restricted to each simplex provides a parametrization of its image on
$\cD$, the bilinear form $a_L$ in \eqref{bilin_form_L_exmp} can be written as
\begin{equation}\label{bilin_form_L_surf}
    \begin{split}
        a_L(f,g)
        &=\int_{\cD}
        \left(\boldsymbol{H}\nabla_{\cD}f\cdot\nabla_{\cD}g+\kappa^2fg\right)\,dA \\
        &=\sum_{T\in\cT_{\cD_h}}
        \int_T
        (D\alpha)^{-1}\boldsymbol{H}^{-\ell}(D\alpha)^{-\top}
        \nabla_{\cD_h}f^{-\ell}\cdot\nabla_{\cD_h}g^{-\ell}\,
        \delta_h\,dA_h \\
        &\quad+
        \sum_{T\in\cT_{\cD_h}}
        \int_T
        (\kappa^{-\ell})^2 f^{-\ell}g^{-\ell}\,\delta_h\,dA_h,
        \qquad f,g\in H^1(\cD),
    \end{split}
\end{equation}
where $D\alpha$ denotes the Jacobian of $\alpha$, $\delta_h:=|\det(D\alpha)|$, and
$dA$ and $dA_h$ denote the surface measures on $\cD$ and $\cD_h$, respectively.

The coefficients appearing in the pullback representation \eqref{bilin_form_L_surf}
satisfy the same type of boundedness and ellipticity conditions as those imposed in
Assumption~\ref{ass_coeff}, provided $h$ is sufficiently small. However, in the
classical finite element method on closed surfaces, one usually avoids the full
pullback form \eqref{bilin_form_L_surf}. Instead, one works directly on the polygonal
surface $\cD_h$ with the simpler bilinear form
\begin{equation}\label{bilin_form_L_surf_2}
    a_{L,h}(\phi_h,\psi_h)
    =
    \int_{\cD_h}
    \boldsymbol{H}^{-\ell}\nabla_{\cD_h}\phi_h\cdot\nabla_{\cD_h}\psi_h\,dA_h
    +
    \int_{\cD_h}
    (\kappa^{-\ell})^2\phi_h\psi_h\,dA_h,
\end{equation}
for $\phi_h,\psi_h\in V_h$.
After lifting, the corresponding discrete structures on $V_h^\ell$ approximate the
structures defined by $a_L$ on $\cD$; this is the standard surface finite element
viewpoint, see, e.g., \cite{dziuk2013finite}.

We now introduce the lumped-mass version of \eqref{bilin_form_L_surf_2}. The quadrature
rule defining $\langle\cdot,\cdot\rangle_h$ in \eqref{quadrature}, now applied on the
triangulated surface $\cD_h$ with the measure $dA_h$, gives an admissible inner product
on $V_h$ in the sense of Definition~\ref{def_admissible_inner_product}. Using this
quadrature rule in the two integrals in \eqref{bilin_form_L_surf_2}, we define
\begin{equation}\label{bilin_form_h_surf}
    a_h(\phi_h,\psi_h)
    =
    \bigl\langle
    \boldsymbol{H}^{-\ell}\nabla_{\cD_h}\phi_h,
    \nabla_{\cD_h}\psi_h
    \bigr\rangle_h
    +
    \bigl\langle
    \kappa^{-\ell}\phi_h,
    \kappa^{-\ell}\psi_h
    \bigr\rangle_h,
    \qquad \phi_h,\psi_h\in V_h.
\end{equation}
The family $\{a_h\}_{h\in(0,1)}$ is therefore an admissible bilinear form in the sense
of Definition~\ref{def_admissible_bilinear_form}. This follows from the quadrature
error estimate in Lemma~\ref{quad_err_est}, combined with the standard geometric
estimates for surface finite elements; see, e.g., \cite{dziuk2013finite}. Hence, the
abstract results of Section~\ref{sec:fem} apply to the lifted approximations on
$\cD$.

To describe the precision-based covariance approximation on surface domains, let
$\mathcal{L}_h:L^2(\cD_h)\to L^2(\cD)$ denote the lifting operator,
$\mathcal{L}_hF:=F^\ell$, and let
$\mathcal{L}_h^*:L^2(\cD)\to L^2(\cD_h)$ be its adjoint. The Galerkin-based covariance
kernel on $\cD$ is then defined as the kernel $\varrho_h^\beta$ associated with the
lifted operator $\mathcal{L}_h L_h^{-2\beta}\Lambda_h\mathcal{L}_h^*$. That is,
\begin{equation*}
    \left(\mathcal{L}_h L_h^{-2\beta}\Lambda_h\mathcal{L}_h^* f\right)(x)
    =
    \int_{\cD}\varrho_h^\beta(x,y)f(y)\,dA(y),
    \qquad f\in L^2(\cD),
\end{equation*}
with the equality understood in $L^2(\cD)$, or equivalently for $dA$-a.e. $x\in\cD$.

\begin{remark}\label{remark_cov_surf}
Let $\rho_h^\beta$ denote the covariance kernel on $\cD_h$ associated with
$L_h^{-2\beta}\Lambda_h$. Then the corresponding covariance kernel on the original
surface $\cD$ is obtained by lifting both variables:
$$
    \varrho_h^\beta(x,y)
    =
    \rho_h^\beta(\alpha^{-1}(x),\alpha^{-1}(y)),
    \qquad x,y\in\cD.
$$
Thus, in practice, one first computes the covariance kernel associated with
$L_h^{-2\beta}\Lambda_h$ on the polygonal surface $\cD_h\times\cD_h$, and then lifts it
to $\cD\times\cD$. This mirrors the standard strategy in surface finite element
methods: the discrete problem is formulated on the simpler approximating surface
$\cD_h$, and the resulting discrete objects are transported back to the original
surface through the lift map. The presence of the lifting operator $\mathcal{L}_h$ and
its adjoint $\mathcal{L}_h^*$ in the operator formulation is what makes this
covariance-level lifting precise.
\end{remark}

Analogously, the Galerkin-based covariance kernel is defined as the kernel
$\widehat{\varrho}_h^\beta$ associated with the lifted operator
$\mathcal{L}_h\widehat L_h^{-2\beta}\Pi_h\mathcal{L}_h^*$, where $\widehat L_h$ is
defined by \eqref{galerkinLdiscretization} with $a_{L,h}$ in place of $a_L$. That is,
$\widehat{\varrho}_h^\beta$ is the covariance kernel obtained from the standard
surface finite element discretization on $\cD_h$ and then lifted to $\cD$.

The proof of Theorem~\ref{main_result_1}, together with the same argument used to
estimate the term $\cI_3$ in the proof of Theorem~\ref{prop_fin}, shows that the
precision-based covariance approximation and the Galerkin-based covariance
approximation differ by an error of order $\min\{4\beta-1,2\}$. More precisely, for
every $\eta<\min\{4\beta-1,2\}$,
\begin{equation}\label{cov_surf_conv_rate_inft}
    \|\widehat{\varrho}_h^\beta-\varrho_h^\beta\|_{L^2(\cD\times\cD)}
    \lesssim h^\eta
\end{equation}
for all sufficiently small $h$.

It remains to quantify the convergence of the lifted Galerkin covariance kernel
$\widehat{\varrho}_h^\beta$ to the exact covariance kernel $\varrho^\beta$ on the
surface $\cD$. Several works establish convergence results for Galerkin approximations
of fractional SPDEs on surfaces driven by Gaussian white noise; see, for example,
\citep{bonito2024numerical,doi:10.1137/21M1400717,herrmann2020multilevel}. However,
these works do not state the covariance-kernel estimate in the form needed here. Because the convergence condition required for the covariance in the surface case is rather lengthy and lies beyond the main scope of the present work, we omit its full derivation here. This condition will be addressed in the forthcoming work of the authors. Nevertheless, for completeness, we state the result below.


\begin{theorem}[Surface Galerkin covariance estimate]\label{thm_surf_galerkin_cov}
Let $L=-\nabla_{\cD}\cdot(\boldsymbol H\nabla_{\cD})+\kappa^2$, and let
$\varrho^\beta$ be the covariance kernel associated with $L^{-2\beta}$ on a surface
domain $\cD$ satisfying condition (D2) in Example~\ref{example_domain}. Let
$\widehat{\varrho}_h^\beta$ be the lifted Galerkin covariance kernel associated with
$\mathcal{L}_h\widehat L_h^{-2\beta}\Pi_h\mathcal{L}_h^*$, where $\widehat L_h$ is
defined by \eqref{galerkinLdiscretization} with $a_{L,h}$ in place of $a_L$. Then, for
$\beta>1/2$ and every $\eta<\min\{4\beta-1,2\}$,
\begin{equation}\label{cov_surf_conv_rate}
    \|\varrho^\beta-\widehat{\varrho}_h^\beta\|_{L^2(\cD\times\cD)}
    \lesssim h^\eta
\end{equation}
for all sufficiently small $h$.
\end{theorem}

Combining Theorem~\ref{thm_surf_galerkin_cov} with the estimate \eqref{cov_surf_conv_rate_inft} shows that the precision-based covariance approximation on surfaces preserves the corresponding Galerkin convergence rate. We record this as the following corollary

\begin{corollary}\label{cor2}Under the assumptions of Theorem~\ref{thm_surf_galerkin_cov}, for every $\eta<\min\{4\beta-1,2\}$ and all sufficiently small $h$, \begin{equation}\label{cov_surf_conv_rate_inft_2}\|\varrho^\beta-\varrho_h^\beta\|_{L^2(\cD\times\cD)}\lesssim h^\eta.\end{equation}
Furthermore, similar conclusions hold for the rational approximation and non-stationary variance-control settings of Theorems~\ref{rat_approx_theorem} and~\ref{prop_fin}, respectively.
\end{corollary}

Finally, in the surface case, the multiplier space $\cR^k$ appearing in
Assumption~\ref{assump2} can again be taken as $W^{k,\infty}(\cD)$. The verification of
the interpolation and stability properties follows analogously to the Euclidean case:
after pulling back to each flat element $T\in\cT_{\cD_h}$, the estimates reduce to the
standard local finite element interpolation estimates and inverse inequalities, with
additional geometric factors controlled by the surface approximation. These geometric
factors are uniformly bounded for sufficiently small $h$ by the standard estimates for
surface finite elements; see, e.g., \cite{dziuk2013finite}. Thus the abstract
assumptions involving $\cR^k$ are satisfied in this setting as well.

Figure~\ref{surf_exmp} illustrates realizations of Gaussian random fields with
precision-based covariance on several approximated surface domains satisfying condition
(D2) in Example~\ref{example_domain}. The numerical implementation uses the
\textsf{FEniCSx} finite element framework, through the \texttt{dolfinx} library
\cite{baratta2023dolfinx}. Mesh generation was performed with \texttt{gmsh}
\cite{geuzaine2009gmsh}, while the variational formulations were specified using
\texttt{ufl} \cite{alnaes2014ufl}. The resulting linear algebra operations were carried
out using \texttt{petsc4py} \cite{dalcin2011petsc4py}.

\begin{figure}[t]
	\begin{center}
		\includegraphics[width=1\linewidth]{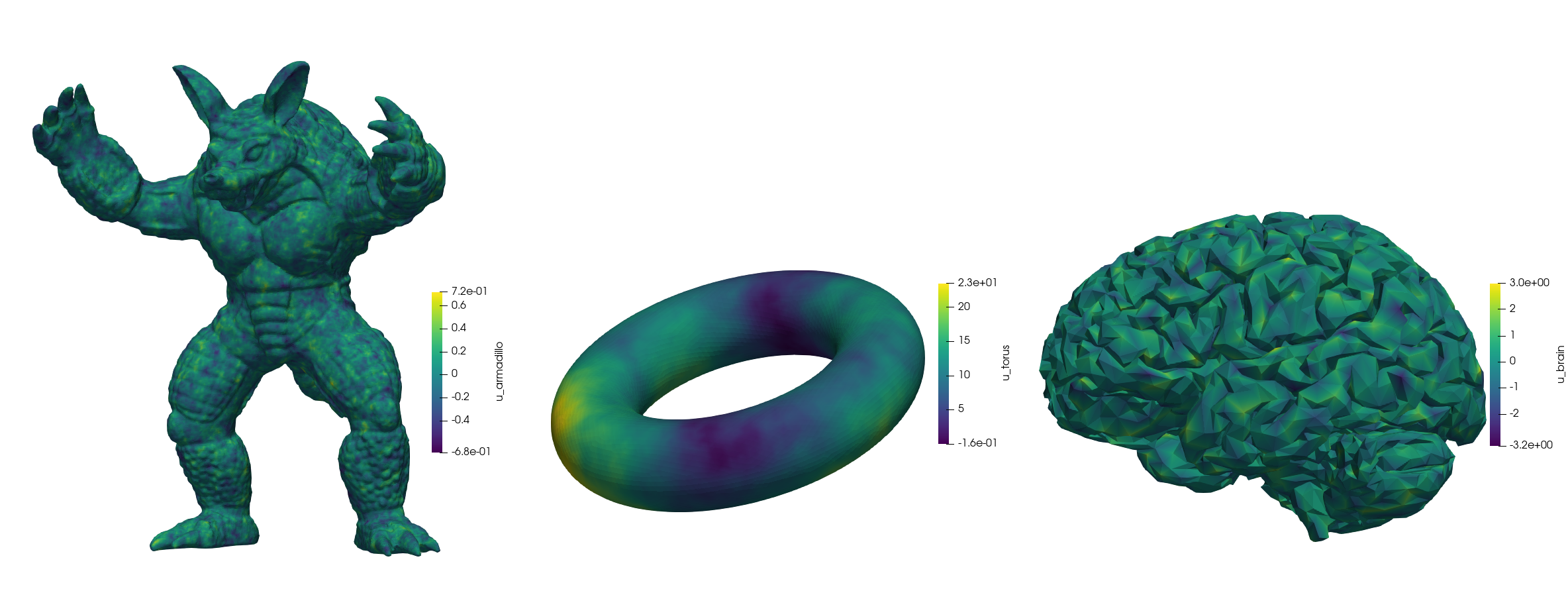}
	\end{center}
\caption{Samples of Gaussian random field approximations on different approximated
surface domains, with $\boldsymbol{H}=I$, $\kappa=1$, and $\beta=1$. The surface
geometries are the Stanford armadillo model (left), sourced from the Stanford 3D
Scanning Repository at \url{https://graphics.stanford.edu/data/3Dscanrep/}; a torus
(middle); and a remeshed brain surface (right), obtained from Sketchfab at
\url{https://sketchfab.com/3d-models/brain-cadd2bde67404c43b2359a6a3281d84a}.}
	\label{surf_exmp}
\end{figure}

\subsection{Metric graph domains}\label{sec:metric_graph_domains}

An important class of domains covered by our framework is given
by metric graphs. A metric graph $\cD=(\cE,\cV)$ consists of a set of edges $\cE$
connected at vertices $\cV$. Each edge $e\in\cE$ is not merely an abstract connection
between two vertices, but a curve parametrized by arc length. The graph is endowed with
the induced shortest-path metric, so that distances can be measured between arbitrary
locations $s_1,s_2\in\cD$, whether they are vertices or points in the interior of
edges; see, e.g., \cite{berkolaiko2013introduction} for a background of metric graphs.

There have been significant recent developments on Gaussian random fields and
fractional SPDEs on metric graphs. In particular, Whittle--Mat\'ern fields on compact
metric graphs were introduced in \cite{bolin2024gaussian}, where existence, sample path
regularity, Karhunen--Lo\`eve expansions, and invariance under insertion or removal of
degree-two vertices are established. Regularity and finite element approximation of
fractional elliptic equations on compact metric graphs are studied in
\cite{bolin2024regularity}, while Markov properties of Gaussian fields defined
through fractional-order differential equations on compact metric graphs are analyzed
in \cite{bolin2026markov}. Related covariance-based approaches on graphs and their
edges are considered in \cite{anderes2020isotropic}.
Gaussian random fields on metric graph provide a natural way to model data observed
on complex networks, such as street or river networks. Statistical inference for
Gaussian Whittle--Mat\'ern fields on metric graphs is developed in
\cite{bolin2026statistical}, including likelihood-based inference, prediction, and
asymptotic properties of maximum likelihood estimators and linear predictors. For an
application to point patterns on road networks, see \cite{bolin2025log}, where such
fields are used to identify high-risk areas on road networks.

For a metric graph $\cD=(\cE,\cV)$, we identify each edge $e\in\cE$ with an interval
$[0,\ell_e]$ through an arc-length parametrization. We fix an arbitrary orientation on
each edge and denote its initial and terminal vertices by $e_-$ and $e_+$, respectively.
Thus, a function $f:\cD\to\bbR$ can be written edgewise as
$f=(f_e)_{e\in\cE}$, where $f_e$ denotes the restriction of $f$ to the edge $e$.
We define $C_c^1(\cD):=\bigoplus_{e\in\cE}C_c^1(e)$, where $C_c^1(e)$ denotes the
space of continuously differentiable functions compactly supported in the interior of
the edge $e$. A function $f=(f_e)_{e\in\cE}\in L^1(\cD)$ is said to have weak derivative
$g=(g_e)_{e\in\cE}\in L^1(\cD)$ if
$$
\sum_{e\in\cE}\int_e f_e\varphi_e'\,d\lambda_e
=
-\sum_{e\in\cE}\int_e g_e\varphi_e\,d\lambda_e
$$
for every $\varphi=(\varphi_e)_{e\in\cE}\in C_c^1(\cD)$. In this case, we write
$\nabla_{\cD}f:=f':=(f_e')_{e\in\cE}:=g$.

For $k\in\mathbb{N}$, the decoupled Sobolev space on the metric graph is defined by
$$
\widetilde H^k(\cD)
:=
\left\{
f=(f_e)_{e\in\cE}\in L^2(\cD):
f_e\in H^k(e)\ \text{for every } e\in\cE
\right\},
$$
equipped with the natural norm
$\|f\|_{\widetilde H^k(\cD)}^2:=\sum_{e\in\cE}\|f_e\|_{H^k(e)}^2$. Thus,
$\widetilde H^k(\cD)$ is decoupled across edges and does not, by itself, impose
continuity or vertex conditions.

For a vertex $v\in\cV$, let
$
\cI(v):=\{(e,\sigma)\in\cE\times\{-,+\}: e_\sigma=v\}.
$
If $e$ is a loop based at $v$, then both $(e,-)$ and $(e,+)$ belong to $\cI(v)$. For
$u\in\widetilde H^2(\cD)$, we define the outward derivative of $u$ at the endpoint
$(e,\sigma)\in\cI(v)$ by
$$
\partial_{\nu_{e,\sigma}}u(v):=
\begin{cases}
-u_e'(0), & \sigma=-,\\
\phantom{-}u_e'(\ell_e), & \sigma=+.
\end{cases}
$$

In the metric-graph setting, we have that
$\dot H_L^1(\cD):=H^{1}(\cD):=\widetilde H^1(\cD)\cap C(\cD)$, where $C(\cD)$ denotes the space of
functions that are continuous on each edge and whose edge traces agree at every vertex.
The operator domain is
$$
\mathscr{D}(L)
=
\left\{
u\in\widetilde H^2(\cD)\cap C(\cD):
\sum_{(e,\sigma)\in\cI(v)}
\partial_{\nu_{e,\sigma}}u(v)=0
\ \text{for every } v\in\cV
\right\}.
$$
The vertex conditions above are called Kirchhoff, or Kirchhoff--Neumann, conditions.
With this notation, the bilinear form associated with $L$ is
$$
a_L(u,v)
=
\sum_{e\in\cE}(H_e u_e',v_e')_{L^2(e)}
+
(\kappa u,\kappa v)_{L^2(\cD)},
\qquad u,v\in\dot H_L^1(\cD),
$$
where $H_e$ denotes the restriction of the coefficient $\boldsymbol H$ to the edge
$e$. The fractional Sobolev spaces induced by $L$ can also be characterized in terms of
interpolation spaces on metric graphs. These identifications are important for
regularity theory and for the analysis of fractional elliptic equations on graphs; we
refer to \cite{bolin2024regularity} for the precise statements.
Turning to the precision-based discretization, let $V_h\subset\dot H_L^1(\cD)$ be the
conforming piecewise linear finite element space on the metric graph. The mesh is
defined edgewise, and derivatives are taken with respect to the arc-length coordinate
on each edge. The lumped-mass quadrature rule introduced in \eqref{quadrature} applies
edge by edge and defines an admissible inner product on $V_h$ in the sense of
Definition~\ref{def_admissible_inner_product}. Using the same quadrature rule in the
bilinear form associated with $L$, we define
$$
a_h(u_h,v_h)
:=
\langle H u_h',v_h'\rangle_h
+
\langle \kappa u_h,\kappa v_h\rangle_h,
\qquad u_h,v_h\in V_h.
$$
Equivalently, this means that the integrals defining $a_L$ are approximated by the
edgewise lumped quadrature rule. Since the estimates reduce on each edge to the
standard one-dimensional finite element quadrature estimates, $a_h$ is an admissible bilinear form in the sense of
Definition~\ref{def_admissible_bilinear_form} for each $h\in(0,1)$.

In this setting, the multiplier space $\cR^k$ in Assumption~\ref{assump2} can be taken
as $W^{k,\infty}(\cD)$ for $k\leq2$. More precisely,
$W^{k,\infty}(\cD)$ is understood edgewise, with the corresponding continuity
requirements at the vertices whenever needed. The verification of the interpolation and
stability estimates in Assumption~\ref{assump2} is obtained by reducing the estimates
to the Euclidean one-dimensional case on each edge. Since the graph has finitely many
edges, the edgewise constants can be combined to obtain constants uniform over $\cD$.


For compact metric graphs, \cite[Theorem~6.9]{bolin2024regularity} establishes the Galerkin covariance convergence rate $\eta<\min\{4\beta-\frac{1}{2},2\}$. Since these graphs have spectral dimension $d=1$, this agrees with the rate $\min\{4\beta-d/2,2\}$ in Theorem~\ref{main_result_1}, which shows that the precision-based discretization preserves this rate. We thus obtain the following consequence.

\begin{corollary}\label{cor3}
Under the assumptions of Theorem~\ref{main_result_1}, for every $\eta<\min\{4\beta-\frac{1}{2},2\}$ and all sufficiently small $h$, $$\|\varrho^\beta-\varrho_h^\beta\|_{L^2(\cD\times\cD)}\lesssim h^\eta.$$In particular, the precision-based discretization preserves the optimal Galerkin covariance convergence rate. In addition, we note that similar conclusions hold for the rational approximation and non-stationary variance-control settings of Theorems~\ref{rat_approx_theorem} and~\ref{prop_fin}, respectively.
\end{corollary}

Figure~\ref{graph_exmp_new_cairo} shows a sample of a Gaussian random field with
precision-based covariance on a metric graph constructed from the road network of New
Cairo, Egypt, for $\beta=0.44$. The graph was built from OpenStreetMap data
\cite{OpenStreetMap}, extracted using the browser-based tool \textsf{Overpass turbo}
\cite{overpassTurbo}. We retained the largest connected component of the subnetwork
formed by primary, secondary, and tertiary highways.
The metric graph construction, mesh generation, finite element matrices, and data
manipulation were carried out using the \textsf{MetricGraph} package
\cite{MetricGraphPackage}. The rational approximation structures and sampling of the
field were computed using \textsf{rSPDE}.

\begin{figure}[t]
	\begin{center}
		\includegraphics[width=1\linewidth]{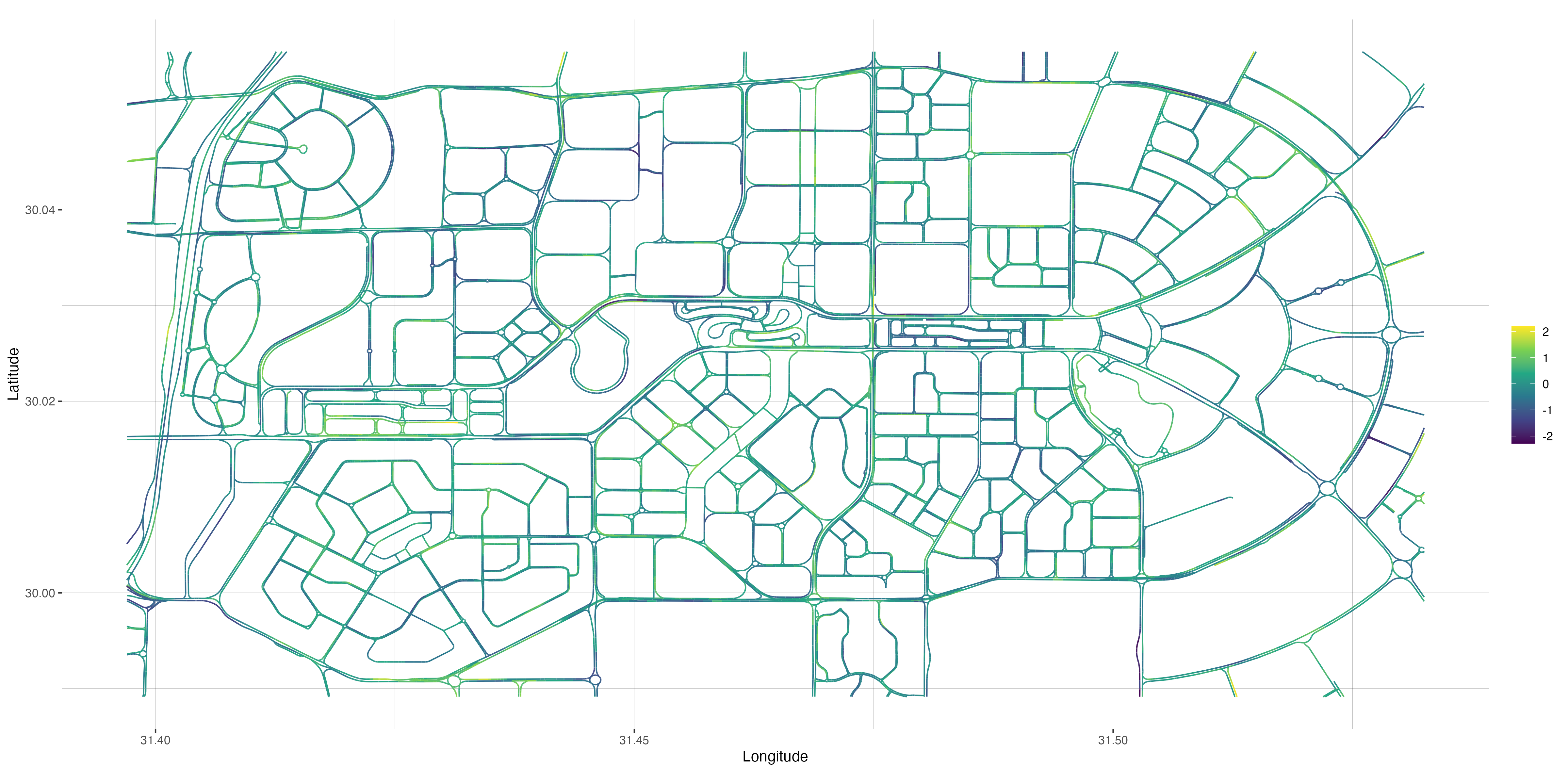}
	\end{center}
\caption{Sample of a Gaussian random field with precision-based covariance on the
largest connected component of the primary, secondary, and tertiary highway network of
New Cairo, Egypt. The parameters are $\boldsymbol{H}=I$, $\kappa=1$, and $\beta=0.44$.}
	\label{graph_exmp_new_cairo}
\end{figure}

\section{Collected proofs}\label{sec_proof_results}

In this section, we prove the results stated in Section~\ref{sec:fem}. We begin with a
spectral comparison result that will be used repeatedly in the sequel.

\begin{proposition}\label{eigenval}
Let $(\lambda_i)_{i\in\mathbb{N}}$ be the eigenvalues of $L$,
enumerated in non-decreasing order and counted with multiplicities. For each
$h\in(0,1)$, let $(\widehat\lambda_{i,h})_{i=1}^{N_h}$ and
$(\lambda_{i,h})_{i=1}^{N_h}$ denote the eigenvalues of the Galerkin operator
$\widehat L_h$ defined in \eqref{galerkinLdiscretization} and of the precision-based
operator $L_h$ defined in \eqref{a_lump}, respectively, again enumerated in non-decreasing order and counted with
multiplicities. Then:
\begin{enumerate}[label=\arabic*)]
    \item\label{eigenvalues_equiv}
    $\widehat{\lambda}_{i,h}\lesssim \lambda_{i,h}\lesssim \widehat{\lambda}_{i,h}$,
    for $1\leq i\leq N_h$;

    \item\label{eigenlower}
    $\lambda_i\leq \widehat{\lambda}_{i,h}$,
    for $1\leq i\leq N_h$;

    \item\label{upperbound_nh}
    $\max\{\widehat{\lambda}_{i,h},\lambda_{i,h}\}\lesssim h^{-2}$,
    for $1\leq i\leq N_h$.
\end{enumerate}
\end{proposition}

\begin{proof}
	The results are straightforward consequences of the Courant min-max principle. Indeed, let $R(u):=\nicefrac{a_{L}(u,u)}{\|u\|_{L^2(\cD)}^2}$, $\widehat{R}_{h}(\varphi):=\nicefrac{a_{L}|_{V_{h}\times V_{h}}(\varphi,\varphi)}{\|\varphi\|_{L^2(\cD)}^2}$ and $R_{h}(\phi):=\nicefrac{a_{h}(\phi,\phi)}{\|\phi\|_{h}^2}$ the Rayleigh quotients of $a_{L}$, $a_{L}|_{V_{h}\times V_{h}}$ and $a_{h}$ over $L^2(\cD)$, $V_{h}$ and $V_{h}$, respectively. Then, it follows from Remark \ref{remark_ord2_bilin_form} that
	\begin{equation}\label{rayleigh_quotient_equiv}
		\widehat{R}_{h}(u) \lesssim R_{h}(u) \lesssim \widehat{R}_{h}(u),\quad (u \in V_h).
	\end{equation}
	Therefore, the Courant min-max principle combined with \eqref{rayleigh_quotient_equiv} implies \ref{eigenvalues_equiv}. Moreover, any subspace of dimension $k$ of $V_h$ is also a subspace of dimension $k$ of $\dot{H}^{1}_{L}(\cD)$. Consequently, from the min-max principle, we have that
	\begin{equation*}
		\lambda_{k}=\min _{V_k \subset H^1_{L}(\cD)} \max _{u \in V_{k}} R(u)\le \min _{V_k \subset V_{h}} \max _{u \in V_{k}} \widehat{R}_{h}(u)=\widehat{\lambda}_{h,k}
	\end{equation*}
	and 2) follows. If $\phi_{k}\in V_{h}$ is the eigenfunction of $\widehat{L}_h$ associated with the eigenvalue $\widehat{\lambda}_{k,h}$ such that $\|\phi_{k}\|_{1}=1$, then $\widehat{\lambda}_{k,h}=a_{L}(\phi_{k},\phi_{k})=\|\phi_{k}\|^{2}_{1}\lesssim h^{-2}$, where the last inequality holds from item (4) in Assumption \ref{assump2}. Now \ref{upperbound_nh} follows by combining the previous inequality with the result of \ref{eigenvalues_equiv}.
\end{proof}

The proof of Theorem~\ref{main_result_1} relies on Proposition~\ref{eigenval} and a fundamental integral representation of fractional powers of self-adjoint bounded operators with spectrum bounded away from zero. The lines of our proof follows similar ideas as those presented in \cite[Theorem~4.2]{almeida-sousa2026finite}, with the main difference being the intermediation of the Hilbert-Schmidt norm, by the Schatten norm, which for a bounded linear operator $\Theta:E\to F$, defined over hilbert spaces $E$ and $F$, is defined for $0\le p < \infty$ as the norm 
$\|\Theta\|_{L_p(E,F)}:=\left(\sum_{i=1}^{\infty} s(\Theta)_i^p\right)^{1/p}$, where $(s(\Theta)_i)_{i\in\mathbb{N}}$ are the singular values of $|\Theta|:=\sqrt{\Theta^*\Theta}$ counted with multiplicities. 

To keep the notation brief, we from now on write $\norm{\cdot}{\LD}$ for both $\norm{\cdot}{L^{2}(\cD)}$ and $\norm{\cdot}{L^{2}(\cD\times\cD)}$, $\norm{\cdot}{\HS}$ for $\norm{\cdot}{L_2(L^2(\cD))}$, and $\norm{\cdot}{\LL}$ for $\norm{\cdot}{L(L^2(\cD))}$.

\begin{proof}[Proof of Theorem~\ref{main_result_1}] Let $\widetilde{\varrho}_h^\beta$ denote the kernel associated with 
    $\widetilde L_h^{-2\beta}\Pi_h$. By the triangle inequality,
    \begin{equation}\label{rt_app_s2}
        \begin{split}
        \left\|\varrho^\beta-\varrho_h^\beta\right\|_{\LD}
        &\lesssim
        \left\|\varrho^\beta-\widehat{\varrho}_h^\beta\right\|_{\LD}
        +
        \left\|\widehat{\varrho}_h^\beta-\widetilde{\varrho}_h^\beta\right\|_{\LD}
        +
        \left\|\widetilde{\varrho}_h^\beta-\varrho_h^\beta\right\|_{\LD}.
        \end{split}
\end{equation}
It remains to estimate the last two terms on the right-hand side of
\eqref{rt_app_s2}. We focus our attention on estimating
\begin{equation*}
\cE_h
:=
\left\|\widetilde{\varrho}_h^\beta-\varrho_h^\beta\right\|_{\LD}
=
\left\|\widetilde L_h^{-2\beta}\Pi_h-L_h^{-2\beta}\Lambda_h
\right\|_{L_{2}(L^{2}(\cD), V_{h})}
\end{equation*}
and note that the estimate for 
$\left\|\widehat{\varrho}_h^\beta-\widetilde{\varrho}_h^\beta\right\|_{\LD}$ follows analogously and in a simpler way since the projections are the same. We split the proof into the cases $\alpha>\frac{1}{2}$ and $\alpha\le\frac{1}{2}$.

\textit{Case $\alpha>\frac{1}{2}$:}  From \cite[Theorem~4.3]{almeida-sousa2026finite}, the identity
\begin{equation}\label{knjnroaa_fp}
    \begin{split}
    A^{-\gamma}
    &=
    \frac{\sin(\pi\gamma)}{\pi}
    \int_r^R t^{-\gamma}(t+A)^{-1}\,dt +
    \frac{R^{1-\gamma}}{2\pi}
    \int_{-\pi}^{\pi}
    e^{i(1-\gamma)\theta}(Re^{i\theta}-A)^{-1}\,d\theta\\
	&\quad-
    \frac{r^{1-\gamma}}{2\pi}
    \int_{-\pi}^{\pi}
    e^{i(1-\gamma)\theta}(re^{i\theta}-A)^{-1}\,d\theta
    \end{split}
\end{equation}
is valid for $A=\widetilde L_h$ and $A=L_h$, provided that
$r<\min\{\widetilde\lambda_{1,h}/2,\lambda_{1,h}/2\}$ and
$R>\max\{\widetilde\lambda_{N_h,h},\lambda_{N_h,h}\}$. In particular, by
Proposition~\ref{eigenval}, we may choose $R=(C+1)h^{-2}$, with $C$ independent of
$h$, so that
\begin{equation}\label{spec_relations}
    r<\min\{\widetilde\lambda_{1,h}/2,\lambda_{1,h}/2\}
\leq
\max\{\widetilde\lambda_{N_h,h},\lambda_{N_h,h}\}
< Ch^{-2}<R.
\end{equation}
Here $(\widetilde\lambda_{i,h})_{i=1}^{N_h}$ and
$(\lambda_{i,h})_{i=1}^{N_h}$ denote the eigenvalues of $\widetilde L_h$ and $L_h$,
respectively, enumerated in non-decreasing order, and
$N_h=\operatorname{dim}(V_h)$. Since $\|S\|_{L_2(V_{h})} \leq \sqrt{N_h}\|S\|_{L(V_{h})}$, the operators in the integrands are strongly measurable and bounded over the intervals under consideration. Consequently, \eqref{knjnroaa_fp} holds inside the space of $L_{2}(V_{h})$-valued Bochner integrable functions. From these observations, we get that
\begin{equation}\label{principal_relation_1}
    \begin{split}
    \cE_{h} &\lesssim \int_r^R t^{-2\beta}\left\|I_1(t)\right\|_{L_{2}(L^{2}(\cD), V_{h})} \,d t  +R^{1-2\beta} \int_{-\pi}^\pi\left\|I\left(R e^{i \theta}\right)\right\|_{L_{2}(L^{2}(\cD), V_{h})} \,d \theta \\
    &\quad +r^{1-2\beta} \int_{-\pi}^\pi\left\|I\left(r e^{i \theta}\right)\right\|_{L_{2}(L^{2}(\cD), V_{h})} \,d \theta
\end{split}
\end{equation}
where we set $I_{1}(t):=(t+\widetilde{L}_{h})^{-1} \Pi_h-(t+L_{h})^{-1} \Lambda_h$ for $t>0$, and $I(z):=(z-\widetilde{L}_{h})^{-1} \Pi_h-(z-L_{h})^{-1} \Lambda_h$ for $z\in \mathbb{C}\setminus \bigl(\sigma(\widetilde L_h)\cup \sigma(L_h)\bigr)$. Using $\widetilde{L}_{h}^{-1}\Pi_{h}=L_{h}^{-1}\Lambda_{h}$ and that an operator always commutes with its resolvent under the Borel functional calculus, we get that
\begin{equation*}
    \begin{split}
    I_1(t)&=-t (t+\widetilde{L}_{h})^{-1} \widetilde{L}_{h}\left(\widetilde{L}_{h}^{-1}-L_{h}^{-1}\right) L_{h}(t+L_h)^{-1}L^{-1}_{h}\Lambda_h \\
    & = -t (t+\widetilde{L}_{h})^{-1}\widetilde{L}_{h}^{1/2} G_{h} L_{h}^{1/2}(t+L_h)^{-1} \Lambda_h
    \end{split}
\end{equation*}
where we set $G_{h}=\widetilde{L}_{h}^{1/2}\left(\widetilde{L}_{h}^{-1}-L_{h}^{-1}\right) L_{h}^{-1/2}$. Now, observing that $1/2=1/4+0+1/4+0$, we can use the Schatten--Hölder inequality; see, for instance, \cite[Theorem~2.7 and Theorem~2.8]{Simon_2010}, to get that 
\begin{equation}\label{schatten_holder_ineq}
    \begin{split}
    \left\|I_1(t)\right\|_{L_2\left(L^{2}(\cD), V_h\right)}  &\leq t\left\|\widetilde{L}_{h}^{1/2}(t+\widetilde{L}_{h})^{-1}\right\|_{L_{4}} \left\|G_{h}\right\|_{L}\times \\
    &\quad\times  \left\|L_{h}^{1/2}(t+L_h)^{-1}\right\|_{L_{4}} \left\|\Lambda_h\right\|_{L\left(L^{2}(\cD), V_h\right)}.
    \end{split}
\end{equation}
Reasoning as in \cite[Theorem~4.2]{almeida-sousa2026finite}, it follows that $\|G_{h}\|_{L}\lesssim h^{2}$ and $\|\Lambda_{h}\|_{L(L^{2}(\cD), V_h)}\lesssim 1$. In order to estimate the $L_{4}$-norms in $V_{h}$, we note that for $Q\in\{\widetilde{L}_h, L_h\}$,
\begin{equation}\label{L4_norm_estimate}
    \| Q^{1/2}(t+Q)^{-1} \|_{L_4}^4=\sum_{j=1}^{N_h} \frac{\lambda_{Q, j}^2}{\left(t+\lambda_{Q, j}\right)^4} \lesssim t^{1 / \alpha-2}
\end{equation}
for $t\in [r,R]$ and $\alpha>1/2$. In fact, if we let $J_{t}:=\floor{K\max\{1,t^{\frac{1}{\alpha}}\}}$, where $K>0$ is a large enough constant such that $\lambda_{Q,j}\geq j^{\alpha}>t$, for every $j>J_{t}$; note that this last part depends only on the constants in \eqref{weyl_law}. We get from Proposition~\ref{eigenval} and the fact that $x^{2}/(t+x)^{4}\le (1/16) t^{-2}$, for every $t,x>0$, that
\begin{equation*}
    \begin{split}
        \sum_{j=1}^{N_h} \frac{\lambda_{Q, j}^2}{\left(t+\lambda_{Q, j}\right)^4}
         \lesssim \sum_{j\le J_{t}} t^{-2} + \sum_{j>J_{t}} \lambda_{Q, j}^{-2} \lesssim J_{t}t^{-2} + \sum_{j>J_{t}} j^{-2\alpha} \lesssim J_{t}t^{-2} + J_{t}^{1-2\alpha}
    \end{split}
\end{equation*}
where in the last step we used that $\alpha>1/2$ and \eqref{weyl_law}. Now, we note that if $t\geq 1$, then we obviously have that $J_{t}t^{-2} + J_{t}^{1-2\alpha}\lesssim t^{\frac{1}{\alpha}-2}$. Using again that $1/2-\alpha<0$, the same conclusion holds for $t\in [r,1]$. In short, we have shown that \eqref{L4_norm_estimate} is true, and due to the other mentioned estimates, for the terms in the right-hand side of \eqref{schatten_holder_ineq}, we have that $\|I_{1}(t)\|_{L_2}\lesssim h^{2}t^{\frac{1}{2\alpha}}$. Reasoning analogously, it is enough to estimate the corresponding resolvent factors in order to estimate the other integrands in the right-hand side of \eqref{principal_relation_1}. We use the explicit choice of $R$ in one of these cases. Indeed, following the same notation for $Q$, we use the relations \eqref{spec_relations} to get the inequality $|Re^{i\theta}-\lambda_{Q, j}|\geq |R-\lambda_{Q, j}|\geq Ch^{-2}=\frac{C}{C+1}R$, and conclude that
\begin{equation*}
    \|Q^{1/2}(Re^{i\theta}-Q)^{-1}\|_{L_4}^4=\sum_{j=1}^{N_h} \frac{\lambda_{Q, j}^2}{|Re^{i\theta}-\lambda_{Q, j}|^4} \le R^{-4}N_{h} (Ch^{-2})^{2} \simeq R^{\frac{1}{\alpha}-2}.
\end{equation*}
Since $r$ is fixed, it is easy to conclude that $\|Q^{1/2}(re^{i\theta}-Q)^{-1}\|_{L_4}^4\lesssim 1$. In short, we conclude that $\|I(Re^{i\theta})\|_{L_2}\lesssim h^{2}R^{\frac{1}{2\alpha}}$ and $\|I(re^{i\theta})\|_{L_2}\lesssim h^{2}$, and we have found an estimate for all integrands in \eqref{principal_relation_1}. Consequently, we have that 
\begin{equation*}
    \begin{split}
        \cE_{h} &\lesssim \int_{r}^{R} t^{-2\beta} h^{2} t^{\frac{1}{2\alpha}} \,dt + R^{1-2\beta}\int_{-\pi}^{\pi} h^{2} R^{\frac{1}{2\alpha}} \,d\theta + r^{1-2\beta}\int_{-\pi}^{\pi} h^{2} \,d\theta \\
        &\lesssim h^{2} \int_{r}^{R} t^{-2\beta+\frac{1}{2\alpha}} \,dt + h^{2}R^{1-2\beta+\frac{1}{2\alpha}} + h^{2} \lesssim  h^{\min \{4\beta-\frac{1}{\alpha}, 2\}}\left(\log \frac{1}{h}\right)^{[4\beta-\frac{1}{\alpha}=2]},
    \end{split}
\end{equation*}
where $[\cdot]$ denotes the Iverson bracket, and we have used the explicit choice of $R$ in the last estimate. 

\textit{Case $\alpha\le\frac{1}{2}$:}  Since $\widetilde{L}_{h}^{-1}\Pi_{h}=L_{h}^{-1}\Lambda_{h}$, we have that
$\widetilde L_h^{-2\beta}\Pi_h-L_h^{-2\beta}\Lambda_{h} = (X^{\rho}-Y^{\rho})Y\Lambda_{h} = K_{\rho}\Lambda_{h}$, where $\rho=2\beta-1$, $X=\widetilde{L}_{h}^{-1}$ and $Y=L_{h}^{-1}$. Observe that with this notation and our definition of $G_{h}$, we have that $X-Y=X^{1/2}G_{h}Y^{-1/2}$. Now, assume that $(\lambda_i,\psi_{i})_{i=1}^{N_h}$ and $(\gamma_{j},\xi_{j})_{j=1}^{N_h}$ denote $(\cdot,\cdot)$-orthonormal and $\langle\cdot,\cdot\rangle_{h}$-orthonormal eigenpairs of $X$ and $Y$, respectively. Further, consider the notation $\ell_{i,j}= (\xi_{j},\psi_{i})$ and $g_{i,j}=(G_{h}\xi_{j},\psi_{i})$. Notice that the following relation holds: 
\begin{equation*}
    \left(\lambda_i-\gamma_j\right) \ell_{i,j}=\left((X-Y) \xi_{j}, \psi_i\right)=\left( X^{1/2}G_{h}Y^{-1/2}\xi_{j}, \psi_i\right)=\lambda_i^{\frac{1}{2}}g_{i,j}\gamma_j^{-\frac{1}{2}}.
\end{equation*}
Consequently, we get that
\begin{equation}\label{K_rho_relation}
    (K_{\rho}\xi_{j},\psi_{i})=\lambda_{i}^{\frac{1}{2}}\gamma_{j}^{\frac{1}{2}}\frac{\lambda_{i}^{\rho}-\gamma_{j}^{\rho}}{\lambda_{i}-\gamma_{j}}g_{i,j}.
\end{equation}
Since $0<\alpha\le \frac{1}{2}$ and $2\beta>\frac{1}{\alpha}$, it follows that $\rho>\frac{1}{2\alpha}\geq 1$. Using that $\rho\geq 1$, the scalar mean value theorem implies that 
\begin{equation}\label{mean_value_theorem}
    \left|\frac{x^{\rho}-y^{\rho}}{x-y}\right|^{2}\leq \rho^{2} (x^{\rho-1}+y^{\rho-1})^{2} \le \rho^{2}2^{2}[x^{2(\rho-1)}+y^{2(\rho-1)}], \quad \forall x,y>0.
\end{equation}
Squaring both sides of \eqref{K_rho_relation},using \eqref{mean_value_theorem}, and taking the summation over $1\le i,j\le N_h$, we obtain that
\begin{equation}\label{K_rho_norm_est0}
    \|K_\rho\|_{L_2} \lesssim_\rho \| X^{\rho-1/2} G_{h} Y^{1/2}\|_{L_2}+\| X^{1/2} G_{h} Y^{\rho-1/2} \|_{L_2}.
\end{equation}
In order to estimate the $L_2$-norms of the terms on the right-hand side of the previous estimate, we use the fact that given two real numbers $a,b>0$ such that $a+b>\frac{1}{2\alpha}$, there exist $2\le p,q<\infty$ such that 
\begin{equation*}
    \frac{1}{p}+\frac{1}{q}=\frac{1}{2}, \quad \alpha p a >1, \quad \alpha q b >1.
\end{equation*}
Using the Schatten--Hölder inequality on each term on the right-hand side of \eqref{K_rho_norm_est0}, with the choices $(a,b)=(\rho-\tfrac12,\tfrac12)$ and $(a,b)=(\tfrac12,\rho-\tfrac12)$, and recalling that $\|G_h\|_{L}\lesssim h^{2}$ and $\|X^{a}\|_{p}+\|Y^{b}\|_{q}\lesssim 1$ whenever $\alpha ap>1$ and $\alpha bq>1$,
we get that
\begin{equation}\label{K_rho_norm_estimate}
    \|K_{\rho}\|_{L_2} \lesssim h^{2}.
\end{equation}
Finally, using the fact that $\|\Lambda_{h}\|_{L(L^{2}(\cD), V_h)}\lesssim 1$, and  \eqref{K_rho_norm_estimate}, we conclude that 
$\cE_{h}\lesssim h^{2}$.
\end{proof}

We now prove Theorem~\ref{rat_approx_theorem}, which quantifies the additional error
introduced by the rational approximation when combined with the lumped-mass
discretization.

\begin{proof}[Proof of Theorem~\ref{rat_approx_theorem}]
By the triangle inequality and the definition of the rational-based approximation, we
have
\begin{equation}\label{rt_app_s}
    \begin{split}
    \left\|\varrho^\beta-\varrho_{m,h}^\beta\right\|_{\LDD}
    &\lesssim
    \left\|\varrho^\beta-\varrho_h^\beta\right\|_{\LDD} 
    +
    \left\|
    L_h^{-2\beta}\Lambda_h
    -
    R_{2\beta,m,h}\Lambda_h
    \right\|_{\HS}.
    \end{split}
\end{equation}
The first term on the right-hand side is bounded by Theorem~\ref{main_result_1}. The
second term is controlled by the rational approximation estimate
\eqref{rel_exp_conv}. Therefore, for every
$\eta<\min\{4\beta-\frac{1}{\alpha},2\}$ and all sufficiently small $h$,
\begin{equation*}
    \left\|\varrho^\beta-\varrho_{m,h}^\beta\right\|_{\LDD}
    \lesssim
    \max\left\{
    \left\|\varrho^\beta-\widehat{\varrho}_h^\beta\right\|_{\LDD},
    h^\eta
    \right\}
    +
    \mathbbm{1}_{\{2\beta\notin\mathbb{N}\}}
    h^{-1/\alpha}e^{-2\pi\sqrt{\{2\beta\}m}}.
\end{equation*}
This is precisely \eqref{err_rat_cov}.
\end{proof}

We now prove Theorem~\ref{prop_fin}.

\begin{proof}[Proof of Theorem~\ref{prop_fin}]
We first rewrite the estimate in terms of covariance operators. For notational convenience, set $B_h:=\cI_hM_{\tau^{-1}}|_{V_h}$. Since
$L^\beta(\tau u)=\cW$, the exact covariance operator is
$M_{\tau^{-1}}L^{-2\beta}M_{\tau^{-1}}$. Hence, by the triangle inequality,
\begin{equation*}
    \begin{split}
    \cE_h
    &:=
    \left\|
    M_{\tau^{-1}}L^{-2\beta}M_{\tau^{-1}}
    -
    B_hR_{2\beta,m,h}\Lambda_h B_h^*\Pi_h
    \right\|_{\HS}
    \\
    &\leq
    \left\|
    M_{\tau^{-1}}L^{-2\beta}M_{\tau^{-1}}
    -
    M_{\tau^{-1}}L_h^{-2\beta}\Lambda_hM_{\tau^{-1}}
    \right\|_{\HS}
    \\
    &\quad+
    \left\|
    M_{\tau^{-1}}L_h^{-2\beta}\Lambda_hM_{\tau^{-1}}
    -
    B_hL_h^{-2\beta}\Lambda_h B_h^*\Pi_h
    \right\|_{\HS}
    \\
    &\quad+
    \left\|
    B_h
    \left(L_h^{-2\beta}\Lambda_h-R_{2\beta,m,h}\Lambda_h\right)
    B_h^*\Pi_h
    \right\|_{\HS}
    =:\cI_1+\cI_2+\cI_3 .
    \end{split}
\end{equation*}
We estimate the three terms separately. The term $\cI_3$ is present only when
$2\beta\notin\mathbb{N}$.

Since $\tau$ is bounded away from zero and $\tau^{-1}\in\cR^k$, Assumption~\ref{assump2}
implies that
$\|\cI_h(\tau^{-1}v)\|_{\LD}\lesssim \|v\|_{\LD}$ for every $v\in V_h$.
Thus $B_h$ is uniformly bounded as an operator on $L^2(\cD)$, and so is
$B_h^*\Pi_h$. If $\{e_i\}_{i\in\mathbb{N}}$ is an orthonormal basis of $L^2(\cD)$, 
\begin{equation}\label{est_usef}
    \begin{split}
    \cI_3^2
    &=
    \sum_i
    \left\|
    B_h
    \left(L_h^{-2\beta}\Lambda_h-R_{2\beta,m,h}\Lambda_h\right)
    B_h^*\Pi_h e_i
    \right\|_{\LD}^2
    \\
    &\lesssim
    \sum_i
    \left\|
    \left(L_h^{-2\beta}\Lambda_h-R_{2\beta,m,h}\Lambda_h\right)
    B_h^*\Pi_h e_i
    \right\|_{\LD}^2\\
    &\leq
    \left\|
    L_h^{-2\beta}\Lambda_h-R_{2\beta,m,h}\Lambda_h
    \right\|_{\HS}^2
    \|B_h^*\Pi_h\|_{\LL}^2
    \\
    &\lesssim
    \left\|
    L_h^{-2\beta}\Lambda_h-R_{2\beta,m,h}\Lambda_h
    \right\|_{\HS}^2
    \lesssim
    \mathbbm{1}_{\{2\beta\notin\mathbb{N}\}}\,
    h^{-2/\alpha}e^{-4\pi\sqrt{\{2\beta\}m}},
    \end{split}
\end{equation}
where in the last step we used \eqref{rel_exp_conv}. Hence
$
\cI_3
\lesssim
\mathbbm{1}_{\{2\beta\notin\mathbb{N}\}}\,
h^{-1/\alpha}e^{-2\pi\sqrt{\{2\beta\}m}}.
$
The bound for $\cI_1$ follows from the boundedness of $M_{\tau^{-1}}$ as 
$\|M_{\tau^{-1}}\|_{\LL}=\|\tau^{-1}\|_{\infty}$, and therefore
$$
\cI_1
\leq
\|\tau^{-1}\|_{\infty}^2
\left\|
L^{-2\beta}-L_h^{-2\beta}\Lambda_h
\right\|_{\HS}
\lesssim
\left\|
\varrho^\beta-\varrho_h^\beta
\right\|_{\LDD}.
$$
By Theorem~\ref{main_result_1}, for every
$\eta<\min\{4\beta-\frac{1}{\alpha},2\}$, we have 
$
\cI_1
\lesssim
\max\left\{
\left\|\varrho^\beta-\widehat{\varrho}_h^\beta\right\|_{\LDD},
h^\eta
\right\}.
$

It remains to estimate $\cI_2$. Adding and subtracting
$\cI_hM_{\tau^{-1}}L_h^{-2\beta}\Lambda_hM_{\tau^{-1}}$, we obtain
\begin{equation*}
    \begin{split}
    \cA
    &:=
    M_{\tau^{-1}}L_h^{-2\beta}\Lambda_hM_{\tau^{-1}}
    -
    \cI_hM_{\tau^{-1}}L_h^{-2\beta}\Lambda_h
    \left(\cI_hM_{\tau^{-1}}|_{V_h}\right)^*\Pi_h
    \\
    &=
    (I-\cI_h)M_{\tau^{-1}}L_h^{-2\beta}\Lambda_hM_{\tau^{-1}}
   +
    \cI_hM_{\tau^{-1}}L_h^{-2\beta}\Lambda_h
    \left(
    M_{\tau^{-1}}
    -
    \left(\cI_hM_{\tau^{-1}}|_{V_h}\right)^*\Pi_h
    \right)
    \\
    &=:\cA_1+\cA_2 .
    \end{split}
\end{equation*}
Thus $\cI_2=\|\cA\|_{\HS}\leq
\|\cA_1\|_{\HS}+\|\cA_2\|_{\HS}$.
We first estimate $\cA_1$. Using Assumption~\ref{assump2}, in particular the
interpolation estimate for functions of the form $\tau^{-1}\psi$ with
$\psi\in V_h$, and recalling that $\tau^{-1}\in\cR^k$, we get, for any orthonormal
basis $\{e_i\}_{i\in\mathbb{N}}$ of $L^2(\cD)$,
\begin{equation}\label{est_cA1}
    \begin{split}
    \|\cA_1\|_{\HS}^2
    &=
    \sum_i
    \left\|
    (I-\cI_h)M_{\tau^{-1}}L_h^{-2\beta}\Lambda_hM_{\tau^{-1}}e_i
    \right\|_{\LD}^2
    \\
    &\lesssim
    \sum_i
    \left\|
    (I-\cI_h)M_{\tau^{-1}}L_h^{-2\beta}\Lambda_h e_i
    \right\|_{\LD}^2
    \lesssim
    h^{2k}
    \sum_i
    \left\|
    L_h^{-2\beta}\Lambda_h e_i
    \right\|_1^2\\
    &=
    h^{2k}
    \sum_i
    \left\|
    L_h^{-2\beta+\frac{1}{2}}\Lambda_h e_i
    \right\|_{\LD}^2
    =
    h^{2k}
    \left\|
    L_h^{-2\beta+\frac{1}{2}}\Lambda_h
    \right\|_{\HS}^2
    \lesssim h^{2k}.
    \end{split}
\end{equation}
Here we used the boundedness of $M_{\tau^{-1}}$ in the second line. The last estimate
follows from the Hilbert--Schmidt summability of
$L_h^{-2\beta+\frac{1}{2}}\Lambda_h$, which is ensured by the condition
$2\beta>\frac{1}{2\alpha}+\frac{1}{2}$.
We now estimate $\cA_2$. The adjoint relation
\begin{equation}\label{adj_rel}
    \left[
    (I-\cI_h)M_{\tau^{-1}}L_h^{-2\beta}\Lambda_h
    \right]^*
    =
    L_h^{-2\beta}\Lambda_h
    \left(
    M_{\tau^{-1}}
    -
    \left(\cI_hM_{\tau^{-1}}|_{V_h}\right)^*\Pi_h
    \right)
\end{equation}
shows that the operator appearing on the right of $L_h^{-2\beta}\Lambda_h$ in
$\cA_2$ is controlled by the same interpolation error as in the estimate of $\cA_1$.
Moreover, by the boundedness argument used in \eqref{est_usef}, it is enough to
bound the Hilbert--Schmidt norm of
$
L_h^{-2\beta}\Lambda_h
\left(
M_{\tau^{-1}}
-
\left(\cI_hM_{\tau^{-1}}|_{V_h}\right)^*\Pi_h
\right).
$
By \eqref{adj_rel} and the fact that the Hilbert--Schmidt norm of an operator equals
that of its adjoint, this norm is bounded by the same expression estimated in
\eqref{est_cA1}. Therefore,
$
\|\cA_2\|_{\HS}\lesssim h^k.
$
Combining the bounds for $\cA_1$ and $\cA_2$ gives
$
\cI_2\lesssim h^k.
$

Putting together the estimates for $\cI_1$, $\cI_2$, and $\cI_3$, we conclude that
\begin{equation*}
    \begin{split}
    \left\|
    \varrho^{\beta,\tau}-\varrho_{m,h}^{\beta,\tau}
    \right\|_{\LDD}
    &\lesssim
    \max\left\{
    \left\|\varrho^\beta-\widehat{\varrho}_h^\beta\right\|_{\LDD},
    h^\eta
    \right\}
    +h^k
    +
    \mathbbm{1}_{\{2\beta\notin\mathbb{N}\}}\,
    h^{-1/\alpha}e^{-2\pi\sqrt{\{2\beta\}m}}.
    \end{split}
\end{equation*}
This is precisely \eqref{fin_res_coverr}.
\end{proof}

\section{Numerical experiments}\label{sec:experiments}

\subsection{Fractional models in one dimension}

We consider a Gaussian generalized Whittle--Mat\'ern random field on
$\cD:=(0,1)$ specified through the SPDE
\begin{align}\label{e:numexp-a}
    (\kappa^2 - \Delta)^{\beta} u(s) &= \white(s), && s \in \cD,\\
    u(s) &= 0, && s\in\partial\cD,
\end{align}
where $\beta>1/4$ and $\kappa(s)^2=(1+s^2)/4$. The equation is discretized on a
uniform mesh using continuous piecewise linear basis functions. The fractional power is
approximated by the sinc-quadrature method described in Section~\ref{sec:fem}, with
quadrature step size $k=-1/(\beta\ln h)$, calibrated as a function of the mesh width
$h$ so that the quadrature error is of the same order as the finite element error.
Since $\kappa^2\in W^{2,\infty}(\cD)$, the theoretical convergence rate for the
$L^2(\cD\times\cD)$ error of the covariance function, with or without mass lumping, is
given by
$
    \min\{4\beta-\tfrac{1}{2},2\}.
$

To verify these rates numerically, we compare three approximation methods. The first is
the standard sinc-Galerkin approximation, where neither the $L^2$ inner product nor the
bilinear form is replaced by a lumped version. The second is the sinc-Galerkin
approximation with partial mass lumping: here the admissible bilinear form is kept
exact, namely $a_h=a_L|_{V_h\times V_h}$, while the mass matrix associated with the
discrete inner product is replaced by its lumped version. This is precisely the mass
matrix whose inverse appears in the matrix representation of the covariance
approximation. The third method is the sinc-Galerkin approximation with full mass
lumping, where both the discrete inner product and the bilinear form are replaced by
their lumped quadrature versions. One could also consider the complementary partial
mass-lumping variant, in which the inner product is kept exact while only the bilinear
form is replaced by its quadrature-based version; this case is not shown separately
here, since the three methods above already illustrate the effects relevant for our
main results.

We consider four values of the fractional exponent,
$\beta\in\{5/16,7/16,9/16,11/16\}$. For each value of $\beta$, we use four uniform
meshes with $N_\ell = 3+2^{5+\ell}$ nodes (including the two Dirichlet boundary nodes) and mesh widths $h_\ell = 1/(2+2^{5+\ell})$, for $\ell=0,\ldots,3$. Since the exact covariance function is
not available in closed form, we use a numerical approximation computed on a finer mesh
with $h_{\mathrm{ref}}=2^{-11}$ as a reference solution. The
$L^2(\cD\times\cD)$ errors are approximated by evaluating the covariance functions on
the reference mesh and computing the corresponding discrete average of the squared
differences. The resulting errors are shown in Figure~\ref{fig:d1errors}.

\begin{figure}[t]
    \centering
    \includegraphics[width=0.48\linewidth]{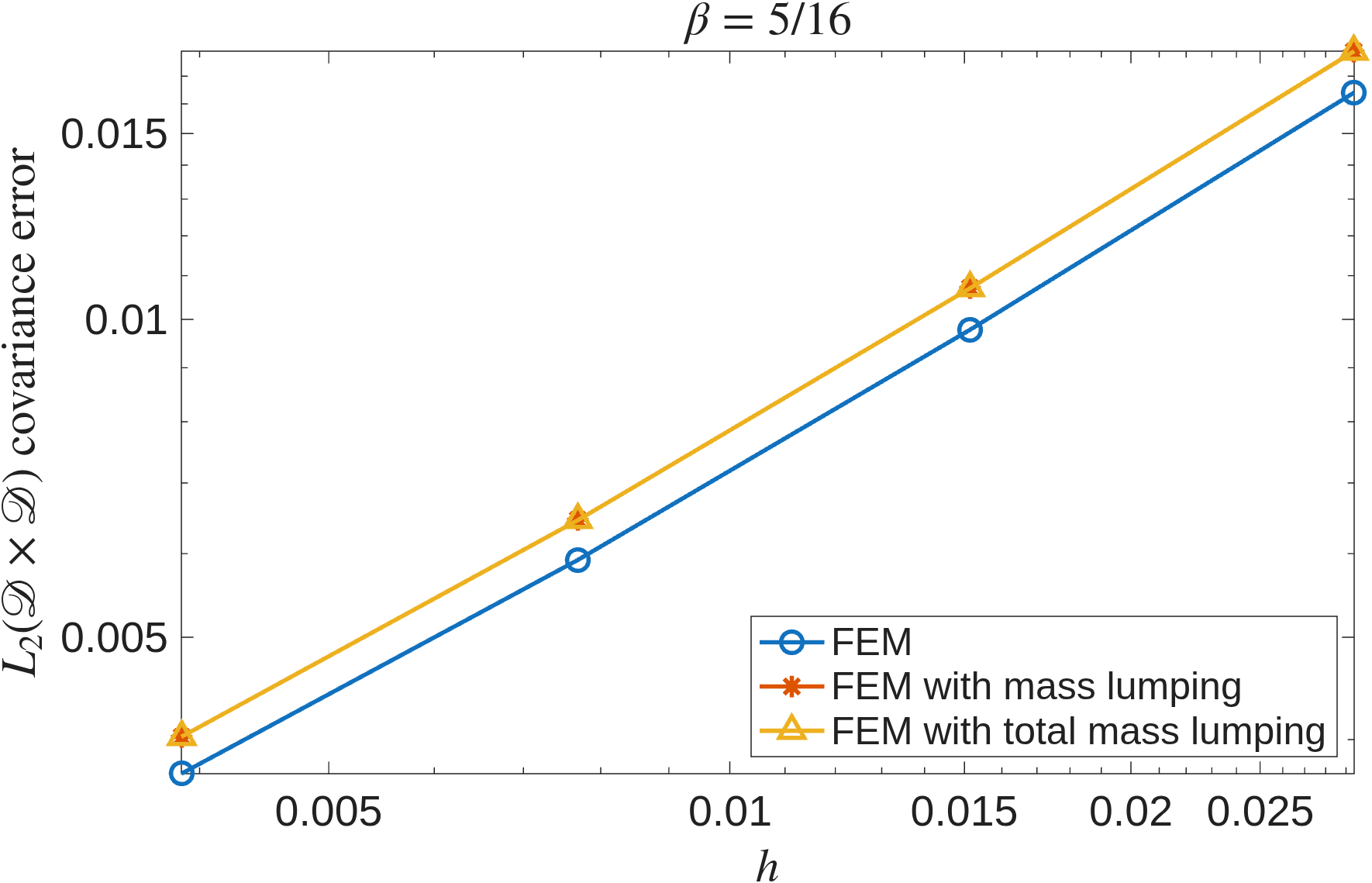}
    \hfill
    \includegraphics[width=0.48\linewidth]{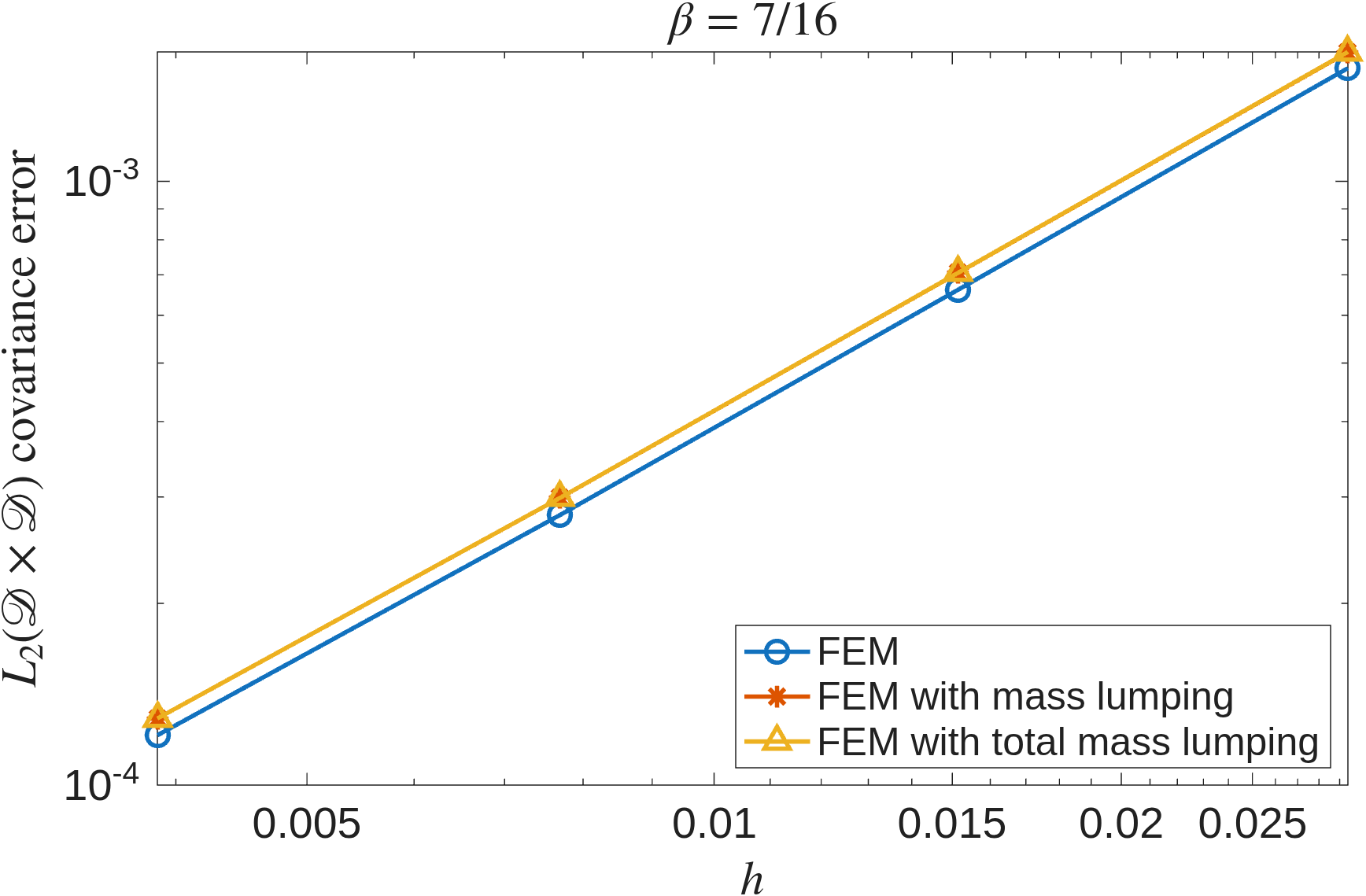}
    
    \vspace{0.2cm}
    \includegraphics[width=0.48\linewidth]{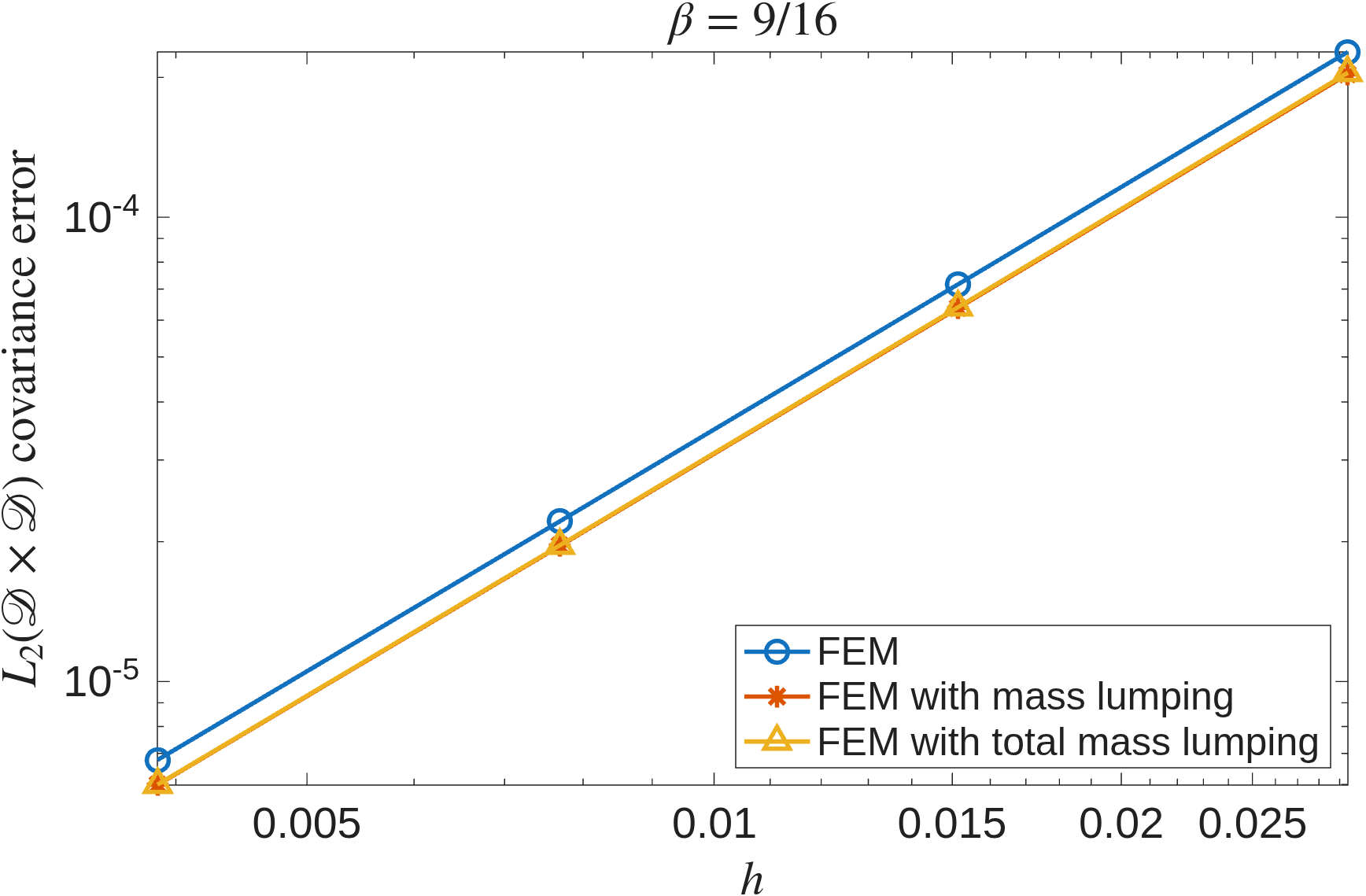}
    \hfill
    \includegraphics[width=0.48\linewidth]{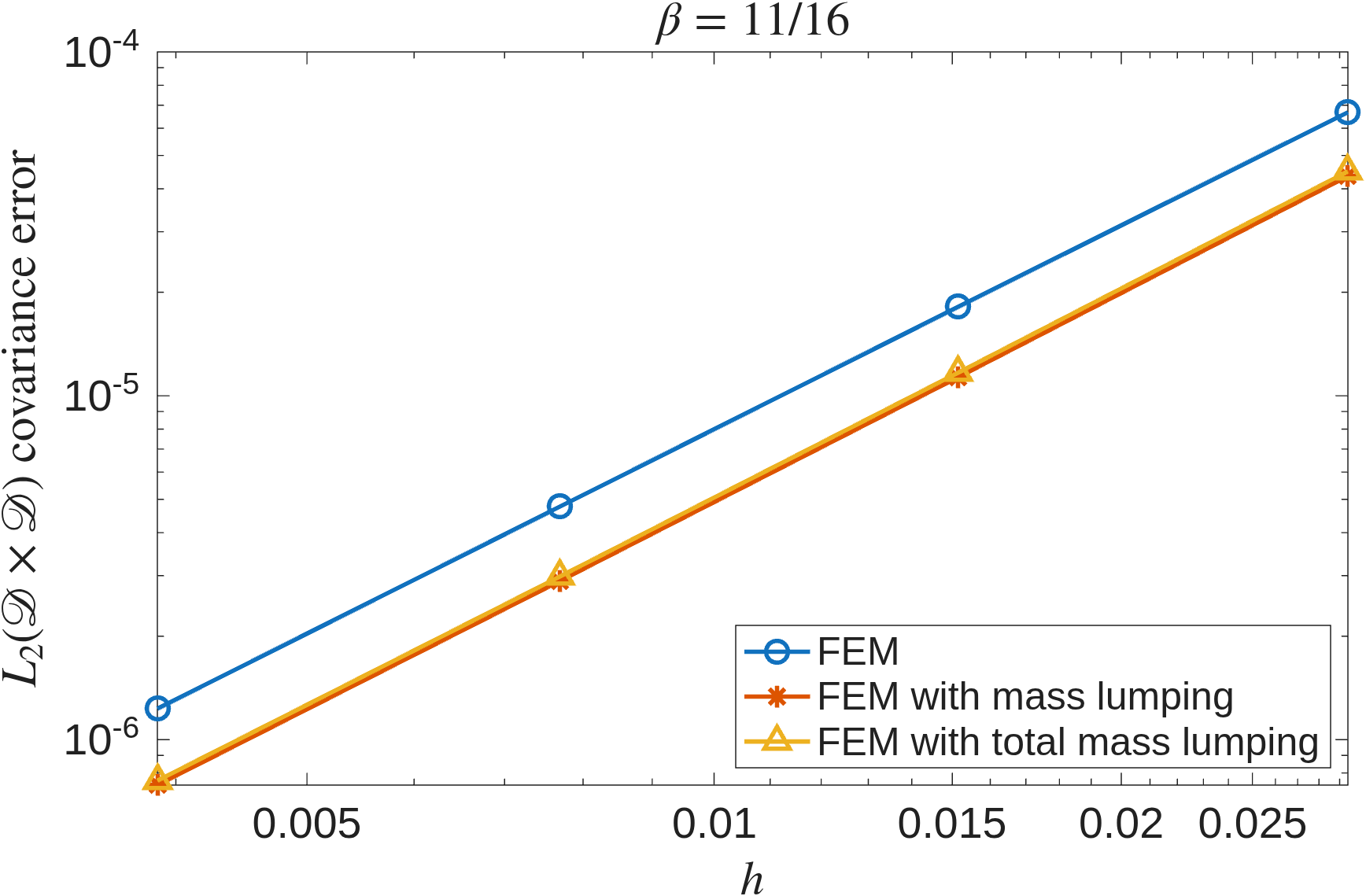}
    \caption{Observed covariance errors, shown on log-log scales as functions of the
    mesh width $h$, for different values of $\beta$.}
    \label{fig:d1errors}
\end{figure}

For each approximation method and each tested value of $\beta$, we obtain a data set
$\{(h_\ell,\mathrm{err}_\ell)\}_\ell$, where $\mathrm{err}_\ell$ denotes the covariance
error for the mesh width $h_\ell$. The observed convergence rate $\mathrm{r}$ is
computed as the least-squares slope in the regression
$\ln \mathrm{err}=\mathrm{c}+\mathrm{r}\ln h$. As shown in Table~\ref{tab:rates}, the
observed rates agree well with the theoretical values predicted by
Theorem~\ref{rat_approx_theorem}.

\begin{table}[t]
    \centering
    \caption{Observed convergence rates for the covariance errors in
    Figure~\ref{fig:d1errors}. The theoretical rates are shown in parentheses.}
    \label{tab:rates}
    \begin{tabular}{lcccc}
        \toprule
        $\beta$ & $5/16$ & $7/16$ & $9/16$ & $11/16$ \\
        \cmidrule(r){2-5}
        No mass lumping       & 0.73 (0.75) & 1.26 (1.25) & 1.73 (1.75) & 1.97 (2.00) \\
        Partial mass lumping  & 0.74 (0.75) & 1.26 (1.25) & 1.74 (1.75) & 2.01 (2.00) \\
        Full mass lumping     & 0.74 (0.75) & 1.26 (1.25) & 1.74 (1.75) & 2.01 (2.00) \\
        \bottomrule
    \end{tabular}
\end{table}

\subsection{An example on a metric graph}

Next, we consider the stochastic equation
\begin{equation}\label{eq:graph_spde}
    (\kappa^2-\Delta_\Gamma)^\beta u=\mathcal{W},
    \qquad \text{on } \Gamma,
\end{equation}
where $\Gamma$ is the compact metric graph shown in Figure~\ref{fig:graph}. The
operator $\Delta_\Gamma$ is equipped with standard Kirchhoff vertex conditions, as
described in Section~\ref{sec:metric_graph_domains}, and we set $\kappa=1$.

\begin{figure}[t]
    \begin{center}
        \includegraphics[width=0.54\linewidth]{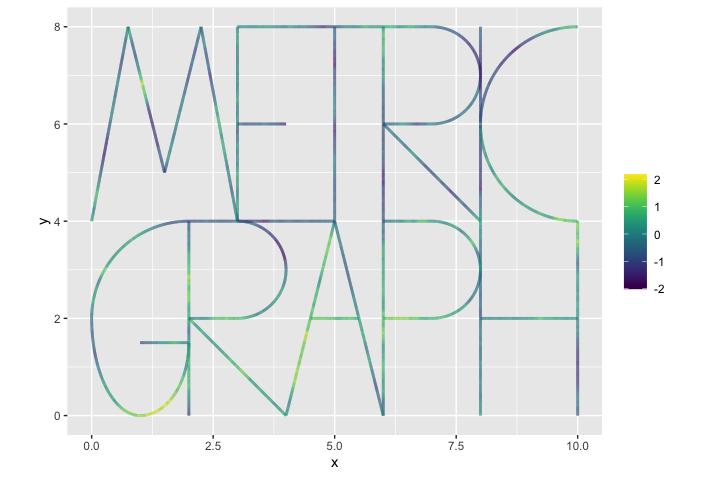}
        \includegraphics[width=0.445\linewidth]{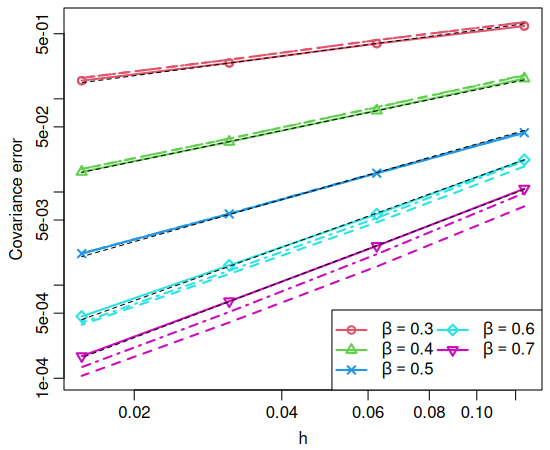}
    \end{center}
    \caption{A simulation of \eqref{eq:graph_spde} with $\beta=\nicefrac12$ and
    $\kappa=1$ on the metric graph used in the logo of the \texttt{MetricGraph}
    package (left). Observed $L^2(\Gamma\times\Gamma)$ covariance errors for different
    values of $\beta$ as functions of the mesh width $h$ (right). The different line types correspond to theoretical rates (black dashed), errors without mass
    lumping (solid), errors with partial mass lumping (dashed), and errors with full mass lumping (dash-dotted).}
    \label{fig:graph}
\end{figure}

By \cite[Theorem~6.9]{bolin2024regularity}, the Galerkin covariance approximation on
compact metric graphs has convergence rate
$\min\{4\beta-\nicefrac12,2\}$ for $\beta\in(\nicefrac14,1)$. Since metric graphs have
spectral dimension $d=1$, this is precisely the rate predicted by the general theory in
the present work. In the sinc-quadrature approximation of the fractional power, we take
the quadrature step size proportional to $|\log h|^{-1}$; more precisely, we use
$k=-1/(\beta\ln h)$. This choice ensures that the quadrature error is of the same order
as the finite element error.

We perform the same type of experiment as in the one-dimensional Euclidean case. We
consider $\beta=\nicefrac{n}{10}$ for $n=3,4,5,6,7$, and use four meshes obtained by
subdividing each edge into equally sized elements so that the maximal segment length is
$h_\ell=2^{-\ell}$, for $\ell=2,3,4,5$. For each value of $\beta$, we compare three
approximations: the standard sinc-Galerkin approximation without mass lumping, the
approximation with partial mass lumping, and the approximation with full mass lumping.
As before, partial mass lumping means that only the mass matrix associated with the
discrete inner product is replaced by its lumped version, whereas full mass lumping
means that both the discrete inner product and the bilinear form are replaced by their
lumped quadrature versions.

Since the exact covariance function is not available in closed form, we use a
sinc-Galerkin approximation without mass lumping on a fine reference mesh with
$h_{\mathrm{ref}}=2^{-8}$ as a proxy for the true covariance. For each approximation,
the $L^2(\Gamma\times\Gamma)$ error is computed by approximating the covariance
functions as piecewise constant on the reference mesh. The resulting errors are shown
in Figure~\ref{fig:graph}.

For each approximation method and each value of $\beta$, we computed the observed
convergence rates as in Section~6.1. The resulting observed rates, shown
in Table~\ref{tab:graph_rates}, agree well with the theoretical rates and confirm
that the use of mass lumping does not reduce the convergence order.

\begin{table}[t]
    \caption{Observed (theoretical) convergence rates for the covariance errors in
    Figure~\ref{fig:graph}. }
    \begin{center}
        \begin{tabular}{lccccc}
            \toprule
            $\beta$ &  $0.3$  & $0.4$ & $0.5$ & $0.6$ & $0.7$\\
            \cmidrule(r){2-6}
            No mass lumping       &  0.65 (0.7)  &  1.11 (1.1) &  1.44 (1.5) & 1.86 (1.9) & 1.99 (2.0)\\
            Partial mass lumping  &  0.66 (0.7)  &  1.11 (1.1) &  1.44 (1.5) & 1.88 (1.9) & 2.01 (2.0)\\
            Full mass lumping     &  0.67 (0.7)  &  1.12 (1.1) &  1.45 (1.5) & 1.92 (1.9) & 2.07 (2.0)\\
            \bottomrule
        \end{tabular}
    \end{center}
    \label{tab:graph_rates}
\end{table}

\subsection{An example on the sphere}
Finally, let us illustrate the theoretical results for an example on a closed surface. In particular, we consider the sphere $\mathbb{S}^2$ and the SPDE
\begin{equation}\label{eq:spde}
	(\kappa^2 - \Delta)^{\beta} \left(\tau u \right) = \mathcal{W}
	\quad 
	\text{on}\;\; \mathbb{S}^2,  
\end{equation}
where $\Delta$ is the Laplace-Beltrami operator, and $\mathcal{W}$ is Gaussian white noise on $L^2(\mathbb{S}^2)$.
By, e.g.\cite{kb-kriging}, the covariance function of $u$ is given by 
\begin{equation}\label{eq:sphere_cov} 
\varrho(x,x') 
= 
\sum_{\ell=0}^{\infty}  
\frac{\tau^{-2}}{(\kappa^2+ \ell(\ell+1))^{2\beta}} \frac{2\ell + 1}{4\pi}P_\ell(\cos d_{\mathbb{S}^2}(x,x')), 
\end{equation}
where $d_{\mathbb{S}^2}(\cdot,\cdot)$ denotes the geodesic distance and 
$P_\ell\from[-1,1]\to\bbR$ is the $\ell$-th Legendre polynomial, 
$$
P_\ell(y) 
= 
2^{-\ell} 
\frac1{\ell!}
\frac{\rd^\ell}{\rd y^\ell} 
\left( y^2 -1 \right)^\ell, 
\quad 
y\in[-1,1], 
\qquad 
\ell\in\bbN_0:=\{0,1,2,\ldots\}.
$$

We again focus on the covariance error, and this time consider the covariance-based rational SPDE approach by  \cite{bolin2024covariance} to approximate the covariance function of $u$ directly. As \cite{bolin2024covariance} suggested, we perform mass lumping to a computationally efficient GMRF approximation, and we now try to verify the theoretical rate of convergence for this approach when mass lumping is performed (\cite{bolin2024covariance} only derived the rate without mass lumping). 

We consider five different quasi-uniform meshes of the sphere:
$$
\begin{array}{c|ccccc}
\hline
\#\text{nodes} & 1692 & 2252 & 3242 & 4002 & 5762\\
\hline
h & 0.12036 & 0.10497 & 0.08812 & 0.07961 & 0.06675\\
\hline
\end{array}
$$
The order of the rational approximation is fixed at $m=8$. For this choice, the error
from the rational approximation is smaller than the finite element error, according to
the calibration of the rational order in \cite{bolin2024covariance}. We therefore
expect the covariance error to converge at the rate $\min\{4\beta-1,2\}$.

The true covariance function is approximated by truncating the infinite series in \eqref{eq:sphere_cov} to $\ell\in[0,60]$ and the $L^2(\mathbb{S}^2\times\mathbb{S}^2)$ error is computed by approximating the covariance functions to be piecewise constant on an overkill mesh consisting of $N_{ok} = 12252$ nodes. The resulting errors can be seen in Figure~\ref{fig:sphere} together with a simulation of the field $u$ on the overkill mesh.  

Based on the observed errors, linear regression is again used to compute the empirical rates of convergence for each value of $\beta$.
The resulting observed rates are shown in Table~\ref{tab:sphere_rates} and validate the theoretical results. 

\begin{figure}[t]
	\begin{center}
		\includegraphics[width=0.4\linewidth]{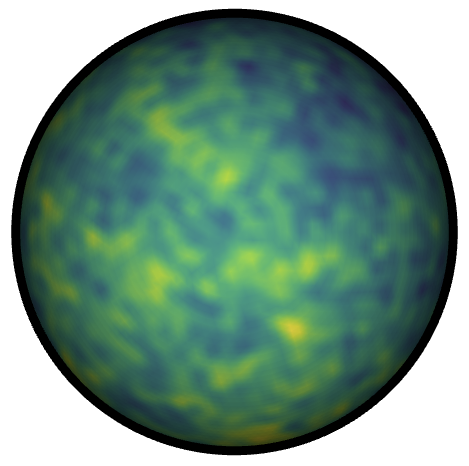}
		\includegraphics[width=0.55\linewidth]{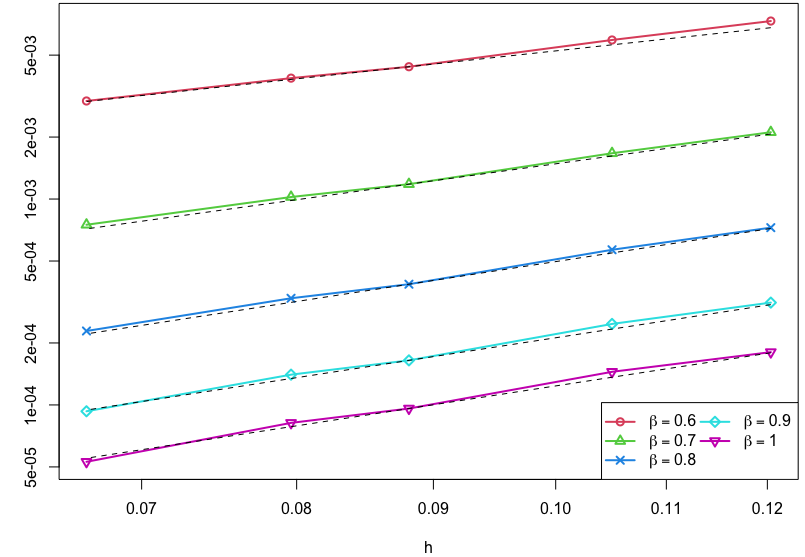}    
	\end{center}
	\caption{A simulation of \eqref{eq:spde} with $\beta = 0.6$ and $\tau=1=\kappa = 1$ on the unit sphere (left) and observed $L^2(\mathbb{S}^{2}\times\mathbb{S}^{2})$ covariance errors for different values of $\beta$ as functions of the mesh size $h$ (right) with the same choice of $\tau$ and $\kappa$. The black dashed lines show the theoretical rates for each case.}
	\label{fig:sphere}
\end{figure}

\begin{table}[t]
	\caption{Observed (theoretical) rates of convergence for the covariance errors in Figure~\ref{fig:sphere}.}
	\begin{center}
		\begin{tabular}{lccccc}
			\toprule
			$\beta$ &  $0.6$  & $0.7$ & $0.8$ & $0.9$ & $1$\\
			\cmidrule(r){2-6}
			Rates 		&  1.52 (1.4)  &  1.76 (1.8) &  1.97 (2.0) & 2.06 (2.0) & 2.08 (2.0)\\
			\bottomrule
	\end{tabular}
\end{center}
	\label{tab:sphere_rates}
\end{table}

\section{Discussion}\label{sec:discussion}

Figure~\ref{fig:d1errors} shows that the fully mass-lumped approximation can have a
smaller error than the approximation in which mass lumping is applied only to the
right-hand side. This phenomenon has previously been observed in the non-fractional
case $\beta=1$ with constant $\kappa$. In that setting, the matrix expression for the
precision operator reveals that some of the errors introduced by replacing the mass
matrix with its lumped version may partially cancel when all mass matrices in the
precision matrix are treated consistently.

Our results provide a theoretical justification for this observation in a much broader
setting. In particular, they show that replacing the exact mass matrices by lumped ones
does not reduce the convergence rate of the covariance approximation. This remains true
for fractional powers of elliptic operators and for the non-stationary extensions
considered in this work.

The analysis of the non-stationary variance-control factor $\tau$ gives an additional
conclusion. For models of the form $L^\beta(\tau u)=\mathcal W$, the covariance of the
solution involves multiplication by $\tau^{-1}$ on both sides of the covariance
operator associated with $L^{-\beta}\mathcal W$. In practical finite element
implementations, this multiplication must also be discretized, typically through nodal
interpolation or quadrature-based approximations. Our results show that the error
introduced by this additional discretization is controlled by the regularity of
$\tau^{-1}$. More precisely, if $\tau^{-1}$ belongs to the multiplier space $\mathcal
R^k$, which is typically of the form $W^{k,\infty}(\cD)$ coupled with appropriate boundary conditions, then the covariance error contains an additional term of order $h^k$. Thus, when
$k\geq2$, the discretization of the variance-control factor does not reduce the
second-order convergence rate of the covariance approximation.

Thus, mass lumping is not only computationally attractive because it leads to sparse
and easily invertible mass matrices, but it can also be used without sacrificing the
asymptotic convergence rate of the covariance approximation, even in the fractional and
non-stationary settings considered here.

\bibliographystyle{siamplain}
\bibliography{references}

@article{Porcu2024,
  author  = {Emilio Porcu and Moreno Bevilacqua and Robert Schaback and Chris J. Oates},
  title   = {The {M}at\'{e}rn Model: A Journey Through Statistics, Numerical Analysis and Machine Learning},
  journal = {Statist. Sci.},
  volume  = {39},
  number  = {3},
  pages   = {469--492},
  year    = {2024},
  doi     = {10.1214/24-STS923}
}

@article{stahl1993best,
  author  = {Stahl, Herbert R.},
  title   = {Best uniform rational approximation of {$x^\alpha$} on {$[0,1]$}},
  journal = {Bull. Amer. Math. Soc.},
  volume  = {28},
  number  = {1},
  pages   = {116--122},
  year    = {1993},
  doi     = {10.1090/S0273-0979-1993-00351-3}
}

@article{harizanov2020analysis,
  author  = {Harizanov, Stanislav and Lazarov, Raytcho and Margenov, Svetozar and Marinov, Pencho and Pasciak, Joseph},
  title   = {Analysis of numerical methods for spectral fractional elliptic equations based on the best uniform rational approximation},
  journal = {J. Comput. Phys.},
  volume  = {408},
  pages   = {109285},
  year    = {2020},
  doi     = {10.1016/j.jcp.2020.109285}
}

@article{lindgren2015rinla,
  author  = {Lindgren, Finn and Rue, H{\aa}vard},
  title   = {{Bayesian} spatial modelling with {R-INLA}},
  journal = {J. Stat. Softw.},
  volume  = {63},
  number  = {19},
  pages   = {1--25},
  year    = {2015},
  doi     = {10.18637/jss.v063.i19}
}

@Manual{R,
  author       = {{R Core Team}},
  title        = {R: A Language and Environment for Statistical Computing},
  organization = {R Foundation for Statistical Computing},
  address      = {Vienna, Austria},
  year         = {2026},
  url          = {https://www.R-project.org/}
}

@Manual{rnaturalearth,
  author = {Massicotte, Philippe and South, Andy},
  title  = {rnaturalearth: World Map Data from Natural Earth},
  year   = {2026},
  note   = {R package version 1.2.0},
  url    = {https://CRAN.R-project.org/package=rnaturalearth}
}

@article{sf,
  author  = {Pebesma, Edzer},
  title   = {Simple Features for {R}: Standardized Support for Spatial Vector Data},
  journal = {R J.},
  volume  = {10},
  number  = {1},
  pages   = {439--446},
  year    = {2018},
  doi     = {10.32614/RJ-2018-009}
}

@Manual{fmesher,
  author = {Lindgren, Finn},
  title  = {fmesher: Triangle Meshes and Related Geometry Tools},
  year   = {2026},
  note   = {R package version 0.7},
  url    = {https://CRAN.R-project.org/package=fmesher}
}

@Manual{rSPDEpackage,
  author = {Bolin, David and Simas, Alexandre B.},
  title  = {{rSPDE}: Rational Approximations of Fractional Stochastic Partial Differential Equations},
  year   = {2025},
  note   = {R package version 2.5.1},
  url    = {https://CRAN.R-project.org/package=rSPDE}
}

@book{krainski2019advanced,
  author    = {Krainski, Elias T. and G{\'o}mez-Rubio, Virgilio and Bakka, Haakon and
                 Lenzi, Amanda and Castro-Camilo, Daniela and Simpson, Daniel and
                 Lindgren, Finn and Rue, H{\aa}vard},
  title     = {Advanced Spatial Modeling with Stochastic Partial Differential Equations Using {R} and {INLA}},
  publisher = {Chapman and Hall/CRC},
  address   = {Boca Raton},
  year      = {2019},
  doi       = {10.1201/9780429031892}
}

@article{ingebrigtsen2014spatial,
  author  = {Ingebrigtsen, Rikke and Lindgren, Finn K. and Steinsland, Ingelin},
  title   = {Spatial models with explanatory variables in the dependence structure},
  journal = {Spat. Stat.},
  volume  = {8},
  pages   = {20--38},
  year    = {2014},
  doi     = {10.1016/j.spasta.2013.06.002}
}

@article{fuglstad2015exploring,
  author  = {Fuglstad, Geir-Arne and Lindgren, Finn and Simpson, Daniel and Rue, H{\aa}vard},
  title   = {Exploring a new class of non-stationary spatial {Gaussian} random fields with varying local anisotropy},
  journal = {Statist. Sinica},
  volume  = {25},
  number  = {1},
  pages   = {115--133},
  year    = {2015},
  doi     = {10.5705/ss.2013.106w}
}

@article{bonito2015numerical,
  author  = {Bonito, Andrea and Pasciak, Joseph E.},
  title   = {Numerical approximation of fractional powers of elliptic operators},
  journal = {Math. Comp.},
  volume  = {84},
  number  = {295},
  pages   = {2083--2110},
  year    = {2015},
  doi     = {10.1090/S0025-5718-2015-02937-8}
}

@article{bonito2019sinc,
  author  = {Bonito, Andrea and Lei, Wenyu and Pasciak, Joseph E.},
  title   = {On sinc quadrature approximations of fractional powers of regularly accretive operators},
  journal = {J. Numer. Math.},
  volume  = {27},
  number  = {2},
  pages   = {57--68},
  year    = {2019},
  doi     = {10.1515/jnma-2017-0116}
}

@article{bolin2026statistical,
  author  = {Bolin, David and Simas, Alexandre B. and Wallin, Jonas},
  title   = {Statistical inference for {Gaussian} {Whittle}--{Mat{\'e}rn} fields on metric graphs},
  journal = {J. R. Stat. Soc. Ser. B Stat. Methodol.},
  year    = {2026},
  doi     = {10.1093/jrsssb/qkag074},
  url     = {https://doi.org/10.1093/jrsssb/qkag074},
}

@book{berkolaiko2013introduction,
  author    = {Berkolaiko, Gregory and Kuchment, Peter},
  title     = {Introduction to Quantum Graphs},
  series    = {Mathematical Surveys and Monographs},
  volume    = {186},
  publisher = {American Mathematical Society},
  address   = {Providence, RI},
  year      = {2013},
  isbn      = {978-0-8218-9211-4}
}

@article{bolin2024gaussian,
  author  = {Bolin, David and Simas, Alexandre B. and Wallin, Jonas},
  title   = {{Gaussian} {Whittle}--{Mat{\'e}rn} fields on metric graphs},
  journal = {Bernoulli},
  volume  = {30},
  number  = {2},
  pages   = {1611--1639},
  year    = {2024},
  doi     = {10.3150/23-BEJ1647}
}

@misc{OpenStreetMap,
  author       = {{OpenStreetMap contributors}},
  title        = {{OpenStreetMap} data},
  year         = {2026},
  url          = {https://www.openstreetmap.org}
}

@misc{overpassTurbo,
  author       = {Raifer, Martin},
  title        = {{Overpass turbo}: A web-based data filtering tool for {OpenStreetMap}},
  year         = {2026},
  url          = {https://overpass-turbo.eu/}
}

@Manual{MetricGraphPackage,
  author = {Bolin, David and Simas, Alexandre B. and Wallin, Jonas},
  title  = {{MetricGraph}: Random fields on metric graphs},
  year   = {2025},
  note   = {R package version 1.6.0},
  url    = {https://CRAN.R-project.org/package=MetricGraph}
}

@article{bolin2026markov,
  author  = {Bolin, David and Simas, Alexandre B. and Wallin, Jonas},
  title   = {{Markov} properties of {Gaussian} random fields on compact metric graphs},
  journal = {Bernoulli},
  volume  = {32},
  number  = {1},
  pages   = {153--178},
  year    = {2026},
  doi     = {10.3150/25-BEJ1853}
}

@article{hofreither2020unified,
  author  = {Hofreither, Clemens},
  title   = {A unified view of some numerical methods for fractional diffusion},
  journal = {Comput. Math. Appl.},
  volume  = {80},
  number  = {2},
  pages   = {332--350},
  year    = {2020},
  doi     = {10.1016/j.camwa.2019.07.025}
}

@misc{almeida-sousa2026finite,
  author        = {Almeida-Sousa, Kelvin J. R. and Bolin, David and Simas, Alexandre B.},
  title         = {Finite element and box-method discretizations for fractional elliptic problems with quadrature and mass lumping},
  year          = {2026},
  doi           = {10.48550/arXiv.2605.12082},
  eprint        = {2605.12082},
  archivePrefix = {arXiv},
  primaryClass  = {math.NA},
  note          = {arXiv:2605.12082}
}

@article{allard2021,
  author  = {Carrizo Vergara, Ricardo and Allard, Denis and Desassis, Nicolas},
  title   = {A general framework for {SPDE}-based stationary random fields},
  journal = {Bernoulli},
  volume  = {28},
  number  = {1},
  pages   = {1--32},
  year    = {2022},
  doi     = {10.3150/20-bej1317}
}

@article{Alonso2021,
  author  = {Daniel Sanz-Alonso and Ruiyi Yang},
  title   = {The {SPDE} Approach to {Matérn} Fields: Graph Representations},
  journal = {Statist. Sci.},
  volume  = {37},
  number  = {4},
  pages   = {519--540},
  year    = {2022},
  doi     = {10.1214/21-sts838}
}

@article{Alonso2021fem,
  author  = {Sanz-Alonso, Daniel and Yang, Ruiyi},
  title   = {Finite element representations of {Gaussian} processes: Balancing numerical and statistical accuracy},
  journal = {SIAM/ASA J. Uncertain. Quantif.},
  volume  = {10},
  number  = {4},
  pages   = {1323--1349},
  year    = {2022},
  doi     = {10.1137/21M144788X}
}

@article{anderes2020isotropic,
  author  = {Anderes, Ethan and M{\o}ller, Jesper and Rasmussen, Jakob Gulddahl},
  title   = {Isotropic covariance functions on graphs and their edges},
  journal = {Ann. Statist.},
  volume  = {48},
  number  = {4},
  pages   = {2478--2503},
  year    = {2020},
  doi     = {10.1214/19-AOS1896}
}

@article{lindgren2024diffusion,
  author  = {Lindgren, Finn and Bakka, Haakon and Bolin, David and Krainski, Elias and Rue, H{\aa}vard},
  title   = {A diffusion‐based spatio‐temporal extension of {Gaussian} {Mat{\'e}rn} fields},
  journal = {Stat. Oper. Res. Trans.},
  volume  = {48},
  number  = {1},
  pages   = {3--66},
  year    = {2024},
  url     = {https://doi.org/10.57645/20.8080.02.13},
  doi     = {10.57645/20.8080.02.13}
}

@article{bakka2019,
  author  = {Bakka, Haakon and Vanhatalo, Jarno and Illian, Janine B.
  and Simpson, Daniel and Rue, H{\aa}vard},
  title   = {Non-stationary {G}aussian models with physical barriers},
  journal = {Spat. Stat.},
  volume  = {29},
  pages   = {268--288},
  year    = {2019},
  doi     = {10.1016/j.spasta.2019.01.002}
}

@article{bolin2023equivalence,
  author    = {Bolin, David and Kirchner, Kristin},
  title     = {Equivalence of measures and asymptotically optimal linear prediction for {Gaussian} random fields with fractional-order covariance operators},
  journal   = {Bernoulli},
  volume    = {29},
  number    = {2},
  pages     = {1476--1504},
  publisher = {Bernoulli Society for Mathematical Statistics and Probability},
  year      = {2023},
  doi       = {10.3150/22-bej1507}
}

@article{BK2020rational,
  author  = {Bolin, David and Kirchner, Kristin},
  title   = {The rational {SPDE} approach for {G}aussian random fields with general smoothness},
  journal = {J. Comput. Graph. Statist.},
  volume  = {29},
  number  = {2},
  pages   = {274--285},
  year    = {2020},
  doi     = {10.1080/10618600.2019.1665537}
}

@article{BKK2018,
  author  = {Bolin, David and Kirchner, Kristin and Kov\'{a}cs, Mih\'{a}ly},
  title   = {Weak convergence of {G}alerkin approximations for fractional
  elliptic stochastic {PDE}s with spatial white noise},
  journal = {BIT Numer. Math.},
  volume  = {58},
  number  = {4},
  pages   = {881--906},
  year    = {2018},
  doi     = {10.1007/s10543-018-0719-8}
}

@article{BKK2020,
  author  = {Bolin, David and Kirchner, Kristin and Kov\'{a}cs, Mih\'{a}ly},
  title   = {Numerical solution of fractional elliptic stochastic {PDE}s
  with spatial white noise},
  journal = {IMA J. Numer. Anal.},
  volume  = {40},
  number  = {2},
  pages   = {1051--1073},
  year    = {2020},
  doi     = {10.1093/imanum/dry091}
}

@article{bolin14,
  author   = {Bolin, David},
  title    = {Spatial {M}at\'{e}rn fields driven by non-{G}aussian noise},
  journal  = {Scand. J. Statist.},
  volume   = {41},
  number   = {3},
  pages    = {557--579},
  year     = {2014},
  url      = {https://doi.org/10.1111/sjos.12046},
  doi      = {10.1111/sjos.12046},
}

@article{bolin2024covariance,
  author    = {Bolin, David and Simas, Alexandre B. and Xiong, Zhen},
  title     = {Covariance--based rational approximations of fractional {SPDEs} for computationally efficient {Bayesian} inference},
  journal   = {J. Comput. Graph. Statist.},
  volume    = {33},
  number    = {1},
  pages     = {64--74},
  publisher = {Taylor \& Francis},
  year      = {2024},
  doi       = {10.1080/10618600.2023.2231051}
}

@article{bolin2024regularity,
  author  = {Bolin, David and Kov{\'a}cs, Mih{\'a}ly and Kumar, Vivek and Simas, Alexandre B.},
  title   = {Regularity and numerical approximation of fractional elliptic differential equations on compact metric graphs},
  journal = {Math. Comp.},
  volume  = {93},
  number  = {349},
  pages   = {2439--2472},
  year    = {2024},
  doi     = {10.1090/mcom/3929}
}

@article{bolin_comparison_2013,
  author  = {Bolin, David and Lindgren, Finn},
  title   = {A comparison between {Markov} approximations and other methods for large spatial data sets},
  journal = {Comput. Statist. Data Anal.},
  volume  = {61},
  pages   = {7--21},
  year    = {2013},
  doi     = {10.1016/j.csda.2012.11.011}
}

@article{bonito2024numerical,
  author    = {Bonito, Andrea and Guignard, Diane and Lei, Wenyu},
  title     = {Numerical approximation of {Gaussian} random fields on closed surfaces},
  journal   = {Comput. Methods Appl. Math.},
  volume    = {24},
  number    = {4},
  pages     = {829--858},
  publisher = {De Gruyter},
  year      = {2024},
  doi       = {10.1515/cmam-2022-0237}
}

@inproceedings{borovitskiy2020matern,
  author    = {Borovitskiy, Viacheslav and Terenin, Alexander and Mostowsky, Peter and Deisenroth, Marc Peter},
  title     = {{Mat{\'e}rn} {Gaussian} processes on {Riemannian} manifolds},
  booktitle = {{Advances in Neural Information Processing Systems}},
  volume    = {33},
  pages     = {12426--12437},
  year      = {2020}
}

@book{brenner2008mathematical,
  author    = {Brenner, Susanne C. and Scott, L. Ridgway},
  title     = {The Mathematical Theory of Finite Element Methods},
  series    = {Texts in Applied Mathematics},
  volume    = {15},
  edition   = {3},
  publisher = {Springer},
  address   = {New York},
  year      = {2008},
  doi       = {10.1007/978-0-387-75934-0}
}

@article{bw20,
  author   = {Bolin, David and Wallin, Jonas},
  title    = {Multivariate type {G} {M}at\'{e}rn stochastic partial differential
  equation random fields},
  journal  = {J. R. Stat. Soc. Ser. B Stat. Methodol.},
  volume   = {82},
  number   = {1},
  pages    = {215--239},
  year     = {2020},
  doi      = {10.1111/rssb.12351},
}

@article{cameletti_spatio-temporal_2013,
  author  = {Cameletti, Michela and Lindgren, Finn and Simpson, Daniel and Rue, Håvard},
  title   = {Spatio-temporal modeling of particulate matter concentration through the {SPDE} approach},
  journal = {AStA Adv. Stat. Anal.},
  volume  = {97},
  number  = {2},
  pages   = {109--131},
  year    = {2013},
  url     = {https://doi.org/10.1007/s10182-012-0196-3},
  doi     = {10.1007/s10182-012-0196-3},
}

@article{suuronen2022cauchy,
  author    = {Suuronen, Jarkko and Chada, Neil K. and Roininen, Lassi},
  title     = {{Cauchy} {Markov} random field priors for {Bayesian} inversion},
  journal   = {Stat. Comput.},
  volume    = {32},
  number    = {2},
  pages     = {33},
  publisher = {Springer},
  year      = {2022},
  doi       = {10.1007/s11222-022-10089-z}
}

@book{ciarlet2002finite,
  author    = {Ciarlet, Philippe G.},
  title     = {The finite element method for elliptic problems},
  publisher = {SIAM},
  year      = {2002},
  doi       = {10.1137/1.9780898719208}
}

@article{cox2020,
  author  = {Cox, Sonja G. and Kirchner, Kristin},
  title   = {Regularity and convergence analysis
  in {S}obolev and {H}\"older spaces
  for generalized {W}hittle--{M}atérn fields},
  journal = {Numer. Math.},
  volume  = {146},
  pages   = {819--873},
  year    = {2020},
  doi     = {10.1007/s00211-020-01151-x}
}

@article{doi:10.1137/21M1400717,
  author  = {Jansson, Erik and Kov\'{a}cs, Mih\'{a}ly and Lang, Annika},
  title   = {Surface Finite Element Approximation of Spherical {Whittle}--{Matérn} {Gaussian} Random Fields},
  journal = {SIAM J. Sci. Comput.},
  volume  = {44},
  number  = {2},
  pages   = {A825-A842},
  year    = {2022},
  url     = {https://doi.org/10.1137/21M1400717},
  doi     = {10.1137/21M1400717}
}

@article{dziuk2013finite,
  author    = {Dziuk, Gerhard and Elliott, Charles M.},
  title     = {Finite element methods for surface {PDEs}},
  journal   = {Acta Numer.},
  volume    = {22},
  pages     = {289--396},
  publisher = {Cambridge University Press},
  year      = {2013},
  doi       = {10.1017/s0962492913000056}
}

@incollection{fix1972effects,
  author    = {Fix, George J.},
  title     = {Effects of quadrature errors in finite element approximation of steady state, eigenvalue and parabolic problems},
  booktitle = {The mathematical foundations of the finite element method with applications to partial differential equations},
  pages     = {525--556},
  publisher = {Elsevier},
  year      = {1972},
  doi       = {10.1016/b978-0-12-068650-6.50024-1}
}

@book{grisvard2011elliptic,
  author    = {Grisvard, Pierre},
  title     = {Elliptic problems in nonsmooth domains},
  publisher = {SIAM},
  address   = {Philadelphia},
  year      = {2011},
  doi       = {10.1137/1.9781611972030}
}

@article{Harbrecht2021,
  author    = {Harbrecht, Helmut and Herrmann, Lukas and Kirchner, Kristin and Schwab, Christoph},
  title     = {Multilevel approximation of {Gaussian} random fields: Covariance compression, estimation, and spatial prediction},
  journal   = {Adv. Comput. Math.},
  volume    = {50},
  number    = {5},
  pages     = {101},
  publisher = {Springer},
  year      = {2024},
  doi       = {10.1007/s10444-024-10187-8}
}

@article{herrmann2020multilevel,
  author    = {Herrmann, Lukas and Kirchner, Kristin and Schwab, Christoph},
  title     = {Multilevel approximation of {Gaussian} random fields: fast simulation},
  journal   = {Math. Models Methods Appl. Sci.},
  volume    = {30},
  number    = {1},
  pages     = {181--223},
  publisher = {World Scientific},
  year      = {2020},
  doi       = {10.1142/s0218202520500050}
}

@article{Hildeman2020,
  author  = {Hildeman, Anders and Bolin, David and Rychlik, Igor},
  title   = {Deformed {SPDE} models with an application to spatial modeling of significant wave height},
  journal = {Spat. Stat.},
  volume  = {42},
  pages   = {100449},
  year    = {2021},
  doi     = {10.1016/j.spasta.2020.100449}
}

@article{jin2013error,
  author    = {Jin, Bangti and Lazarov, Raytcho and Zhou, Zhi},
  title     = {Error estimates for a semidiscrete finite element method for fractional order parabolic equations},
  journal   = {SIAM J. Numer. Anal.},
  volume    = {51},
  number    = {1},
  pages     = {445--466},
  publisher = {SIAM},
  year      = {2013},
  doi       = {10.1137/120873984}
}

@article{kb-kriging,
  author  = {Kirchner, Kristin and Bolin, David},
  title   = {Necessary and sufficient conditions
  for asymptotically optimal linear prediction
  of random fields on compact metric spaces},
  journal = {Ann. Statist.},
  volume  = {50},
  number  = {2},
  pages   = {1038--1065},
  year    = {2022},
  doi     = {10.1214/21-aos2138}
}

@misc{baratta2023dolfinx,
  author    = {Baratta, Igor A. and Dean, Joseph P. and Dokken, J{\o}rgen S. and
                    Habera, Michal and Hale, Jack S. and Richardson, Chris and
                    Rognes, Marie E. and Scroggs, Matthew W. and Sime, Nathan and
                    Wells, Garth N.},
  title     = {{DOLFINx}: The next generation {FEniCS} problem solving environment},
  publisher = {Zenodo},
  year      = {2023},
  doi       = {10.5281/zenodo.10447666},
  note      = {DOI: 10.5281/zenodo.10447666}
}

@article{geuzaine2009gmsh,
  author  = {Geuzaine, Christophe and Remacle, Jean-Fran{\c{c}}ois},
  title   = {{Gmsh}: A three-dimensional finite element mesh generator with built-in
               pre- and post-processing facilities},
  journal = {Internat. J. Numer. Methods Engrg.},
  volume  = {79},
  number  = {11},
  pages   = {1309--1331},
  year    = {2009},
  doi     = {10.1002/nme.2579}
}

@article{alnaes2014ufl,
  author  = {Aln{\ae}s, Martin S. and Logg, Anders and {\O}lgaard, Kristian B. and
               Rognes, Marie E. and Wells, Garth N.},
  title   = {Unified Form Language: A domain-specific language for weak formulations
               of partial differential equations},
  journal = {ACM Trans. Math. Software},
  volume  = {40},
  number  = {2},
  pages   = {9:1--9:37},
  year    = {2014},
  doi     = {10.1145/2566630}
}

@article{dalcin2011petsc4py,
  author  = {Dalcin, Lisandro D. and Paz, Rodrigo R. and Kler, Pablo A. and
               Cosimo, Alejandro},
  title   = {Parallel distributed computing using {Python}},
  journal = {Adv. Water Resour.},
  volume  = {34},
  number  = {9},
  pages   = {1124--1139},
  year    = {2011},
  doi     = {10.1016/j.advwatres.2011.04.013}
}

@article{lindgren11,
  author  = {Lindgren, Finn and Rue, H{\aa}vard and Lindstr\"{o}m, Johan},
  title   = {An explicit link between {G}aussian fields and {G}aussian
  {M}arkov random fields: the stochastic partial differential
  equation approach},
  journal = {J. R. Stat. Soc. Ser. B Stat. Methodol.},
  volume  = {73},
  number  = {4},
  pages   = {423--498},
  year    = {2011},
  doi     = {10.1111/j.1467-9868.2011.00777.x}
}

@article{lindgren2022spde,
  author    = {Lindgren, Finn and Bolin, David and Rue, H{\aa}vard},
  title     = {The {SPDE} approach for {Gaussian} and non-{Gaussian} fields: 10 years and still running},
  journal   = {Spat. Stat.},
  volume    = {50},
  pages     = {100599},
  publisher = {Elsevier},
  year      = {2022},
  doi       = {10.1016/j.spasta.2022.100599}
}

@book{matern60,
  author    = {Mat\'{e}rn, Bertil},
  title     = {Spatial variation: {S}tochastic models and their application
  to some problems in forest surveys and other sampling
  investigations},
  pages     = {144},
  publisher = {Meddelanden Fr{\aa}n Statens Skogsforskningsinstitut,
  Band 49, Nr.\ 5, Stockholm},
  year      = {1960},
  doi       = {10.18174/201735}
}

@book{thomee2007galerkin,
  author    = {Thom{\'e}e, Vidar},
  title     = {Galerkin finite element methods for parabolic problems},
  series    = {Springer Series in Computational Mathematics},
  volume    = {25},
  publisher = {Springer},
  year      = {2007},
  doi       = {10.1007/3-540-33122-0}
}

@article{Wallin15,
  author   = {Wallin, Jonas and Bolin, David},
  title    = {Geostatistical modelling using non-{G}aussian {M}at\'{e}rn fields},
  journal  = {Scand. J. Statist.},
  volume   = {42},
  number   = {3},
  pages    = {872--890},
  year     = {2015},
  url      = {https://doi.org/10.1111/sjos.12141},
  doi      = {10.1111/sjos.12141},
}

@article{whittle63,
  author  = {Whittle, P.},
  title   = {Stochastic processes in several dimensions},
  journal = {Bull. Inst. Internat. Statist.},
  volume  = {40},
  pages   = {974--994},
  year    = {1963}
}

@book{zeidler2012applied,
  author    = {Zeidler, Eberhard},
  title     = {Applied functional analysis: Applications to mathematical physics},
  series    = {Applied Mathematical Sciences},
  volume    = {108},
  publisher = {Springer},
  year      = {2012},
  doi       = {10.1007/978-1-4612-0821-1}
}

@book{Nualart_1995,
  author    = {Nualart, David},
  title     = {The {M}alliavin calculus and related topics},
  series    = {Probability and its Applications},
  publisher = {Springer},
  address   = {New York},
  year      = {1995},
  doi       = {10.1007/978-1-4757-2437-0},
  isbn      = {9781475724370},
}

@book{Lee2013Smooth,
  author    = {Lee, John M.},
  title     = {Introduction to Smooth Manifolds},
  series    = {Graduate Texts in Mathematics},
  volume    = {218},
  pages     = {xvi+708},
  edition   = {2},
  publisher = {Springer},
  address   = {New York, NY},
  year      = {2013},
  doi       = {10.1007/978-1-4419-9982-5}
}

@book{Cohn_2013,
  author    = {Cohn, Donald L.},
  title     = {Measure theory: Second edition},
  series    = {Birkhäuser Advanced Texts Basler Lehrbücher},
  publisher = {Springer},
  address   = {New York},
  year      = {2013},
  doi       = {10.1007/978-1-4614-6956-8},
  isbn      = {9781461469568}
}

@book{Simon_2010,
  author    = {Simon, Barry},
  title     = {Trace ideals and their applications},
  series    = {Mathematical Surveys and Monographs},
  publisher = {American Mathematical Society},
  year      = {2010},
  doi       = {10.1090/surv/120},
  url       = {https://doi.org/10.1090/surv/120},
  isbn      = {9781470413477}
}

@article{bolin2025log,
  author  = {Bolin, David and Saduakhas, Damilya and Simas, Alexandre B.},
  title   = {{Log}-{Gaussian} {Cox} process on general metric graphs},
  journal = {Biometrika},
  year    = {2026},
  doi     = {10.1093/biomet/asag043},
  url     = {https://doi.org/10.1093/biomet/asag043},
}

\end{document}